\documentclass{compositio}
\usepackage{amssymb,bm,amsmath, mathrsfs }

\usepackage[all]{xy}
\usepackage{color}
\usepackage{hyperref}

\usepackage[OT2,T1]{fontenc}
\DeclareSymbolFont{cyrletters}{OT2}{wncyr}{m}{n}
\DeclareMathSymbol{\Sha}{\mathalpha}{cyrletters}{"58}
\newcommand{\map}[1]{\xrightarrow{#1}}
\newcommand{\iso}{\xrightarrow{\sim}}
\newcommand{\define}{:=}
\newcommand{\red}{\mathrm{red}}

\newcommand{\Hom}{\mathrm{Hom}}
\newcommand{\Aut}{\mathrm{Aut}}
\newcommand{\End}{\mathrm{End}}
\newcommand{\Spec}{\mathrm{Spec}}
\newcommand{\Q}{\mathbb Q}
\newcommand{\Z}{\mathbb Z}
\newcommand{\R}{\mathbb R}
\newcommand{\C}{\mathbb C}
\newcommand{\F}{\mathbb F}
\newcommand{\A}{\mathbb A}
\newcommand{\co}{\mathcal O}
\newcommand{\ord}{\mathrm{ord}}

\newcommand{\RZfsm}{\mathrm{RZ}^{\mathrm{fsm}}}
\newcommand{\RZnilp}{\mathrm{RZ}^{\mathrm{nilp}}}
\newcommand{\RZc}{\mathrm{RZ}}

\newcommand{\GSpin}{\mathrm{GSpin}}
\newcommand{\SO}{\mathrm{SO}}
\newcommand{\GSp}{\mathrm{GSp}}
\newcommand{\OGr}{\mathrm{OGr}}
\newcommand{\Gm}{{\mathbb{G}_m}}

\newcommand{\GL}{\mathrm{GL}}
\newcommand{\Fil}{\mathrm{Fil}}
\newcommand{\action}{\bullet}
\newcommand{\FF}{\mathscr{F}}

\newcommand{\Db}{{\mathbb D}}
\newcommand{\Sh}{{\mathrm {Sh}}}
\newcommand{\sS}{{\mathscr S}}

\newcommand{\dia}{\diamond}

\newcommand{\ti}{\widetilde}
\newcommand{\Spf}{\mathrm{Spf}}
\newcommand{\bZp}{W}

\newcommand{\fX}{{\mathfrak X}}

\newcommand{\adl}{{\phi}}
\newcommand{\xo}{X_0}

\newcommand{\lps}{[\![}
\newcommand{\rps}{]\!]}

\newcommand{\quash}[1]{}  %%Anything in \quash is ignored

\newcommand{\et}{{\rm {et}}}

\makeatletter
\renewcommand\subsubsection{\leftskip 0pt\@startsection {subsubsection}{3}{\z@}%
                                   {-1ex}%
                                   {-0ex}%
                                   {\normalfont}}
\makeatother

\begin{document}
\author[B. Howard]{Benjamin Howard}
\thanks{B.H. is partially supported by NSF grants  DMS-1201480 and DMS-1501583}
\address{Dept.~of Mathematics\\ Boston College\\
Chestnut Hill\\ MA 02467-3806}
\author[G. Pappas]{Georgios Pappas}
 \thanks{G.P. is partially supported by NSF grants  DMS-1102208 and DMS13-60733}
\address{Dept.~of Mathematics\\ Michigan State University\\ E. Lansing\\ MI 48824-1027}

\title[Rapoport-Zink spaces for spinor groups]{Rapoport-Zink spaces for spinor groups}
\date{\today}

\begin{abstract}
After the work of Kisin, there is a good theory of canonical integral models of Shimura varieties of Hodge type at primes of good reduction.  
The first part of this paper develops a theory of Hodge type Rapoport-Zink formal schemes, which  uniformize certain formal completions of such
integral models.  In the second part, the general theory is applied to the special case of Shimura varieties associated with 
groups of spinor similitudes, and the reduced scheme underlying the Rapoport-Zink space is determined explicitly.
\end{abstract}

\maketitle

\tableofcontents

\theoremstyle{plain}
\newtheorem{BigThm}{Theorem}
\swapnumbers
\newtheorem{theorem}[subsubsection]{Theorem}
\newtheorem{proposition}[subsubsection]{Proposition}
\newtheorem{lemma}[subsubsection]{Lemma}
\newtheorem{corollary}[subsubsection]{Corollary}
\newtheorem*{claim*}{Claim}
\newtheorem{conjecture}[subsubsection]{Conjecture}

\theoremstyle{definition}
\newtheorem{definition}[subsubsection]{Definition}

\newtheorem{remark}[subsubsection]{Remark}

\numberwithin{equation}{subsubsection}
\renewcommand{\labelenumi}{(\roman{enumi})}
\renewcommand{\theBigThm}{\Alph{BigThm}}

\bigskip

\section{Introduction}
\label{intro:hodge}

This paper contributes to the theory of integral models of Shimura varieties,  and to the related  theory of Rapoport-Zink formal schemes.
We concentrate our attention on Shimura varieties of Hodge type with hyperspecial level subgroup. 
In this case, canonical smooth integral models of the Shimura varieties were constructed by Kisin (see also work of Vasiu). 
Using these models, we give a construction of certain Hodge type Rapoport-Zink formal schemes, 
and describe their field-valued points in terms of certain \emph{refined} affine 
Deligne-Lusztig sets.

A large portion of the paper concerns what is arguably the most interesting 
family of Shimura varieties that are of Hodge but not of PEL type:  those associated to the  spinor similitude groups of quadratic spaces over $\Q$ of signature $(d, 2)$.
For this family of Shimura varieties, we use our results on Rapoport-Zink spaces to explicitly describe  the  basic (supersingular)
locus in the reduction modulo $p$ of the canonical integral model.

In what follows, we describe our results in more detail.  First, we will discuss the construction 
of Rapoport-Zink formal schemes for general Hodge type Shimura varieties, and then we will explain the description of the supersingular locus in the reduction modulo $p$ of the  Shimura varieties for spinor similitude groups.

\subsection{Rapoport-Zink spaces for Hodge type Shimura varieties}
\label{intro:RZ}

Let $( G , \mathcal{H}  )$ be  a Hodge type Shimura datum with reflex field $E \subset \C$.
Fix a prime $p>2$ and a sufficiently small compact open subgroup 
\[
U = U_p U ^p \subset G(\A_f)
\]
with $U_p \subset G(\Q_p)$ hyperspecial.    This implies that $G$ extends to a reductive group scheme over $\Z_{(p)}$, denoted the 
same way, with  $U_p=G(\Z_p)$.
Denote by $\Sh_U(G , \mathcal{H})$ the corresponding Shimura variety; it is a smooth quasi-projective variety over $E$ with complex points
 \[
 \Sh_U(G , \mathcal{H})(\C)  = G(\Q) \backslash \mathcal{H} \times G(\A_f) / U.
 \]

\subsubsection{}

For  $(G,\mathcal{H})$ to be of Hodge type means  that there is an embedding of Shimura data
\begin{equation}\label{intro hodge}
(G,\mathcal{H}) \to ( \GSp_{2g} , \mathcal{H}_{2g} ),
\end{equation}
where  $\mathcal{H}_{2g}$ is the union of the upper and lower Siegel half-spaces of genus $g$.   This embedding
may be chosen in a particular way:   we can find a  self-dual symplectic space 
$( C ,\psi )$ over $\Z_{(p)}$   and a closed immersion 
\begin{equation}\label{intro symplectic}
G \hookrightarrow \GSp(C,\psi)
\end{equation}
of reductive groups over $\Z_{(p)}$ whose generic fiber induces (\ref{intro hodge}).

Moreover,   $G$ can be realized as the pointwise stabilizer of a finite set of tensors 
$
(s_\alpha) \subset C^\otimes.
$
Here $C^\otimes$ is the \emph{total tensor algebra}; it is defined as the direct sum of all free $\Z_{(p)}$-modules 
that can be formed from $C$  using the operations of taking duals, tensor products,
symmetric powers,  and exterior powers.   In particular, if we set
\begin{equation}\label{contra}
D= \Hom(C,\Z_{(p)})
\end{equation}
with its  contragredient action $(g d)(c) = d(g^{-1} c)$ of $G$,  
then $C^\otimes = D^\otimes$ as representations of $G$.

For a prime $v\mid p$ of $E$,  Kisin \cite{KisinJAMS} has proved that the   Shimura variety  $\Sh_U(G , \mathcal{H})$  over $E$
admits a canonical smooth integral model 
\[
\mathscr{S} = \mathscr{S}_U(G , \mathcal{H})
\] 
over the localization $\co_{E, (v) }$.  The integral model is constructed, using (\ref{intro hodge}),  
as the normalization of the Zariski closure of $\Sh_U(G , \mathcal{H})$ 
in the integral model of a Siegel moduli variety.   In particular,  $\mathscr{S}$  carries over it a  ``universal'' family of abelian varieties with 
additional structure, obtained as the pullback of the universal family over the Siegel variety.  The universal family on $\mathscr{S}$
depends on the choice of Hodge embedding (\ref{intro hodge}), but  the integral model $\mathscr{S}$ does not.

  The special fiber   of the canonical integral model  comes with its \emph{Newton stratification}, 
  whose strata are   defined by fixing the isogeny class of the universal $p$-divisible group with additional structure.    Among the Newton strata there  is a distinguished closed stratum, called the \emph{basic locus} (see \cite{RapoportRich, Wortmann}).
For many  Hodge type Shimura varieties,  the basic locus   is the  \emph{supersingular locus}: the locus of points at  which the 
universal  abelian variety is isogenous to a product of supersingular elliptic curves.  This will be the case for the 
spinor similitude Shimura varieties discussed below in \S\ref{intro:gspin}.

When $\Sh_U(G , \mathcal{H})$ is a  PEL-type Shimura variety, the completion of the integral model along the basic locus is described  
 via the  \emph{$p$-adic uniformization theorem}   of Rapoport-Zink  \cite{RapZinkBook}  as a quotient of  what is now called  a
 Rapoport-Zink formal scheme.

\subsubsection{} The first main result of this paper is the construction of Rapoport-Zink formal schemes for 
general Hodge type Shimura varieties as above.  Such a construction also appears in the recent preprints of 
Kim \cite{KimRZ, KimUnif}.
We have followed Kim in the  characterization of our formal schemes as moduli spaces of quasi-isogenies between $p$-divisible groups endowed with so-called \emph{crystalline Tate tensors},  however, our construction of these spaces is  more direct than Kim's, and uses the existence of the integral model $\mathscr{S}$.

 Let $k$ be an algebraic closure of the residue field of the place $v\mid p$  fixed above, and let $W=W(k)$ be the ring of Witt vectors of $k$.
 For us, a  Hodge type Rapoport-Zink formal scheme over $W$ is characterized in terms of  the  \emph{local Shimura-Hodge datum} 
 $( G_{\Z_p} ,b_{x_0},\mu_{x_0}, C_{\Z_p})$ attached to a point 
 $
 x_0\in \mathscr{S}(k).
 $
 The reductive group scheme $G_{\Z_p}$ over $\Z_p$ and the representation   
 \[ G_{\Z_p} \hookrightarrow \GL( C_{\Z_p})\] were described above,   
 and we must now explain the meaning of $b_{x_0}$ and $\mu_{x_0}$.

Denote by $X_0$  the $p$-divisible group of the 
fiber  of the universal abelian scheme at  $x_0$, and let    $\Db(X_0)$ be its 
 contravariant Grothendieck-Messing crystal.   The evaluation $\Db(X_0)(W)$  of the crystal on $W$
 is the Dieudonn\'e module of $X_0$.  Kisin shows that this comes equipped with a collection of  \emph{crystalline  tensors} 
\[
t_{\alpha, 0}\in \Db(X_0)(W)^\otimes,
\] 
which are Frobenius invariant in $\Db(X_0)(W)^\otimes[1/p]$.  Moreover, there is a $W$-module isomorphism
\begin{equation}\label{dieudonne trivialization}
D \otimes_{\Z_p}W\xrightarrow{\sim}\Db(X_0)(W)
\end{equation}
identifying   $s_\alpha \otimes 1$ with  $t_{\alpha, 0}$.   Under any such identification  the Frobenius 
operator  on $\Db(X_0)(W)$ induces an operator on $D\otimes_{\Z_p} W$ of  the form  
\[
F=b_{x_0}\circ  \sigma
\] 
for some $b_{x_0}\in G(K)$.    Here $\sigma\in \Aut(W)$ lifts the absolute  Frobenius on $k$, and $K=W[1/p]$ is the fraction field of $W$.

Kisin shows that the Hodge filtration on $\Db(X_0)(k)$   is split  by a  
$G_k$-valued cocharacter,  which is the reduction of a minuscule cocharacter 
 \[
 \mu_{x_0}: \Gm_W\to G_W
 \]
satisfying
$
b_{x_0}\in G(W) \mu_{x_0}^\sigma(p) G(W).
$

   The $G(W)$-conjugacy class
 of $\mu_{x_0}$ is independent of (\ref{dieudonne trivialization}) and it agrees with 
 the conjugacy class of the  \emph{inverse} of the Deligne cocharacter $\mu_h: \Gm_\C \to G_\C$ associated to the symmetric domain $\mathcal{H}$. 
More precisely, $\mu_{x_0}$ and $\mu_h^{-1}$ become conjugate after we fix an isomorphism $\C \iso \bar{K}$ whose restriction to $E \hookrightarrow \bar{K}$ induces the place $v$ chosen above.

Having fixed $x_0$ and  (\ref{dieudonne trivialization}), we abbreviate  $b=b_{x_0}$ and $\mu=\mu_{x_0}$.  
Define   an  algebraic group $J_b$ over $\Q_p$  with functor of points 
\[
J_b(R)=\{g\in G(R\otimes_{\Q_p}K) :  gb\sigma(g)^{-1}=b\}
\]
for any $\Q_p$-algebra $R$. The element $b$ is \emph{basic} if and only if   $J_b$ is an inner form of $G$.

\begin{BigThm}\label{bigthm:RZ space}
There exists a formal scheme $\RZc_G$ over ${\rm Spf}(W)$ that is formally smooth and locally formally of finite type,  admits a left action of $J_b(\Q_p)$,
and has the following properties:
\begin{enumerate}
\item  It  is a formal closed subscheme of the usual Rapoport-Zink formal scheme $\RZc(X_0)$ 
over ${\rm Spf}(W)$  representing pairs $(X, \rho)$ of a $p$-divisible group $X$
and a quasi-isogeny $\rho: X_0\dasharrow X$, as in \cite{RapZinkBook}.

\item There is a bijection 
\[
\RZc_G(k)\xrightarrow{\sim}
{X}_{G, b,  \mu^\sigma }(k) 
\]
where $X_{G,b,\mu^\sigma}(k)$ is the affine Deligne-Lusztig set 
\[
\big\{g\in G(K) : g^{-1}b\sigma(g)\in G(W) \mu^\sigma (p) G(W) \big\}/G(W ).
\]

\item  Assume in addition that  $b$ is   basic, or,  equivalently, that the point  $x_0 $ lies in the basic locus. 
Then there is an  isomorphism of formal schemes
\[
\Theta^b: I (\Q)\backslash  \RZc_G\times  G({\mathbb A}^p_f)/U^p  \iso  (\widehat \sS_W)_{/\sS_b}.
\]
Here $(\widehat \sS_W)_{/\sS_b}$ is the completion of the base change $\sS_W$ along the basic locus $\sS_b$ of the special fiber, and $I$ is a reductive group over $\Q$,  which is an inner form of $G$ admitting identifications
\[
I(\Q_\ell) = 
\begin{cases}
J_b(\Q_p) & \mbox{if }\ell=p \\
G(\Q_\ell) & \mbox{if }\ell\neq p,
\end{cases}
\]
and with $I(\R)$ compact modulo center.

 \end{enumerate} 
\end{BigThm}

In fact, $\RZc_G(k)$ can be identified with the set of isomorphism classes of  triples $(X, \rho, (t_\alpha))$ in which $X$ is a $p$-divisible
group over $k$, \[( t_\alpha) \subset  \Db(X)(W)^\otimes\] is a collection of Frobenius invariant tensors, and 
$\rho: X_0\dasharrow X$  is a quasi-isogeny identifying  $t_\alpha$ with  $t_{\alpha, 0}$.
Some additional technical properties are required; see Definition  \ref{defRZG}. We can give a similar moduli
description of the $R$-valued points of $\RZc_G$ for $R$   any formally smooth 
formally finitely generated $W$-algebra, but not for general $R$. This description
uniquely determines $\RZc_G$.

\begin{remark}\label{rem:nonemptiness}
As noted earlier, Theorem \ref{bigthm:RZ space} already appears in the recent preprints of Kim \cite{KimRZ, KimUnif}.  
It is only our \emph{construction} of the space $\RZc_G$ that is new.  It is essential for our construction (but not for Kim's)
that the local Shimura-Hodge datum  $(G_{\Z_p} ,b,\mu,C_{\Z_p})$ arises from a point $x_0 \in \mathscr{S}(k)$ on a global  Hodge type Shimura variety as above. Given results on the non-emptiness of Newton strata for Shimura varieties of Hodge type which have been recently announced by Kisin, Madapusi Pera, and Shin, one should be able to show that this 
happens most of the time; we would like to return to this question on another occasion.
\end{remark}

\subsubsection{}
We can also give a  concrete description of $\RZc_G(k')$ when $k'/k$ is any finitely generated field extension and the $p$-divisible group $X_0$ is formal. This involves the new notion of a {\sl refined} affine Deligne-Lusztig set,  which we  now explain. 
Let $W'$ be the Cohen ring of $k'$, let $K'=W'[1/p]$ be its fraction field, and 
suppose $\sigma: W'\to W'$ is an appropriate lift of Frobenius (see Proposition \ref{fgfield}). 
The following is then obtained by using Zink's theory of displays and windows.

\begin{BigThm}\label{bigthm:RADL}
There is a bijection 
\begin{equation*}
\RZc_G(k')\xrightarrow{\sim}
{X}_{G, b,  \mu^\sigma, \sigma}(k')  
\end{equation*}
where   the refined affine Deligne-Lusztig set ${X}_{G, b,  \mu^\sigma, \sigma}(k')$  is, by definition, the image of the natural map
\[
\big\{g\in G(K') : g^{-1}b\sigma(g)\in G(W') \mu^\sigma (p)  \big\}
\to 
G(K') /G(W') .
\]
\end{BigThm}

Our refined affine Deligne-Lusztig set is a subset of the \emph{naive} affine Deligne-Lusztig set
\[
\big\{g\in G(K') : g^{-1}b\sigma(g)\in G(W') \mu^\sigma (p) G(W') \big\}/G(W' ),
\]
and equality holds if $k'$ is perfect.
The above description of $\RZc_G(k')$ is entirely group-theoretical (\emph{i.e.} does not involve $p$-divisible groups), 
and is thus quite useful.

\begin{remark}
There is a simpler parallel theory when we consider only the points 
of Rapoport-Zink formal schemes with values in perfect ${\mathbb F}_p$-algebras.
Indeed, then one can even use the Witt vector affine Grassmannian as in \cite{ZhuWitt} and \cite{BhattScholzeWitt} to  obtain 
a more straightforward and general construction of (at least) the reduced locus of Rapoport-Zink schemes
but {\sl only up to perfection}. These constructions allow one to also consider non-minuscule coweights. However, 
this comes at the cost of passing to the non-finite type perfection which loses a lot of information.

In contrast, in this paper  we can consider points with values in non-perfect rings,
but we must restrict to minuscule coweights connected to Shimura varieties. 
Allowing non-perfect rings as in  Theorem  \ref{bigthm:RADL} is  essential 
for the application to spinor Shimura varieties described below. A different approach towards 
describing  Rapoport-Zink formal schemes as functors on more general (not necessarily perfect) rings 
directly from the group data 
is pursued in work in preparation of one of us (G.P.) with O.~B\"ultel.
\end{remark}

\begin{remark}
 A direct construction of an  adic analytic space corresponding to the limit 
 of  Rapoport-Zink spaces over all $p$-level subgroups 
has been given by  Scholze and Weinstein \cite{ScholzeWeinsteinModuli} using Scholze's perfectoid spaces. 
Recently, there has been further progress in defining related spaces by  Scholze using his theory of  diamonds. 
The constructions  of Scholze and Scholze-Weinstein concern the generic fiber,  and do not provide a uniformization of the  integral model.
\end{remark}

\subsection{Spinor similitude Shimura varieties}
\label{intro:gspin}
In large part, our motivation for studying Rapoport-Zink spaces for Hodge type Shimura varieties 
is  to  apply the general theory to  the  Shimura varieties associated with spinor similitude groups.

 By combining our general results, specialized to the case of GSpin, with 
the linear algebra of lattices in quadratic spaces as in \cite{HP},  we obtain a very explicit description
of the basic locus of the special fiber of the integral model, and of the underlying reduced scheme of the corresponding 
 Rapoport-Zink formal scheme.

\subsubsection{}

 Start with an odd prime $p$ and a self-dual quadratic space $(V,Q)$ over $\Z_{(p)}$ of signature $(d,2)$ with $d\ge 1$.
 The corresponding bilinear form is denoted
 \begin{equation}\label{bilinear}
 [x,y] = Q(x+y) -Q(x) -Q(y).
 \end{equation}
  This determines a reductive group scheme 
  $
  G=\GSpin(V)
  $ 
  over $\Z_{(p)}$.   By slight abuse of notation,  we sometimes use the same letter to denote the generic fiber of $G$.

Define   a hyperspecial subgroup \[U_p = G(\Z_p) \subset G(\Q_p).\]  By  setting $U=U^pU_p$  for any sufficiently small compact open subgroup $U^p \subset G(\A_f^p)$, we obtain a $d$-dimensional
 Shimura variety   $\Sh_U(G , \mathcal{H})$ over $\Q$.   Here  $G(\R)$ acts on the hermitian domain
 \begin{equation}\label{hermitian domain}
 \mathcal{H} = \{ z\in V_\C : [z,z]=0,\,  [z,\bar{z}]<0 \} /\C^\times
 \end{equation}
 via the natural surjection  $G\to \SO(V)$.

 \subsubsection{}
The group $G$ is, by definition, a subgroup of the unit group of the Clifford algebra  $C=C(V)$, and hence
$G$ acts on $C$ by left multiplication.  For an appropriate choice of perfect symplectic form $\psi$ on $C$,  this defines a closed immersion (\ref{intro symplectic}) of reductive groups over $\Z_{(p)}$,

Thus we find ourselves in exactly the situation described in \S \ref{intro:RZ}.  Let $\mathscr{S}=\mathscr{S}_U(G , \mathcal{H})$
be the canonical smooth integral model over $\Z_{(p)}$,  equipped with the universal abelian scheme determined by the 
symplectic embedding (\ref{intro symplectic}).  This universal abelian scheme is also known as the \emph{Kuga-Satake abelian scheme},
and the locus of points
\begin{equation}\label{intro:ss locus}
\mathscr{S}_{ss} \subset \mathscr{S}  \otimes_{\Z_{(p)}} k
\end{equation}
at which it is supersingular is precisely the basic locus.
As before we set   $k=\bar{\F}_p$ and $W=W(k)$.

 Fix a supersingular point   $x_0 \in \mathscr{S}(k)$, and 
let $(G_{\Z_p} ,b,\mu,C_{\Z_p} )$ be the corresponding unramified local 
 Shimura-Hodge datum  as in \S  \ref{intro:RZ}.   
Let  $\RZc=\RZc_G$ be the associated formal scheme over $W$, as in  Theorem \ref{bigthm:RZ space}.
Our  main result is an explicit description of the underlying reduced locally finite type $k$-scheme $\RZc^\mathrm{red}$.   First, we give a formula for its dimension.

\begin{BigThm}
Let $n=d+2$ be the dimension of $V_{\Q_p}$.
All irreducible components  of $\RZc^\mathrm{red}$ are isomorphic, and are  smooth of dimension 
\[
\dim( \RZc^\mathrm{red} ) = \frac{1}{2} \begin{cases}
n-4 & \mbox{if $n$ is even and $\det(V_{\Q_p}) = (-1)^{\frac{n}{2}} $} \\
n-3 & \mbox{if $n$ is odd} \\
n-2 & \mbox{if $n$ is even and $\det(V_{\Q_p}) \neq (-1)^{\frac{n}{2}} $,}
\end{cases}
\]
where the equalities involving $\det(V_{\Q_p})$ are understood to be in $\Q_p^\times$ modulo squares.
Equivalently,
\[
\dim( \RZc^\mathrm{red} )  =  \begin{cases}
(d/2)  -1 & \mbox{if $V_{\Q_p}$ is a sum of hyperbolic planes } \\ 
\left\lfloor d/2\right\rfloor   & \mbox{otherwise.}
\end{cases}
\]
\end{BigThm}

\subsubsection{}
In fact, we give essentially a complete description of $\RZc^\mathrm{red}$, in the same spirit as the work of Vollaard \cite{VollaardSS}, Vollaard-Wedhorn \cite{VollaardWedhorn}, Rapoport-Terstiege-Wilson \cite{RTW}, and the authors \cite{HP} for some unitary Shimura varieties.    
To explain its structure requires some more notation.

Consider the quadratic space $V_K$ over $K=W[1/p]$, with its natural action $G_K \to \SO(V_K)$.
The operator $\Phi = b \circ \sigma$ makes $V_K$ into a slope $0$ isocrystal, and its  subspace
of  $\Phi$-invariant vectors $V_K^\Phi$ is a $\Q_p$-quadratic space of the same dimension and  determinant as $V_{\Q_p}$, 
 but with different Hasse invariant.  In fact, the self duality of $V$ implies that $V_{\Q_p}$ has Hasse invariant $1$, and so  $V_K^\Phi$ has Hasse invariant $-1$.

 A \emph{vertex lattice}  is a $\Z_p$-lattice $\Lambda \subset V_K^\Phi$  satisfying 
 $p\Lambda \subset \Lambda^\vee \subset \Lambda$. The quadratic form $pQ$ on $V_K^\Phi$ induces 
 a quadratic form on the $\F_p$-vector space
 \[
 \Omega_0=\Lambda /\Lambda^\vee.
 \] 
 The \emph{type}    $t_\Lambda = \dim( \Omega_0)$ of $\Lambda$ is   is even, and satisfies
$
2 \le t_\Lambda \le t_\mathrm{max},
$  
where
\begin{equation}\label{tmax}
t_\mathrm{max} =   \begin{cases}
n-2 & \mbox{if $n$ is even and $\det(V_{\Q_p}) = (-1)^{\frac{n}{2}} $} \\
n-1 & \mbox{if $n$ is odd} \\
n & \mbox{if $n$ is even and $\det(V_{\Q_p}) \neq (-1)^{\frac{n}{2}} $}.
\end{cases}
\end{equation}
One may characterize  $\Omega_0$ as the unique quadratic space over $\F_p$ of dimension 
$t_\Lambda$ that admits no Lagrangian (= totally isotropic of dimension $t_\Lambda/2$) subspace.

Of course the base change of $\Omega_0$ to $k$ does admit Lagrangian subspaces, and we 
define a smooth projective $k$-variety $S_\Lambda$ with $k$-points
\[
S_\Lambda(k) = \left\{ \mathrm{Lagrangians\ }\mathscr{L} \subset \Omega_0\otimes_{\F_p} k : \dim( \mathscr{L} + \Phi(\mathscr{L}) ) = \frac{t_\Lambda}{2} +1 \right\}.
\]
Here $\Phi = \mathrm{id} \otimes \sigma$ is the  operator on $\Omega_0\otimes k$ induced by the absolute Frobenius
$\sigma(x)=x^p$ on $k$.  The variety $S_\Lambda = S_\Lambda^+ \sqcup S_\Lambda^-$ has two connected components,
which are (non-canonically) isomorphic, and smooth of dimension $(t_\Lambda/2)-1$.  
As we will explain in \S\ref{DLsection}, these can be identified with closures of Deligne-Lusztig varieties for ${\rm SO}(\Omega_0)$.

\begin{BigThm}
The Rapoport-Zink formal scheme $\RZc=\RZc_G$ admits a decomposition 
\[
\RZc = \bigsqcup_{\ell \in \Z} \RZc^{(\ell)}
\]
with the following properties:
\begin{enumerate}
\item
Each open and closed formal subscheme  $\RZc^{(\ell)}$ is connected, and 
\[
\RZc^{(\ell)} \iso \RZc^{(\ell +1)}.
\]
\item
Each connected component  $\RZc^{(\ell)}$  has a collection of   closed  formal subschemes 
$
\RZc_\Lambda^{(\ell) } \subset \RZc^{(\ell) }
$
indexed by the vertex lattices $\Lambda \subset V_K^\Phi$, and the underlying reduced schemes  satisfy
\[\RZc_\Lambda^{(\ell) ,\mathrm{red} } \iso S_\Lambda^\pm.\]  Moreover, for any vertex lattices $\Lambda_1$ and $\Lambda_2$,
\[
\RZc_{\Lambda_1} (k) \cap \RZc_{\Lambda_2}(k)
= \begin{cases}
\RZc_{\Lambda_1 \cap \Lambda_2 }(k) & \mbox{if $\Lambda_1\cap\Lambda_2$ is a vertex lattice} \\
\emptyset & \mbox{otherwise.}
\end{cases}
\]
\item
The irreducible components of $\RZc^{(\ell) ,\mathrm{red}}$ are precisely the  closed subschemes $\RZc_\Lambda^{(\ell) ,\mathrm{red} }$
indexed by the vertex lattices of type $t_\Lambda =t_\mathrm{max}$.
\end{enumerate}
\end{BigThm}

 Loosely speaking, the theorem asserts that the irreducible components of $\RZc^{\mathrm{red} }$, their intersections, the intersections of their intersections, \emph{etc.}~are all isomorphic to varieties of the form $S_\Lambda^\pm$ for various choices of $\Lambda$.
 The following result  is an immediate corollary of this and the uniformization result of Theorem \ref{bigthm:RZ space}.

 \begin{BigThm}
For $U^p \subset G(\A_f^p)$ sufficiently small,  every irreducible component of the supersingular locus (\ref{intro:ss locus})
is isomorphic to  a connected component of the smooth projective $k$-variety
\[
\left\{ \mathrm{Lagrangians\ }\mathscr{L} \subset \Omega_0\otimes k : \dim( \mathscr{L} + \Phi(\mathscr{L}) ) =  \frac{t_\mathrm{max}}{2} +1 \right\},
\] 
where $\Omega_0$ is the unique quadratic space over $\F_p$ having dimension $t_\mathrm{max}$, and admitting no Lagrangian subspace.
In particular, all irreducible components of $\mathscr{S}_{ss}$ are smooth and projective of dimension 
\[
\dim( \mathscr{S}_{ss} )  =  \frac{t_\mathrm{max}}{2}  -1   =  \begin{cases}
(d/2)  -1 & \mbox{if $V_{\Q_p}$ is a sum of hyperbolic planes } \\ 
\left\lfloor d/2\right\rfloor   & \mbox{otherwise.}
\end{cases}
\]

 \end{BigThm}

 \subsection{Applications and directions of further inquiry}

\subsubsection{}
One motivation for wanting such an explicit description of the supersingular locus   for $\GSpin$ Shimura varieties is 
because of its relevance to conjectures of Kudla \cite{KudlaSurvey} relating intersections of special cycles on orthogonal  Shimura varieties to derivatives of Eisenstein series.  Indeed, Kudla and Rapoport \cite{KR1,KR2} were able to verify many cases of these conjectures for Shimura varieties attached to the low rank groups $\GSpin(2,2)$ and $\GSpin(3,2)$, and their arguments depend in an essential way on having concrete descriptions of  the supersinglar loci.  

With the results of \S \ref{intro:gspin} now in hand, it should be possible to extend the results of [\emph{loc.~cit.}] to all Shimura varieties of type $\GSpin(d,2)$.  Some results in this direction will appear in the forthcoming Boston College Ph.D.~thesis of Cihan Soylu.

 \subsubsection{}
 G\"ortz and He \cite{GortzHe} have studied all
basic minuscule affine Deligne-Lusztig varieties  for equicharacteristic discrete valued fields.  
They give a list of cases where these affine Deligne-Lusztig varieties  can be expressed as a union 
of classical Deligne-Lusztig varieties, and that list contains   equicharacteristic analogues of the  $\GSpin$
Rapoport-Zink spaces considered here. In fact, these spaces are the only (absolutely simple) types in their list 
with hyperspecial level subgroups which are not of EL or PEL type.  The  results  of G\"ortz and He  in the equicharacteristic case
are analogous to our  mixed characteristic results.

 There are other Hodge type cases 
for which a similar description should be possible,  but for more general 
parahoric level subgroups. Extending our construction of Rapoport-Zink formal schemes to the general parahoric case, by using, for example, 
 the integral models of Shimura varieties given in \cite{K-P}, is an interesting problem.
 If this is done, 
then our results should extend to cover all the cases listed in \cite{GortzHe}. 
This will probably require generalizing, via Bruhat-Tits theory, the algebra of 
 lattices in quadratic spaces we use in this paper. In
another direction, it would also be interesting to understand our results from the point of view of the stratifications 
introduced by  Chen and Viehmann in \cite{ViehmannChen}. 

 \subsubsection{}

In the cases considered in \cite{GortzHe}, the affine Deligne-Lusztig varieties are unions of Ekedahl-Oort (EO) strata. Such strata can be defined in the hyperspecial mixed characteristic case  following \cite{ZhangCaoEO} or \cite{ViehmannAnnals}. In the 
$\GSpin$ case considered here, the EO strata should be indexed by the possible
types $t_\Lambda$ of vertex lattices $\Lambda$. In fact, we expect that each EO stratum is the union of all \emph{Bruhat-Tits strata} 
\[
\mathrm{BT}_\Lambda = \RZc_\Lambda^\red \smallsetminus \bigcup_{ \Lambda'\subsetneq \Lambda } \RZc_{\Lambda'}^\red
\] 
in the sense of  \S \ref{ss:BT}, with $\Lambda$ ranging over all vertex lattices of 
the corresponding type.

\smallskip

  \subsection{Organization and contents}
In \S \ref{s:General RZ space} we first fix notations and  recall some general facts about windows and crystals  
for $p$-divisible groups, and  about local Shimura data.
When the local Shimura datum $(G,[b],\{\mu\})$ is of Hodge type, 
and after fixing  a suitable Hodge embedding, we define in \S \ref{ss:hodge RZ space} a functor $\RZnilp_G$ on $p$-nilpotent algebras. We also consider a functor $\RZfsm_G$ 
 defined (only) on formally smooth formally of finite type $p$-adic algebras,  which is essentially given by a limit of 
 values of $\RZnilp_G$. In \S\ref{par:ADL}, we describe the field valued points of these functors via refined affine Deligne-Lusztig sets. 
 
 In  \S \ref{ss:integral models}  we switch to the global set-up of Shimura varieties and recall some properties of the canonical integral models constructed by Kisin. Then, in \S \ref{newproofpar} we prove the first main result of the paper (Theorem \ref{mainKim}): Roughly speaking, we show that when the local Shimura datum is obtained from a global one,
the functor $\RZfsm_G$ is representable by a formal scheme $\RZc_G$. 
In  \S\ref{uniformpar} we prove a uniformization theorem for the formal completion of the integral model of the Shimura variety along its basic locus.
 
 The rest of the paper is devoted to Rapoport-Zink formal schemes and Shimura varieties for  spinor similitude groups.
 
In \S \ref{s:GSpin RZ space} we describe the corresponding local Shimura data and  define the GSpin Rapoport-Zink formal schemes. 
We devote \S \ref{s:lattices} to the algebra of certain type of lattices (``vertex lattices'' and ``special lattices'') in quadratic spaces. 
 This, together with our previous general results,  is used in \S\ref{s:RZ structure} to describe the reduced scheme underlying the basic GSpin Rapoport-Zink formal schemes (see especially \S \ref{ss:mainRZGSPin}).    Finally, in \S\ref{s:global GSpin}, we apply our local results to the global problem of describing the supersingular loci of Shimura varieties of type $\GSpin$.

 \smallskip
 
\subsection{Acknowledgements} We would like to thank W. Kim, M.~Kisin,  and M.~Rapoport for helpful discussions and comments, and the referees for some useful suggestions.

\smallskip
\subsection{Notation and conventions}

Throughout the paper,  $k=\bar{\F}_p$, where $p>2$.  The absolute Frobenius on $k$ is  denoted $\sigma(x)=x^p$.
We also denote  by $\sigma$ the induced automorphism of the ring of Witt vectors $W=W(k)$ and its fraction field $K=W[1/p]$.

%\bigskip
%\vfill\eject
%%%%%%%%%%%%%%%%%%%%%%%%%%%%%%%%%%%%

\section{Rapoport-Zink spaces of Hodge type}
\label{s:General RZ space}

%%%%%%%%%%%%%%%%%%%%%%%%%%%%%%%%%%%%

%%%%%%%%%%%%

\subsection{Preliminaries}
\label{ss:2prelim}
 
In this section we introduce  notation for various categories of $W$-algebras.  We also recall some facts about divided power thickenings and  crystals of $p$-divisible groups, and Zink's theory of windows.

\subsubsection{}  
As in \cite{RapZinkBook}, we will denote by ${\rm Nilp}_W$ the category of $W$-schemes $S$ such that $p$ is Zariski locally nilpotent in $\co_S$. 
Denote by 
\[
{\rm ANilp}_W \subset {\rm Nilp}_W^{ op}
\] 
the full  subcategory of Noetherian $W$-algebras  in which  $p$ is  nilpotent.
We denote by ${\rm ANilp}^{\rm f}_W$  the  category    of Noetherian adic $W$-algebras in which  $p$ is nilpotent, and  embed
 \[
 {\rm ANilp}_W \subset{\rm ANilp}^{\rm f}_W
 \] 
 as a full subcategory  by endowing any  $W$-algebra  in ${\rm ANilp}_W$ with its  $p$-adic topology.

 We say that an adic $W$-algebra $A$ is \emph{formally finitely generated}
if $A$ is Noetherian,  and if $A/I$ is a  finitely generated $W$-algebra for some ideal of definition $I\subset A$. 
Thus ${\rm Spf}(A)$ is a formal scheme which is formally of finite type over ${\rm Spf}(W)$. 
If, in addition, $p$ is nilpotent in $A$, then $A$ is a quotient of $W/(p^n)\lps x_1,\ldots, x_r\rps [y_1,\ldots, y_s]$
for some $n$, $r$, and $s$.

We will denote by 
\[
{\rm ANilp}^{\rm fsm}_W \subset {\rm ANilp}^{\rm f}_W
\] 
the   full subcategory  whose  objects are $W$-algebras that  are formally finitely generated and formally smooth over $W/(p^n)$,
for some $n\geq 1$.

\subsubsection{} 

As in   \cite[Chapter 2.1] {RapZinkBook}, every  formal scheme $\fX$ over $\Spf(W)$ defines a functor on ${\rm Nilp}_W$.
We restrict this functor to ${\rm ANilp}_W$, and then extend to ${\rm ANilp}^{\rm f}_W$ as follows:
 For $A$ in ${\rm ANilp}^{\rm f}_W$ with ideal of definition $I$,  define $\fX(A)$ to be the set 
\[
\fX(A)={\rm Hom}_{\Spf(W)}(\Spf(A), \fX)=\varprojlim\nolimits_n  \fX(A/I^n) .
\]

\subsubsection{}\label{PDlift} 

If $R$ is an object of  ${\rm ANilp}_W^{\rm fsm}$, the quotient  $\bar R=R/pR$ satisfies the condition \cite[(1.3.1.1)]{deJongBTIHES}. 
Thus,  by \cite[Lemma 1.3.3]{deJongBTIHES},   $\bar R$ admits a PD thickening
\[
\widetilde R\to \widetilde R/p\widetilde R =  \bar R
\] 
 by a formally smooth, $p$-adically complete $W$-algebra $\widetilde R$,  unique up to non-canonical isomorphism.
 The absolute Frobenius on $\bar R$  lifts to $\widetilde R$. 
The formal smoothness of $R$ over  $W/p^nW$ implies that  
$R\iso \widetilde R/p^n\widetilde R$,   and hence $\widetilde R$ also provides a    PD thickening $\widetilde R\to R$. 
 One can show  that $\widetilde R$ is a quotient   of a $W$-algebra of the form 
 \[
 W\lps x_1,\ldots, x_r\rps\{y_1,\ldots, y_s\}=\varprojlim_{n} W/(p^n)\lps x_1,\ldots, x_r\rps [y_1,\ldots, y_s],
 \]
 and hence   is Noetherian.

\subsubsection{}\label{sss:cohen}  

Suppose that $R$ is any $k$-algebra admitting  a $p$-basis in the sense of    \cite[\S1.1]{BerthelotMessingIII}. 
An explicit construction of a PD thickening $\widetilde R \to R$ is then explained in [\emph{loc.~cit.}].
This applies in particular when  $R=k'$ is   any field extension of $k$, in which case $\widetilde R$ is isomorphic to the Cohen ring $W'$ of $k'$.

Recall that the Cohen ring $W'$ is the unique, up to non-canonical isomorphism, discrete valuation
ring with $k'$ as a residue field and $p$ as uniformizer. It is flat over the Witt ring  $W$ of $k$.
If $(x_i)_i$ is a $p$-basis of $k'$, then a choice of elements $y_i\in W'$ with 
$
x_i^p \equiv y_i \pmod{ pW' }
$
determines a lift $\sigma: W'\to W'$ of the absolute Frobenius $k'\to k'$. 
Set $K'=W'[1/p]$.

\subsubsection{}\label{Zinkpar}
Continue with the above notation,  and fix a lift $\sigma: W'\to W'$ of the Frobenius of the field $k'$. 
The triple $(W',pW',k')$ gives a    frame for $k'$ in the sense  of \cite{ZinkWindows}. 
As in   [\emph{loc.~cit.}, Definition~2], a  \emph{Dieudonn\'e $W'$-window}  over $k'$ consists of a triple $(M, M_1, F),$ in which
\begin{itemize} 
\item $M$ is a free finitely generated $W'$-module,
\item $M_1\subset M$ is a  $W'$-submodule such that 
$ 
pM\subset M_1\subset M,
$ 
\item $F: M\to M$ is a $\sigma$-semi-linear map such that $F(M_1) \subset pM$, 
and $p^{-1}F(M_1)$ generates $M$ as an $W'$-module.
\end{itemize}
These conditions  imply  $F(M)\subset p^{-1}F(M_1)$, and so are equivalent to the conditions appearing in [\emph{loc.~cit.}, Definition~2].
If $(M, M_1, F)$ is a Dieudonn\'e $W'$-window  then 
\[
M_1=F^{-1}(pM)=\{x\in M [1/p] :  F(x)\in pM\}.
\]

If  $M$ is a free finitely generated $W'$-module,  and $F: M[1/p] \to M[1/p]$  is a $\sigma$-semi-linear map such that
\begin{itemize}
\item $pM\subset F^{-1}(pM)\subset M$, and
\item $F(F^{-1}(M))$ generates $M$ as a $W'$-module,
\end{itemize}
then  $F(M)\subset M$ and $(M, F^{-1}(pM), F)$ is a Dieudonn\'e $W'$-window.

A Dieudonn\'e $W'$-window is called simply a  \emph{$W'$-window} when 
the additional nilpotence condition of  [\emph{loc.~cit.}, Definition 3] is satisfied.

\subsubsection{}
\label{crystalstuff} Let $S$ be a scheme such that  $p$ is Zariski locally nilpotent in $\co_S$.  Set 
\[
\bar S = S\otimes_{\Z_p}\F_p,
\]
and denote by $\sigma: \bar S\to\bar S$ the absolute Frobenius morphism. 

For a $p$-divisible group $X$ over $S$, we will denote by $\Db(X)$ its contravariant Dieudonn\'e crystal. It is a crystal of
locally free $\co_S/\Z_p$-modules of rank equal to the height $h(X)$ of $X$. We refer the reader to \cite{MessingBT}, \cite{BerthelotBreenMessingII}, and \cite{deJongBTIHES}  for
background on the construction and properties of the Dieudonn\'e crystal. 

The crystal
$\Db(X)$ is  equipped with the Hodge filtration 
\[
{\rm Fil}^1(X) = {\rm Lie}(X)^*\subset \Db(X)_{S},
\]
where $\Db(X)_S$ is the pull-back of $\Db(X)$ to the Zariski site of $S$;  it  is a locally free $\co_S$-module of rank $h(X)$, and  
the $\co_S$-submodule  ${\rm Fil}^1(X)$ is   locally a direct summand. 
We also have  the Frobenius morphism 
\begin{equation}\label{crystalline frob}
F : \sigma^*\Db(X) \to \Db(X),
\end{equation}
  where the pull-back $\sigma^*\Db(X) $ is defined as in [\emph{loc.~cit.}].

Define crystals
\[
{\bf 1}=\Db(\Q_p/\Z_p), \quad {\bf 1}(-1)=\Db(\mu_{p^\infty}),
\] 
and note that ${\bf 1}$ is the structure sheaf  $\co_S/\Z_p$ with the usual Frobenius structure and ${\rm Fil}^1=(0)$. 
We often confuse a global section $t$ of a crystal $\mathbb{D}$ with the corresponding morphism of crystals
$t: \bf{1} \to \mathbb{D}$.

We define $\Db(X)^*$ to be the  $\co_S/\Z_p$-linear
dual with the dual filtration.  Note that $\Db(X^\vee)^*=\Db(X)(-1)$, where $X^\vee$ is the dual $p$-divisible group.  
There is a  Frobenius structure on $\Db(X)^*$ as in (\ref{crystalline frob}) but it is defined defined only
``up to isogeny", \emph{i.e.~}only after we view $\Db(X)^*$ as an isocrystal  as below.

We define the category of isocrystals over $S$ as follows:
\begin{itemize}
\item  Objects are crystals $\Db$ of locally free $\co_S/\Z_p$-modules.  We write  $\Db[1/p]$ if we view $\Db$ as an isocrystal.

\item Morphisms $\Db[1/p]\to \Db'[1/p]$ are given by global sections of the Zariski sheaf 
$\underline {\rm Hom}(\Db,\Db')[1/p]$ over $S$, where ${\rm Hom}(\Db, \Db')$ is  taken in the category of crystals of locally free 
$\co_S/\Z_p$-modules.

\end{itemize}

Every quasi-isogeny $\rho: X\dashrightarrow X'$ of $p$-divisible groups over $S$, in the sense of \cite[Def. 2.8]{RapZinkBook}, induces an
isomorphism of isocrystals 
\begin{equation}\label{crystal isogeny}
\Db(\rho): \Db(X')[1/p]\iso  \Db(X)[1/p].
\end{equation}

 The \emph{total  tensor algebra} $\Db(X)^\otimes $ is defined as the direct sum of all the  crystals of locally free $\co_S/\Z_p$-modules 
which can be formed from $\Db(X)$  using the operations of taking duals, tensor products,
symmetric powers and exterior powers.  It is a crystal   of locally free $\co_S/\Z_p$-modules over $S$.
The Hodge filtration on $\Db(X)_S$  induces a natural filtration ${\rm Fil}^\bullet(\Db(X)^\otimes_S)$
on $\Db(X)^\otimes _S$, and the Frobenius morphism (\ref{crystalline frob}) 
induces an isomorphism of isocrystals 
\[
F : \sigma^*\Db(X)^\otimes[1/p] \iso  \Db(X)^\otimes[1/p ].
\]
 For any quasi-isogeny $\rho: X\dashrightarrow X'$ of $p$-divisible groups over $S$, the isomorphism (\ref{crystal isogeny}) extends to
\[
\Db(\rho) : \Db(X')^\otimes[1/p] \iso \Db(X)^\otimes[1/p].
\]

A similar discussion applies  to formal schemes $S$ over $\Spf(\Z_p)$, as in   \cite[Ch. 2]{deJongBTIHES}.

\subsubsection{} 
Suppose $X$ is a formal $p$-divisible  group over a field $k'$ of characteristic $p$.  Again letting $W'$ be the Cohen ring of $k'$,  
the  evaluation  
\[
 \Db(X)(W')=\varprojlim\nolimits_n \Db(X)(W'/p^nW')
 \]
  of the crystal $\Db(X)$ on   $W'$  has a natural structure of a $W'$-window over $k'$. 
  (Combine the proof of \cite[Theorem 1.6]{ZinkWindows}   with \cite[Theorem 6]{Zinkdisplay}.)

By  \cite[Thm.~4]{ZinkWindows}, the functor $X \mapsto \Db(X)(W')$ gives an anti-equivalence of categories between formal $p$-divisible groups over $k'$  and  $W'$-windows over $k'$.
 More precisely, the equivalence of [{\it loc.~cit.}] uses the covariant Dieudonn\'e crystal, and we compose the functor defined there with Cartier duality.

If $k'$ is perfect,  classical Dieudonn\'e theory   (or \cite[Thm.~3.2]{ZinkWindows})  gives in  the same way an anti-equivalence 
between (all) $p$-divisible groups over $k'$ and Dieudonn\'e modules over $W'=W(k')$.

\subsection{Local Shimura data} 
\label{ss:local shimura data}

For the remainder of \S \ref{s:General RZ space}, $G$ is   a connected 
reductive group scheme  over $ \Z_p $. The  generic fiber of $G$ is therefore  a connected reductive group  over $\Q_p$,  and is unramified in the sense that it is quasi-split  and split over an unramified extension of $\Q_p$. 
Conversely, every unramified  connected reductive group   over $\Q_p$ is isomorphic to the generic fiber of such a $G$.

\subsubsection{}
\label{general}

Let $([b],  \{\mu\} )$ be a pair consisting of:

\begin{itemize}
\item a $G(\bar K)$-conjugacy class $\{\mu\}$ of  cocharacters $\mu: \Gm_{\bar K}\to G_{\bar K}$,

\item  a $\sigma$-conjugacy class $[b]$ of elements  $b\in G(K)$.
\end{itemize}
Here  $b$ and $b'$ are $\sigma$-conjugate if there is $g\in G(K)$ with $b'=gb\sigma(g)^{-1}$.

  We let $E \subset \bar{K}$ be the field of definition of the conjugacy class $\{\mu\}$. This is the 
\emph{local reflex field.}
Denote by $\co_E$ its valuation ring and by $k_E$ its (finite) residue field. 
In fact, under our assumption on $G$, the field $E\subset \bar K$ is 
contained in $K$ and there is a cocharacter 
$
\mu: \Gm_E\to G_E
$ 
in the conjugacy class $\{\mu\}$ that is  defined over $E$; see \cite[Lemma (1.1.3)]{KottTwisted}.
In fact, we can find a representative  $\mu$ that extends to an integral cocharacter 
\begin{equation}\label{integral mu rep}
\mu: \Gm_{\co_E}\to G_{\co_E},
\end{equation}
and the $G(\co_E)$-conjugacy class of such an $\mu$ is well-defined. 
In what follows, we usually assume that $\mu$ is such a representative.
	
To the conjugacy class  $\{\mu\}$ we associate the homogeneous space 
\[
M_{G, \mu }=G_{\co_E}/P_{\mu }
\] 
over $\co_E$, in which  $P_{\mu}\subset G_{\co_E}$ is the parabolic subgroup defined by $\mu $.
More precisely,  $P_\mu$ is  the parabolic subgroup  such that $P_\mu\times_{\co_E}W$
contains exactly the root groups $U_a$ of the split group $G_W$, 
for all roots $a$ with $a\cdot \mu\geq 0$.  The group $P_\mu\times_{\co_E}W$
stabilizes the filtration  defined by $\mu$ in any representation of $G_W$.
\quash{Denote by $U^\mu=U^\mu_G $ the unipotent radical of the opposite to $P_\mu$ parabolic subgroup of $G_{\co_E}$ and by $U^{\mu,\wedge}$ the formal completion of $U^\mu $
at its identity section over $\co_E$.  Then $U^{\mu,\wedge}$ can be identified with
the formal completion of $M_{G, \mu}$ at the section given by $1\cdot P_\mu/P_\mu$.}

We  write $\mu^\sigma=\sigma(\mu)$ for the Frobenius conjugate of (\ref{integral mu rep}).

\begin{definition} (cf.~\cite[Def. 5.1]{RapoportVi})
A  {\it local  unramified  Shimura datum}  is a triple 
$(G, [b], \{\mu\})$, in which  $G$ is a connected reductive group
over $\Z_p$, the pair $([b], \{\mu\})$ is as above, and we assume
\begin{enumerate}
\item  $\{\mu\}$ is minuscule,  

\item  for some (equivalently, any)   integral representative  (\ref{integral mu rep})  of $\{\mu\}$,
the $\sigma$-conjugacy class $[b]$ has a representative  
\begin{equation}\label{b coset}
 b\in G(W)\mu^\sigma(p)G(W).
 \end{equation}
\end{enumerate}  
 
\end{definition}

By \cite[Theorem 4.2]{RapoportRich}, assumptions (i) and (ii) imply that $[b]$ lies in the set $B(G_{\Q_p}, \{\mu\})$ 
of neutral acceptable elements for $\{\mu\}$; see \cite[Definition 2.3]{RapoportVi}. 
In particular, $(G_{\Q_p}, [b], \{\mu\})$ is a local Shimura datum in the sense of  \cite[Def. 5.1]{RapoportVi}.

\begin{definition}\label{Hodgetypedef}
%\label{def:shimura-hodge} 
The local unramified Shimura datum  $(G, [b], \{\mu\})$  is of {\sl Hodge type} if there exists a closed group scheme embedding 
$
\iota : G\hookrightarrow \GL( C),
$
for  a free $\Z_p$-module $ C$ of finite rank, with the following properties: The central torus $\Gm\subset \GL( C)$ is contained in $G$, and,
after a choice of basis $ C_{\co_E}\iso  \co_E^n$, the composite cocharacter 
\[
\iota \circ \mu: \Gm_{\co_E}\to {\rm GL}_{n, \co_E}
\] 
is the inverse of the minuscule cocharacter\footnote{The notation $a^{(r)}$ means that there are $r$ copies of $a$.}
 \[ a\mapsto {\rm diag}(a^{(r)}, 1^{(n-r)})\]  for some $1\leq r< n$.
 \end{definition}
 
 \begin{definition}\label{def:shimura-hodge} 
Let $(G, [b], \{\mu\})$ be a  local unramified Shimura datum  of Hodge type. 
A \emph{local  Hodge embedding datum} for $(G, [b], \{\mu\})$ consists of 
\begin{itemize}
\item
 a group scheme embedding $\iota : G\hookrightarrow \GL( C)$ as above, \ref{integral mu rep}),
\item
the $G(W)$-$\sigma$-conjugacy class $\{gb\sigma(g)^{-1} : g\in G(W)\}$ of a representative 
\[
b \in G(W)\mu^\sigma(p)G(W) 
\]  
of $[b]$, where  $\mu : \mathbb{G}_{mW} \to G_W$ is chosen to be an integral representative of the $G(\bar{K})$-conjugacy class $\{\mu\}$.
Note that such a  representative $\mu$ is unique  up to $G(W)$-conjugacy.   
\end{itemize}
  The quadruple $(G, b, \mu,  C)$, where $\mu$ is given up to $G(W)$-conjugation, and $b$ up to $G(W)$-$\sigma$-conjugation, is a  {\sl local unramified Shimura-Hodge datum}.  
\end{definition}

By definition,  there is a  surjection 
$
(G, b, \mu,  C) \mapsto (G, [b], \{\mu\})
$
 from the set of local unramified Shimura-Hodge data to the set of local unramified Shimura data of Hodge type.

Fix a  local unramified  Shimura-Hodge datum  $(G, b, \mu,  C)$,  and set 
$
 D = \Hom_{\Z_p}( C , \Z_p)
$
with the contragredient action of $G$. 

\begin{lemma}\label{lemma121} 
 Up to isomorphism, there is a unique   $p$-divisible group 
\[
X_0=X_0(G, b, \mu,  C)
\] 
over $k$ whose  contravariant Dieudonn\'e module   is  $\Db(X_0)(W) =  D_W$ 
with Frobenius $F=b \circ \sigma$.  Moreover,  the Hodge filtration 
\[
V D_k\subset  D_k=\Db(X_0)(k)
\] 
is   induced by   a conjugate of the reduction $\mu_k: \Gm_k\to G_k $.
\end{lemma}

\begin{proof} By our assumption  on $\mu$ in  Definition \ref{Hodgetypedef},
we have $\mu(p) D_W\subset  D_W$.
Therefore, by (\ref{b coset}), the lattice 
$
 D_W  \subset  D_W[1/p]
$
is  $F$--stable.  To determine $VD_W$,   write  $b=h' \mu^\sigma(p)h $  with $h, h'\in G(W)$, so that
\begin{align*}
V D_W=pF^{-1} D_W 
& = p\sigma^{-1} (b^{-1} D_W)\\
& =\sigma^{-1} (h^{-1}p\mu^\sigma(p)^{-1}h'^{-1} D_W) \\
&=h_1 p\mu(p)^{-1}h_1^{-1} D_W
\end{align*}
for  $h_1=\sigma^{-1}(h^{-1})\in G(W)$.
Observe that $p\mu(p)^{-1} D_W\subset  D_W$, and in fact
the filtration 
\[
(p\mu(p)^{-1} D_W)/p D_W\subset  D_W/p D_W= D_k
\]
 is induced by $\mu_k: \Gm_k\to G_k $.  
 
The above calculation shows that $V D_W\subset  D_W$, and that the Hodge filtration
$V D_k\subset  D_k$ is induced by the conjugate $\bar h_1\mu_k\bar h_1^{-1}$.
\end{proof}

\subsubsection{}\label{sss:tate tensors}
By \cite[Prop. (1.3.2)]{KisinJAMS}, there is finite list  $(s_\alpha)$ of tensors $s_\alpha$ in the total tensor algebra $ C^\otimes$ 
that ``cut out'' the group $G$, in the sense that 
\[
G(R)=\{g\in \GL( C \otimes_{\Z_p} R ) :  g\cdot (s_\alpha\otimes 1)=(s_\alpha\otimes 1),\ \forall \alpha   \}
\]
for all $\Z_{p}$-algebras $R$. 
Using the canonical isomorphism  $ C^{\otimes}= D^{\otimes }$,   the tensors 
$s_\alpha \in  C^\otimes$ determine  tensors   $s_\alpha \otimes 1\in D^\otimes \otimes _{\Z_p} W$.
Thus, if  $X_0$ is the $p$-divisible group  of Lemma \ref{lemma121}, we obtain tensors 
\[
t_{\alpha, 0}=s_\alpha\otimes 1\in  D^\otimes \otimes_{\Z_p} W = \Db(X_0)(W) ^\otimes,
\]
which are  Frobenius invariant when viewed in  $\Db(X_0)(W) ^\otimes[1/p]$.  The tensors 
$t_{\alpha, 0}$  uniquely determine morphisms of 
crystals $t_{\alpha, 0}: {\bf 1}\to \Db(X_0)^\otimes$ over $\Spec(k)$,
such that each
\[
t_{\alpha, 0}: {\bf 1} [ 1/p] \to \Db(X_0)[1/p]^\otimes
\]
 is Frobenius equivariant.\footnote{
Since $\Db(X)^\otimes$ also involves the dual, the Frobenius is not defined on $\Db(X)^\otimes$, but only on $\Db(X)^\otimes[1/p]$; see \S\ref{crystalstuff}.}  Here, as before,   we denote by ${\bf 1}=\Db(\Q_p/\Z_p)$ the crystal determined by the 
 Dieudonn\'e module  $W$ with $F=\sigma$. Using Lemma \ref{lemma121},  we easily see that
$
t_{\alpha, 0}(k)\in {\rm Fil}^0(\Db(X_0)(k)^\otimes).
$

\subsubsection{}\label{sss:J}

By \cite{KottIso}, every $\sigma$-conjugacy class in $G(K)$ is   decent in the sense of \cite[Def.~1.8]{RapZinkBook}.  
By \cite[Prop.~1.12]{RapZinkBook},  any   $b\in G(K)$ determines  a smooth affine group scheme $J_b$ over $\Q_p$
with functor of points 
\[
J_b(R)=\{g\in G(R\otimes_{\Q_p}K) :  gb\sigma(g)^{-1}=b\}
\]
for any $\Q_p$-algebra $R$.  Up to isomorphism, $J_b$ depends only on the $\sigma$-conjugacy class $[b]$.

\subsubsection{}\label{sss:slope}
Let $\mathbb{T}$ be the pro-torus over $\Z_p$ with character group $\Q$.  For any $\Z_p$-algebra $R$, an $R$-point
$z\in \mathbb{T}(R)$ consists of a tuple 
\[
z=(z_m\in R^\times)_{m\in \Z_{>0} }
\]
such that  $z_m = z_{md}^d$ for all positive $m$ and $d$.  The character indexed by the rational number $s/t$ sends
$z\mapsto z_t^s$.  

Kottwitz \cite{KottIso} attaches to every $b\in G(K)$ a \emph{slope cocharacter} 
\[
\nu_b : \mathbb{T}_K \to G_K
\]
such that for any representation $\psi :G_{\Q_p} \to \GL(M)$ on a $\Q_p$-vector space $M$, the decomposition
\[
M_K = \bigoplus_{s/t \in \Q} M^{s/t}_K
\] 
of $M_K$ determined by the cocharacter $\psi\circ \nu_b:\mathbb{T}_K \to \GL(M_K)$ agrees with the slope decomposition 
of the isocrystal $(M_K, \psi(b) \circ \sigma)$.   The slope cocharacter depends only on the $\sigma$-conjugacy class $[b]$.

An element  $b\in G(K)$ is \emph{basic}  if its slope cocharacter  $\nu_b$ factors through  
the center of $G_{ K}$.   By \cite{KottIso}, $b$ is basic if and only if the group $J_b$ is an inner form of $G$.

\subsection{Rapoport-Zink formal schemes and functors} 
\label{ss:hodge RZ space}

In this paragraph we define Rapoport-Zink formal schemes and functors
associated to a local unramified Shimura-Hodge datum $(G, b, \mu,  C)$
as defined in \S \ref{ss:local shimura data}. We start by recalling the definition of some 
``classical'' Rapoport-Zink functors.

\subsubsection{}\label{sss:RZ} Suppose that $X_0$ is any $p$-divisible group over $k$. 
The \emph{Rapoport-Zink space}  $\RZc(X_0)$ of deformations of $X_0$ up to quasi-isogeny
is,  as in \cite{RapZinkBook},   the formal scheme over $\Spf(W)$ that represents
the functor  assigning to each  scheme $S$ in ${\rm Nilp}_W$  the set of isomorphism classes of pairs $(X,  \rho)$ in which
\begin{itemize}
\item $X$ is a $p$-divisible group   over $S$,

\item $\rho: X_0\times_k \bar S\dashrightarrow X\times_S\bar S$ is a quasi-isogeny, where $\bar S=S\otimes_Wk$.
\end{itemize}

Suppose now that $X_0$ comes with a principal polarization 
$
\lambda_0: X_0\xrightarrow{\sim} X_0^\vee.
$
The \emph{symplectic Rapoport-Zink space} $\RZc(X_0,\lambda_0)$   is the formal scheme over $\Spf(W)$ 
that represents  the functor   that assigns to each  $S$ in ${\rm Nilp}_W$  the set of isomorphism classes of 
triples $(X, \lambda, \rho)$ in which
\begin{itemize}
\item $X$ is a $p$-divisible group  over $S$,

\item $\lambda: X\iso X^\vee$ is a principal polarization,

\item $\rho: X_0\times_k \bar S\dashrightarrow X\times_S\bar S$ is a quasi-isogeny that respects polarizations up to a scalar, 
in the sense that, Zariski locally on $\bar S$, we have 
\[
\rho^\vee\circ\lambda\circ \rho=c^{-1}(\rho)\cdot \lambda_0,
\]
for some $c(\rho)\in \Q_p^\times$. 
\end{itemize}
 By [\emph{loc.~cit.}] the formal schemes  $\RZc(X_0)$ and $\RZc(X_0,\lambda_0)$ are formally smooth and locally formally of finite type over $W$,
and forgetting the polarization defines a closed immersion 
$
\RZc(X_0,\lambda_0) \to \RZc(X_0) .
$

\subsubsection{} Suppose that $(G, b, \mu,  C)$ is a local unramified Shimura-Hodge datum. 
Choose tensors $(s_\alpha)$ that cut out $G$ as in \S\ref{sss:tate tensors}. Denote by 
\[
X_0=X_0(G, b, \mu,  C)
\] 
the corresponding $p$-divisible group over $k$ of Lemma \ref{lemma121}, with  its  Frobenius invariant crystalline tensors $(t_{\alpha, 0})$.

\begin{definition}\label{defRZG}
Consider the functor \[\RZnilp_G=\RZnilp_{G, b, \mu, C,(s_\alpha)} : {\rm ANilp}_W \to {\rm Sets}\]
that assigns to each  $R\in {\rm ANilp}_W$ the set of isomorphism classes of triples
$(X,   \rho, (t_\alpha))$ in which 
  \begin{itemize}
\item $(X,   \rho)$ consists of a $p$-divisible group over $\Spec(R)$ and a quasi-isogeny 
\[
\rho: X_0\otimes_k \bar R\dashrightarrow X\otimes_R\bar R,
\] 
with $\bar R=R/pR$, as in the definition of the   Rapoport-Zink formal scheme $\RZc(X_0)$,

\item the collection $(t_\alpha)$  consists of morphisms of crystals $t_\alpha: {\mathbf 1}\to \Db(X)^\otimes$  over $\Spec(R)$
with $t_\alpha: {\mathbf 1}[1/p]\to \Db(X)^\otimes[1/p]$ Frobenius equivariant,
\end{itemize}
satisfying the following properties:
\begin{enumerate}
\item
For some nilpotent ideal $J\subset R$ with $p\in J$, the pull-back of $t_\alpha$ over $\Spec(R/J)$ 
is identified  with $t_{\alpha, 0}$ under the isomorphism of isocrystals
\[\Db(\rho):  \Db(X_{R/J})^{\otimes}[1/p]\iso   \Db(X_0\times_k R/J)^{\otimes}[1/p]\]
induced by the quasi-isogeny  $\rho$.

\item
 The sheaf  of $G_W$-sets  over ${\rm {CRIS}}(\Spec(R)/W)$ given by isomorphisms 
 \[
 \underline{\rm Isom}_{t_\alpha, s_\alpha\otimes 1 }( \Db(X),    D \otimes_{\Z_p} R)
 \] 
 that respect  the tensors as indicated, is a crystal of  
 $G_W$-torsors, \emph{i.e.} a crystal fppf locally
 isomorphic to the crystal defined by $G_W$.
  
\item
 There exists an \'etale cover $\{U_i\}$ of $\Spec(R)$,  and for each $i$ 
an isomorphism \[\Db(X_{U_i})_{U_i}\iso   D\otimes_{\Z_p}\co_{U_i}\] of vector bundles 
respecting the tensors $t_\alpha$ and $s_\alpha\otimes 1$ as in (ii), such that the 
Hodge filtration  
\[
{\rm Fil}^1 ( X_{U_i} ) \subset\Db(X_{U_i})_{U_i} \iso  D\otimes_{\Z_p}\co_{U_i}
\]   
is induced by a  cocharacter that  is $G(U_i)$-conjugate to $\mu$.
\end{enumerate}

Two triples $(X,  \rho, (t_\alpha))$ and $(X',   \rho', (t'_\alpha))$
are identified if there is an isomorphism $X\iso X'$ of $p$-divisible
groups that respects the rest of the data in the obvious manner.

Above, ${\rm {CRIS}}(\Spec(R)/W)$ denotes the big fppf crystalline site of $\Spec(R)$ over $(W, (p), \gamma)$
with $\gamma$ the natural PD-structure,  as in \cite[1.1]{BerthelotBreenMessingII}. Condition (ii)  implies
that,  for any nilpotent PD thickening $R'\to R$ of $R$, the $\Spec(R')$-scheme of isomorphisms of finite locally free $R'$-modules 
\[
T_{R'} =\underline{\rm Isom}_{R', t_\alpha(R'), s_\alpha\otimes 1 }( \Db(X)(R'),    D\otimes_{\Z_p} R'    )
\]
 is a $G_{R'}$-torsor.
\end{definition}

\begin{remark}\label{rem:base point}
There is a distinguished point
\[
x_0 = ( X_0 , \rho_0 , (t_{\alpha,0}) ) \in \RZnilp(k),
\]
defined by taking $\rho_0$ to be the identity quasi-isogeny $X_0 \dashrightarrow X_0$.
\end{remark}

\begin{remark}\label{remark1}
a)  Suppose $(p)\subset J'\subset J$ with $J'$ also 
 nilpotent. Then a power of the Frobenius of $R/J'$ factors through $R/J$. 
 Since $t_{\alpha}$ are Frobenius equivariant we obtain that condition (i)
 is independent of the ideal $J$. In particular, we can simply take $J=(p)$.

b) Conditions (ii) and (iii) together imply  the following:
The $\Spec(R)$-scheme of $R$-linear isomorphisms 
\[
\Db(X)(R)\iso   D\otimes_{\Z_p} R
\] 
identifying  $t_\alpha(R)$ with  $s_\alpha\otimes 1$, and 
identifying the Hodge filtration ${\rm Fil}^1(X) \subset \Db(X)(R)$ with  standard filtration 
$F_\mu\otimes_W  R\subset  D \otimes_{\Z_p}R$ defined by $\mu$, is a $P_{\mu}\times_WR$-torsor.

c) For $R$ in ${\rm ANilp}^\mathrm{f}_W$, the categories of $p$-divisible groups
over $\Spf(R)$ and over $\Spec(R)$ are naturally equivalent, by \cite[Lemma 2.4.4]{deJongBTIHES}.
For $R$   in  ${\rm ANilp}^{\rm fsm}_W$,  the argument in the proof of  \cite[Prop. 2.4.8]{deJongBTIHES} shows that each
morphism of crystals $\hat t_\alpha: {\mathbf 1}\to \Db(X)^\otimes$ over $\Spf(R)$
is induced by a unique   morphism of crystals  $t_\alpha: {\mathbf 1}\to \Db(X)^\otimes$ over $\Spec(R)$.

d) For $\bar R$ of finite type over $k$, we will see that
it is enough to verify (i) over one closed point of 
each connected component of $\Spec(R)$; see   Lemma \ref{convIsocr}
and its proof. 
 \end{remark}

\subsubsection{}\label{RZGfsm} Define a functor $\RZfsm_G$ on
  ${\rm ANilp}_W^{\rm fsm}$ by setting 
\[
\RZfsm_G(A)=\varprojlim\nolimits_n \RZnilp_G(A/I^n),
\]
where $I$ is an ideal of definition of $A$.

Assume  that $I$ is chosen with $p\in I$.  By Remark (c) above and the rigidity of quasi-isogenies \cite{DrinfeldCoverings}, 
we see that elements of $\RZfsm_G(A)$ correspond to isomorphism classes of triples
 $(X,  \rho, (t_\alpha))$, in which  $X$ is a $p$-divisible group over $\Spec(A)$, 
 \[
 \rho : X_0\times_k A /I   \dashrightarrow  X\times_A  A/I
 \]
is  a quasi-isogeny, and $t_\alpha$ a morphism of crystals over $\Spec(A)$, such that (i), (ii), and (iii) above are satisfied.  
The definition is independent of the choice of $I$.

Since any object $A$ of  ${\rm ANilp}_W^{\rm fsm}$ is also  an object of   ${\rm ANilp}_W$,
it  makes sense to consider $\RZnilp_G(A)$. We will rarely do this unless $A$ is discrete,
in which case  
\[
\RZnilp_G(A)=\RZfsm_G(A).
\]
The  difference between $\RZfsm_G(A)$ and $\RZnilp_G(A)$
is that, in the former, we ask that the quasi-isogeny $\rho$ only exists over $A/I$,
with $I$ an ideal of definition of the adic algebra $A$. For  $A$ in ${\rm ANilp}_W^{\rm fsm}$  it will often be the case that  $\RZnilp_G(A)=\emptyset$, while $\RZfsm_G(A)\neq\emptyset$.

\subsubsection{}
The closed immersion  $\iota : G\hookrightarrow \GL( C)$ induces
an injective homomorphism from the group $J_b(\Q_p)$ into the group 
${\rm Aut}_{\Q_p}(X_0)$ of quasi-automorphisms of  the $p$-divisible group $X_0$, \emph{i.e.}~of automorphisms of   $X_0$ 
up to isogeny.  
In addition, we can see that   the induced action of $J_b(\Q_p)$ on $\Db(X_0)(W)^\otimes[1/p]$
preserves the tensors $t_{\alpha, 0}$.
Therefore, the group $J_b(\Q_p)$ acts on the functors $\RZnilp_G$ and  $\RZfsm_G$  on the left by 
\begin{equation}\label{Jaction}
g\cdot (X,  \rho, (t_\alpha))=(X,  \rho\circ g^{-1}, (t_\alpha)).
\end{equation}

\subsection{Field valued points and   affine Deligne-Lusztig sets} \label{par:ADL}

We  now introduce some  refined  affine  Deligne-Lusztig sets, and  show that these 
can be used to parametrize the set $\RZnilp_G(k')$  for any finitely generated field extension $k'/k$.

\subsubsection{}

For $G$, $b$ and $\mu: \Gm_W\to G_W$  as in the beginning of \S\ref{general},
the ``classical'' affine Deligne-Lusztig set  is
\[
X_{G, b, \mu}(k) =\big\{g\in G(K) :   g^{-1}b\sigma(g)\in G(W){\mu}(p) G(W) \big\}/G(W).
\]
We will define an analogous set for any extension field  $k'/k$.

Let $W'$ be the Cohen ring of $k'$, let $K'=W'[1/p]$ be its fraction field,
 and  let  $\sigma: W'\to W'$  be  a lift of the absolute Frobenius as in \S \ref{sss:cohen}.
 Consider the set
\begin{equation}\label{PreRADL}
  \big\{g \in G(K')  :  g^{-1}b\sigma(g) \mu (p)^{-1} \in G(W')\big\},
\end{equation}
and define
\[
Q_\mu(W' ) =G(W' )\cap  \mu^{\sigma^{-1}}(p)^{-1}  G(W')\mu^{\sigma^{-1}}(p),
\]
the intersection  taking place in $G(K')$.

The right translation action of $Q_\mu(W')$  on $G(K')$ preserves (\ref{PreRADL}).
Indeed, if $q\in Q_\mu(W')$ and $g$ belongs to (\ref{PreRADL}), then
$
U_g=g^{-1}b\sigma(g)\mu(p)^{-1}
$ 
belongs to $G(W')$, and  hence so does
\[
(gq)^{-1} b\sigma(gq)\mu(p)^{-1}=q^{-1}g^{-1}b\sigma(g)\sigma(q)\mu(p)^{-1}
=q^{-1}U_g\mu(p)\sigma(q)\mu(p)^{-1}.
\]
 Thus $gq$ belongs to (\ref{PreRADL}).

\begin{definition}
The \emph{refined affine Deligne-Lusztig set}   is the quotient
\[
{X}_{G, b,  \mu, \sigma}(k') =\big\{g\in G(K') :  g^{-1}b\sigma(g) \mu (p)^{-1} \in G(W') \big\}/Q_{\mu}(W' ).
\]
Similarly, we have the \emph{naive affine Deligne-Lusztig set} 
\[
{X}^{\rm naive}_{G,b, \mu, \sigma}(k') = \big\{g\in G(K') : g^{-1}b\sigma(g)\in G(W') \mu (p) G(W') \big\}/G(W' ).
\]

\end{definition}

For simplicity, we will often omit $\sigma$ from the list of subscripts. However, we do not know if the set
$
{X}_{ G, b, \mu, \sigma}(k')$ is  independent of the choice of the lift of Frobenius $\sigma$.

\begin{proposition}
%\label{RvsADL} 
The refined affine Deligne-Lusztig sets have the following properties.
\begin{enumerate}
\item
 Sending $gQ_{\mu}(W' )$ to $gG(W')$ defines an injection
\begin{equation*}
\adl(k'): {X}_{G,b, \mu}(k')\hookrightarrow {X}^{\rm naive}_{G, b,\mu}(k')\subset G(K')/G(W').
\end{equation*}
  If $k'$ is perfect,   then $\adl(k')$ is a bijection. 

\item
 If $b'$ is $\sigma$-conjugate to $b$, say  $b'=h^{-1}b\sigma(h) $ with  $h\in G(K)$,
then $g\mapsto hg$ defines a bijection 
\[
X_{G,b', \mu}(k')\iso   X_{G,b,\mu}(k').
\]
  
\item
 If $k'$ is perfect, then  $g\mapsto \sigma^{-1}(b^{-1} g)$  defines  a bijection
  \begin{equation*}
  {X}_{ G,b, \mu^\sigma}(k')\iso  {X}_{G,b, \mu}(k').
  \end{equation*}
  \end{enumerate}
\end{proposition}

\begin{proof} We first show (i).
The condition for $g$ in the refined notion is stronger then the condition in the naive notion;
since $Q_\mu(W')\subset G(W')$ the map is well-defined. 
It remains to show that it is injective. Let $g$, $g'\in G(K')$, and assume   
\begin{align*}
U_g \define g^{-1}b\sigma(g) \mu (p)^{-1} &\in G(W') \\
U_{g'} \define {g'}^{-1}b\sigma(g') \mu (p)^{-1} & \in G(W').
\end{align*}
 Suppose  there is an $h\in G(W')$ such that $g'=gh$.  Then we obtain
\[
U_{g'}=h^{-1}g^{-1}b\sigma(g)\sigma(h)\mu(p)^{-1}=h^{-1}U_g  \mu(p)\sigma(h)\mu(p)^{-1}.
\]
As $h^{-1}U_g$ and  $U_{g'}$ are in $G(W')$, we see that   the element
\[
\mu(p)\sigma(h)\mu(p)^{-1}=\sigma(\mu^{\sigma^{-1}}(p)h\mu^{\sigma^{-1}}(p)^{-1}) \in G(K)
\] 
actually lies in  $G(W')$. Since $\sigma: W'/p^iW'\to W'/p^iW'$ is injective for all $i$, and $G$ is affine and flat over $\Z_p$,
this implies that
\[
\mu^{\sigma^{-1}}(p)h\mu^{\sigma^{-1}}(p)^{-1}\in G(W').
\]
It follows that $h\in Q_\mu(W')$.  This shows the injectivity of the map $\adl(k')$. 

Now suppose that   $k'$ is perfect  so that $\sigma^{-1}$ makes sense on $W'=W(k')$.
If $g\in G(K')$ is such that $g^{-1}b\sigma(g)=h_1\mu^\sigma(p)h_2$ with $h_i\in G(W')$,
then $g'=g\sigma^{-1}(h_2)h_1$ satisfies the refined condition. Hence  $\adl(k')$
is surjective. 

Part (ii) is routine. To show part (iii), observe that for $h=\sigma^{-1}(b^{-1}g)$,
we have 
\[
h^{-1}b\sigma(h)=\sigma^{-1}(g^{-1}b)bb^{-1}g=\sigma^{-1}(g^{-1})\sigma^{-1}(b)g=\sigma^{-1}(g^{-1}b\sigma(g))
\]
and the result follows.
\end{proof}

\subsubsection{} Suppose  
$\iota : G\hookrightarrow \GL_n$ is a closed  immersion of group schemes   
over $\Z_p$,  and set $\nu= \iota \circ \mu: \Gm_W\to \GL_{n,W}$. For $\tau=\sigma^{-1}\in \Aut(W)$, we have
\[
G(W')=G(K')\cap \GL_n(W'),
\]
and
\[
  \mu^{\tau}(p)^{-1} G(W')\mu^{\tau}(p)=
G(K')\cap \nu^{\tau}(p)^{-1}\GL_n(W')\nu^{\tau}(p),
\]
the intersections taking place in $\GL_n(K')$.
The imbedding $\iota$ then induces injections 
\[
G(K')/G(W')\hookrightarrow \GL_n(K')/\GL_n(W'),
\]
and
\[
 G(K')/Q_\mu(W')\hookrightarrow \GL_n(K')/Q^{\GL_n}_\nu(W').
\]
Moreover, $g\in G(K')$ satisfies $g^{-1}b\sigma(g)\mu(p)^{-1}\in G(W')$ if and only if 
$\iota(g)\in \GL_n(K')$ satisfies the corresponding condition with $(G, b,\mu)$ replaced by $( \GL_n , \iota(b),\nu)$.
It follows that $\iota$ defines an injection
\[
{ X}_{G, b, \mu}(k')\hookrightarrow {X}_{\GL_n, \iota (b), \nu}(k').
\]

\subsubsection{} We now return to the set up of \S\ref{ss:hodge RZ space}.
Assume that $(G, [b], \{\mu\})$ is an unramified local Shimura datum of Hodge 
type.   Fix a corresponding local Shimura-Hodge datum $(G, b, \mu,  C)$ and a set of tensors $(s_\alpha)$
that cuts out $G\subset \GL( C)$.  We then have the functor $\RZnilp_G$  as before.

Fix a point
\[
 (X,  \rho, (t_\alpha)) \in \RZnilp_G(k').
 \] 
 Consider the value $M=\Db(X)(W')$ of the crystal $\Db(X)$ on the Cohen ring $W'$ of $k'$,
viewed as a  ${\rm PD}$-thickening of $k'$. We have the tensors $t_\alpha(W')\in M^{\otimes}$,  which are Frobenius
invariant in $M^\otimes[1/p]$,  and the quasi-isogeny $\rho$,  which induces an isomorphism
\[
\Db(\rho): M[1/p]\iso  M_0[1/p]
\]
 such that  $\Db(\rho)(t_{\alpha}(W'))= t_{\alpha, 0}(W)\otimes 1$,  by (i) of Definition \ref{defRZG}.

\begin{lemma}\label{Ctorsor}
Under these assumptions, the scheme
\[
T=\underline{\rm Isom}_{W', t_\alpha(W'),  s_\alpha\otimes 1 }( M,    D\otimes_{\Z_p} W'    )
\]
is a trivial $G_W'$-torsor over $\Spec(W')$. 
\end{lemma}

\begin{proof}
Note that $T$ is an affine finite type $W'$-scheme carrying an action of 
the group scheme $G_W'$. We will first show that $T$ is a $G_W'$-torsor over $\Spec(W')$.

For any $m\geq 1$, we have the nilpotent PD thickening  $W'/p^mW'\to k'$.
Therefore, by condition (ii) in the definition of $\RZnilp_G$, 
the base change $T\times_{W'} {W'/p^mW'}$ is a $G\times_{\Z_p}W'/p^mW'$-torsor. 
It follows from the local criterion of flatness that $T$ is $W'$-flat and hence 
 also faithfully flat (since the special fiber is non-empty).
Since $G$ acts transitively on the points of $T$ it now follows 
that $T$ is a (fppf locally trivial) $G_W'$-torsor over 
$\Spec(W')$. Since $G$ is smooth, the torsor $T$ splits locally 
for the \'etale topology of $\Spec(W')$. 

We can easily see that the generic fiber  
$T\times_{W'}K'$ is a trivial $G_{K'}$-torsor with  a section constructed using 
a composition of $\Db(\rho)(W')$ with the identification $ \Db(X_0)(W')\cong  D\otimes_{\Z_p}W'$. By \cite[Theorem 5.2]{NisnevichThesis} (a very special case
of a conjecture of Grothendieck),
which applies since $G$ is  quasi-split, $T$ is   a trivial torsor. \end{proof}

\subsubsection{} 
We  now describe $\RZnilp_G(k')$ in terms of a {refined} affine Deligne-Lusztig 
(ADL) set.   The following  may be standard, but we could not find a reference.

\begin{proposition}\label{fgfield} Suppose that $k'/k$ is  a finitely generated field extension
and denote by $W'$ the Cohen ring of $k'$. There exists a lift of Frobenius $\sigma: W'\to W'$ with the following property. We can write $k'=\varinjlim R$, where $R$ are finitely generated smooth $k$-algebras,
each having a finite $p$-basis, such that for each $R$ that appears in the limit there
is a $W$-flat formally smooth $p$-adically complete and separated lift $\widetilde R$ of $R$ with 
$\widetilde R\subset W'$ lifting $R\to k'$ which is such that  $\sigma(\widetilde R)\subset \widetilde R$. 
\end{proposition}

\begin{proof} 
Suppose that $R$ is a finitely generated smooth $k$-algebra  which is a  domain and is such that
$k'$ is the fraction field of $R$. By replacing $R$ by a localization we can assume that   the differentials $\Omega_{R/k}$ are a free $R$-module of rank equal to the Krull dimension of $R$;  let $dx_i$, $i=1,\ldots, d$,   be an  $R$-basis of $\Omega_{R/k}$.
 In this situation, the absolute Frobenius $\phi_R: R\to R$ is injective and makes $R$ into a finitely generated $R$-module.
Therefore,    by \cite[Ch. $0$, Prop. (21.1.7)]{EGAIV},  the tuple $(x_i)$ with  $i=1,\ldots, d$  is a system of $p$-generators of $R$ over $k$, \emph{i.e.}~ $R=k[R^p, (x_i)]$. In fact, we can easily see that,
since $dx_i$ are $R$-linearly independent, the $x_i$ are $p$-independent
(\emph{cf.}~\cite[p.~276]{Matsumura}). Therefore, the  $x_i$'s  form a $p$-basis of $R$ over $k$.

If we start with a $p$-basis  $(x_i)$ of $k'$, then   \cite[Theorem 86]{Matsumura} implies that  $(dx_i)$ are a basis of the $k'$-vector space 
$\Omega_{k'/k}$.  If $R\subset k'$ is any smooth finite type $k'$-algebra with $k'={\rm Frac}(R)$ such that $x_i\in R$ and $(dx_i)$
generate $\Omega_{R/k}$, then,   by the above,  $(x_i)$ also provide   a $p$-basis of $R$. 

Since $k$ is perfect, we can write $k'=\varinjlim R$,
where $R$ is as above. 
Now,  as in \cite[\S1.1]{BerthelotMessingIII}, using the $p$-basis $(x_i)$ we  obtain a 
concrete construction of the Cohen ring 
\[
W'=\varprojlim_n A_n(k'),
\]  
and of a $W$-flat lift \[ \widetilde R=\varprojlim_n A_n(R)\]  of $R$. 
Here  $A_n(R)$ and  $A_n(k')$, are certain subrings 
of the truncated Witt vector rings $W_n(R)$ and  $W_n(k')$. By  [\emph{loc.~cit.}], sending $x_i\in R$ to $x_i\in k'$   
gives   ring homomorphisms $i_n: A_n(R)\to A_n(k')$. Since $R\to k'$ is injective,
 $A_n(R)\to A_n(k')$ is injective,  and so also
$\widetilde i: \widetilde R\to W'$ is injective. Recall that, by \cite[Prop. 1.2.6]{BerthelotMessingIII}, a lift 
of Frobenius on $A_n(R)$, resp. $W_n(k')$, is uniquely determined by giving (arbitrary) lifts 
$y_{i, n}\in A_n(R)$, resp. $A_n(k')$, of all the elements $x_i^p$. Therefore, we can choose this way lifts
of  Frobenius on $\widetilde R$ and $W' $ that are compatible under $\widetilde i: \ti R\hookrightarrow W' $. 
\end{proof}

\subsubsection{} 
Let $(G, b, \mu,  C)$  be a local unramified Shimura-Hodge datum as in Definition \ref{def:shimura-hodge}, and let   
$(s_\alpha)$ be tensors in $ C^\otimes$ that cut out $G \subset \GL( C)$, as in \S\ref{sss:tate tensors}.

Let $k'/k$ be a finitely generated field extension, and suppose
 that the lift $\sigma: W'\to W'$ of the Frobenius  is chosen as Proposition \ref{fgfield}. 
 Suppose also that the Dieudonn\'e module structure
 on $ D_W$ determined by $\iota(b) \in \GL( C_W)$ has no zero slopes 
   (equivalently,  the base point $p$-divisible group $X_0$ over $k$ defined in Lemma \ref{lemma121} is formal).

\begin{theorem}\label{bijectionRADL}
Under the above assumptions, there are natural bijections
\[
\pi: \varinjlim\nolimits_{R}\RZnilp_G(R)\iso    \RZnilp_G(k')\iso 
{X}_{ G, b, \mu^\sigma,\sigma}(k')
\]
where the limit  over $R$ is as in Proposition \ref{fgfield} above.
\end{theorem}

\begin{proof} 
Let $X_0$ be the $p$-divisible group of Lemma \ref{lemma121}, and recall that 
 $\RZc(X_0)$ is  the (undecorated) Rapoport-Zink formal scheme from \S\ref{sss:RZ}.
Notice that $(X, \rho, (t_\alpha))\mapsto  (X,  \rho)$ defines an injection 
$
\RZnilp_G(k') \hookrightarrow \RZc(X_0)(k'),
$ 
as  $t_\alpha$ is determined by $t_\alpha(W')$.

 Similarly, $\RZnilp_G(R)$ injects to $\RZc(X_0)(R)$. Indeed, $t_\alpha$ is determined by $t_\alpha(\widetilde R)$, and, as
   $\widetilde R$ is torsion-free,  $t_\alpha(\widetilde R)$ is determined by 
   \[
   t_\alpha(\widetilde R)[1/p]=\Db(\rho)^{-1}(t_{\alpha, 0}[1/p]).
   \]
   Since $\RZc(X_0)$ is formally locally of finite type over $W$, we have
\[
\RZc(X_0)(k')=\varinjlim_{R}\RZc(X_0)(R),
\]
and so
$
\varinjlim\nolimits_{R}\RZnilp_G(R)\hookrightarrow \RZnilp_G(k').
$

Pick a point  $x= (X,\rho, (t_\alpha))\in \RZnilp_G(k')$, and  consider the value 
$
M:=\Db(X)(W')
$
 of the crystal $\Db(X)$ on  $W'$, endowed with the tensors  $t_\alpha(W')\in M^{\otimes}$. 
By Lemma \ref{Ctorsor},  the scheme
\[
T_x =\underline{\rm Isom}_{W', t_\alpha(W'),  s_\alpha\otimes 1 }( M,    D\otimes_{\Z_p} W'    )
\]
is a trivial $G$-torsor over $\Spec(W')$.

If $k'=k$ and $x=x_0$  is the base point of Remark \ref{rem:base point},  then there is an isomorphism
\[
\beta_0: M_0:=\Db(X_0)(W)\iso   D\otimes_{\Z_p}W
\]
with
$\beta^{\otimes}_0(t_{\alpha, 0}(W))=s_\alpha\otimes 1$ and we can use this to identify
$M_0= D\otimes_{\Z_p}W.$ In general, the generic fiber of $T_x$ has a section constructed using $\Db(\rho)$ and $\beta_0$. 
 Since the $G$-torsor $T_x$ is trivial  there is $\beta: M \iso  D_W' $  such that 
$\beta^{\otimes}(t_{\alpha}(W'))=s_\alpha\otimes 1$.
Using $\Db(\rho)$ we can identify 
\[
M_\Q=M_0\otimes_{W} K'= D\otimes_{\Z_p} K'
\]
and therefore think of $M\subset M_\Q$ as a $W'$-lattice in $M_0\otimes_{W}K'= D\otimes_{\Z_p}K'$. Under this identification, the choice of $\beta$ is equivalent to picking $g\in G(K')$ such that 
$
M=g\cdot ( D\otimes_{\Z_p} W');
$
 then $\beta$ is given by left multiplication by $g^{-1}$.
Notice that the coset  $gG(W')$ is independent of the choice of $\beta_0$ and $\beta$
and so we have a well-defined map
\[
\RZnilp_G(k')\to G(K')/G(W') 
\]
given by $(X,\rho, (t_\alpha))\mapsto gG(W')$.

By Zink's theory,  as in  \S\ref{Zinkpar} (see especially \cite[Theorem 4]{ZinkWindows}), and using the inclusion
\[
\RZnilp_G(k')\subset \RZc(X_0)(k'),
\] 
we see that the $W'$-lattice $M\subset M_0\otimes_{W}W'[1/p]$
uniquely determines the point $  (X, \rho, (t_\alpha)) \in \RZnilp_G(k')$. 
On the other hand, 
\[
M=g\cdot (M_0\otimes_{W}W')\subset M_0\otimes_{W}W'[1/p]
\]
is uniquely determined by $gG(W')$,  and so this construction gives an injection 
\[
\pi: \RZnilp_G(k')\hookrightarrow G(K')/G(W').
\] 
We have to show that the image of this injection is exactly the refined ADL set ${X}_{G, b, \mu^\sigma, \sigma}(k')$.

This amounts to showing that $M$ is the $W'$-window corresponding to a point in $\RZnilp_G(k')$
if and only if we can pick $g\in G(K')$ such that 
$M=g\cdot (M_0\otimes_{W}W')$ and $g^{-1}b\sigma(g) \mu^\sigma (p)^{-1}\in G(W')$.

Suppose that $M$ indeed corresponds to the point $x\in \RZnilp_G(k')$ as above. 
Then 
\[
\Db(X)(k')=M\otimes_{W'} k'=M/pM
\]
with Hodge filtration given by  ${\rm Fil}^1(X)=M_1/pM\subset M/pM$. Since the Hodge filtration is a  $G$-filtration of type $\mu$ by Definition \ref{defRZG} (iii), we can pick a trivialization $\bar\beta: M/pM\iso    D_k'$ that preserves the tensors
 (\emph{i.e.}~a $k'$-section of the trivial torsor $T_x$ above), and such that  
\[
M_1/pM\subset M/pM\cong  D_k'
\] 
is actually the filtration given by  $\mu$. (Indeed, since $G_W$ is split, the quotient morphism $G_W\to G_W/P_\mu$ splits
locally for the Zariski topology over $W$. Therefore, $G(k')$ acts transitively 
on the set of $G$-filtrations of 
$ D_k'$ of type $\mu$.) By Hensel's lemma, we can lift $\bar\beta$ to 
$\beta: M \iso   D_W'$. Then 
$ 
\beta(M_1)=  p\mu(p)^{-1} D_W' .
$ 
As before, $\beta$ corresponds to $g\in G(K')$
such that 
\[
M=g\cdot (M_0\otimes_W W')=g\cdot  D_W'
\]
and then $\beta$ is multiplication by $g^{-1}$. Hence,  
\[
M= g\cdot  D_W',\quad M_1=g\cdot p\mu(p)^{-1} D_W'.
\]

The quadruple  $(M, M_1, F=b\circ \sigma)$ defines a $W'$-window if and only if
the $W'$-submodule $\langle p^{-1}F(M_1)\rangle$ of $M[1/p] $  generated
by $p^{-1}F(M_1)$ is equal to $M$. Indeed, if  $(M, M_1, F=b\circ \sigma)$ is a $W'$-window
this condition follows from the definition in \S \ref{Zinkpar}. Conversely, assume that
$\langle p^{-1}F(M_1)\rangle=M$. Then 
\[
pM\subset M_1\subset F^{-1}(pM)\subset M
\] 
and   $\langle  F(M_1)\rangle=M=\langle  F(F^{-1}(pM))\rangle$. Hence $M_1=F^{-1}(pM)$,
and then $(M, M_1, F=b\circ \sigma)$ gives a $W'$-module by the observation in \S\ref{Zinkpar}  (the nilpotence condition also follows by a slope argument).  

The condition that  $\langle p^{-1}F(M_1)\rangle=M$ 
reads
\[
\big\langle p^{-1} b \sigma\big(gp\mu(p)^{-1} \cdot  D\otimes_{\Z_p}W' \big) \big\rangle= g\cdot 
( D\otimes_{\Z_p}W'),
\]
which translates to
\begin{equation}\label{eqLattices}
\big\langle  g^{-1}b\sigma(g) \mu^\sigma(p)^{-1}\cdot( D\otimes_{\Z_p}\sigma(W')) \big\rangle
= D\otimes_{\Z_p}W'. 
\end{equation}
Set 
$
u = gb\sigma(g)\mu^\sigma(p)^{-1}\in G(K').
$
Since 
\[
\langle h\cdot ( D\otimes_{\Z_p}\sigma(W'))\rangle=h\cdot\ ( D\otimes_{\Z_p}W') ,
\]
 for any $h\in \GL_n(K')$, the equation (\ref{eqLattices}) above amounts to $u\in \GL_n(W')$.
We obtain that $u$ is in $G(W')=G(K')\cap \GL_n(W')$, and thus $g\in X_{G, b, \mu^\sigma,  \sigma}(k')$.
Therefore, the image of $\pi$ is contained in $ X_{G, b, \mu^\sigma,  \sigma}(k')$.

Let us now discuss the converse.   Start with $g\in X_{G, b, \mu^\sigma,  \sigma}(k')$, and set
\[
M=g\cdot  D_W', \quad M_1=gp\mu(p)^{-1}\cdot  D_W'.
\]
By the argument above, $(M, M_1, b\circ \sigma)$  is a $W'$-window.
By Zink's theory, there is a corresponding  $p$-divisible group $X$ over $k'$ with a quasi-isogeny 
$
\rho:X_0\times_kk'\dasharrow X,
$
and $(X, \rho )$ gives a $k'$-point of the   Rapoport-Zink space $\RZc(X_0)$. 

Since  the underlying reduced scheme  $\RZc(X_0)^{\rm red}$ is locally  of finite type,   we can find a smooth domain $R$ over $k$ 
as above with $k'={\rm Frac}(R)$ and a $p$-divisible group with quasi-isogeny  $(X_R, \rho_R)$ over $R$ that extends $(X, \rho )$.
By replacing $R$ by a localization we can further assume that $R$ has a lift $\widetilde R$ such that $\sigma(\widetilde R)\subset \widetilde R$ and that $
\Db(X_R)(\widetilde R)$ is $\widetilde R$-free.
We have
\[
\Db(X_R)(\widetilde R)\otimes_{\widetilde R}W'\cong \Db(X)(W')\cong M=g\cdot (M_0\otimes_W W').
\]
We will now produce a corresponding $R$-valued point $(X_R, \rho_R,  (t_\alpha))$
of $\RZnilp_G$. For this, we will construct a morphism of crystals  
$t_{\alpha}: {\bf 1}\to \Db(X_R)^\otimes$ such that $t_\alpha: 
{\bf 1}\to \Db(X_R)^\otimes[1/p]$ are Frobenius invariant and then check that $t_\alpha$  satisfy (i), (ii) and (iii)
of Definition \ref{defRZG}. 

By \cite[Prop.~1.3.3]{BerthelotMessingIII}  or \cite[Cor.~2.2.3]{deJongBTIHES}, to give $t_{\alpha}$ as above, 
it suffices to give $t_{\alpha}(\widetilde R)(1)\in \Db(X_R)(\widetilde R)^\otimes$ which are
horizontal for the connection and are Frobenius invariant in $\Db(X_R)(\widetilde R)^\otimes[1/p]$. Consider the images $t_\alpha$ of the
``constant'' tensors $s_\alpha\otimes 1$ under the isomorphism of isocrystals
\[
\Db(\rho)^{-1}  :   \Db ( X_{0,R}   )^\otimes [1/p] \iso \Db    (X_R)^\otimes [1/p]
\] 
induced by the quasi-isogeny $\rho$. Since $\rho$ is defined over $R$, we obtain  $t_\alpha(\widetilde R)\in \Db(X_R)(\widetilde R)^{\otimes}[1/p]$. However, since $g$ is in $G(K')$, 
we actually have $t_\alpha(\widetilde R)=s_\alpha\otimes 1$ in 
\[
M^\otimes[1/p]=\Db(X_R)(\widetilde R)^\otimes\otimes_{\widetilde R}W'[1/p]
\]
 and these lie in 
\[
M^\otimes =\Db(X_R)(\widetilde R)^\otimes\otimes_{\widetilde R}W' .
\]

 Since $\Db(X_R)(\widetilde R) $  is $\widetilde R$-free
and $\widetilde R[1/p]\cap W'=\widetilde R$, we can see that 
our tensors $t_\alpha(\widetilde R)=s_\alpha\otimes 1$ lie in $\Db(X_R)(\widetilde R)^\otimes$.  They are horizontal and Frobenius
invariant  since this is true over $W'$,  and   $\widetilde R\subset W'$. Moreover, conditions (i), (ii) and (ii)
can now be seen to be satisfied after possibly further localizing $R$.  We have now produced an $R$-valued point in $\RZnilp_G$.

These two  constructions are inverses of each other. This shows 
that  $\varinjlim_{R}\RZnilp_G(R)= \RZnilp_G(k')$ and  the image of $\pi$ is $X_{G, b, \mu^\sigma, \sigma}(k')$.
\end{proof}

\begin{remark}\label{remark:ftpoints}
{\rm 
Observe that   $\RZfsm_G(R)=\RZnilp_G (R)$,  and so Theorem \ref{bijectionRADL} also gives 
\[
\varinjlim\nolimits_R\RZfsm_G(R)\iso  {X}_{ G, b, \mu^\sigma,\sigma}(k').
\]
This will be important for the application to Rapoport-Zink formal schemes. 
As we will see, these are defined via the functor $\RZfsm_G$ which can be evaluated 
at $R$ as above but not at $k'$. 

We can directly obtain a bijection $\RZnilp_G(k')\iso 
{X}_{ G, b, \mu^\sigma,\sigma}(k')$, without assuming that 
$k'/k$ is finitely generated, by a simpler version of the above argument.  }
\end{remark}

%%%%%%%%%%%%%%%%%%%%%%%%%%%%%%%

\section{Shimura varieties and representability}
\label{ShimuraSect}

\subsection{Integral models of Shimura varieties of Hodge type}
\label{ss:integral models}

Here, we recall  results of \cite{KisinJAMS} about integral models of Shimura varieties of Hodge type. 
We actually follow the set-up of \cite[(1.3)]{KisinLR},  to which  the reader is referred 
for more details.

\subsubsection{}\label{sss:Shimura set up} 

Let $(G, \mathcal{H})$ be a  Hodge type Shimura datum in the sense of \cite{DeligneCorvallis},
so that $G$ is a connected reductive group over $\Q$ and  $\mathcal{H}=\{h\}$ is the $G(\R)$-conjugacy class of a 
Deligne cocharacter $h: {\rm Res}_{\C/\R}\Gm\to G_{\R}$.   Define  $\mu_h: \Gm_\C\to G_\C$,
as usual,   by $\mu_h(z)=h_{\C}(z, 1)$.  
The reflex field $E\subset \bar\Q\subset \C$ is the field of definition of the conjugacy class $\{\mu_h\}$.

The condition of Hodge type  means that there is an   algebraic group embedding 
\begin{equation}\label{generic global hodge}
\iota: G\hookrightarrow \GSp_{2g}
\end{equation}
over $\Q$  inducing a morphism of Shimura data $(G, \mathcal{H}) \to ({\rm GSp}_{2g}, {\mathcal H}_{2g} )$.  
Here $\mathcal{H}_{2g}$ is the union of the 
usual Siegel upper and lower half-spaces of genus $g$.
The composition $ \iota \circ \mu_h $ is conjugate to the standard minuscule
cocharacter $\mu_{\rm std}: \Gm\to \GSp_{2g}$ given by $\mu_{\rm std}(a)={\rm diag}(a^{(g)}, 1^{(g)})$. 
As $G(\R)$ contains the image of the weight  homomorphism $w_h$,
we see that $G$ has to contain the torus of scalars (diagonal matrices) of $\GSp_{2g}$.

We assume that $G$ extends to a connected reductive group over $\Z_{(p)}$, which we again denote by $G$.
As in \cite[Lemma (2.3.1)]{KisinJAMS}, this implies that there is a rank $2g$ symplectic space $(C,\psi)$ over $\Z_{(p)}$
and a closed immersion 
$
\iota : G \hookrightarrow \GL(C)
$
of reductive groups over $\Z_{(p)}$ whose generic fiber factors through the subgroup $\GSp(C_\Q,\psi) \subset \GL(C_\Q)$ and induces   (\ref{generic global hodge}) after fixing an identification $\GSp(C_\Q,\psi) = \GSp_{2g}$.
By Zarhin's trick, after replacing $C$ by ${\rm Hom}_{\Z_{(p)}} (C , C )^{\oplus 4}$ and enlarging $g$,
we may  assume that $C$ is self-dual with respect to $\psi$.  We  then have 
a closed immersion  of reductive group schemes
\begin{equation}\label{global hodge}
\iota : G  \hookrightarrow \GSp(C  , \psi)
\end{equation}
over $\Z_{(p)}$ with generic fiber (\ref{generic global hodge}).

As in (\ref{contra}) let $D$ be the $G$-representation contragredient to $C$.
By [\emph{loc.~cit.}]  there is finite list  $(s_\alpha)$ of tensors $s_\alpha\in C^\otimes=D^\otimes$ that cut out $G \subset \GL(C)$, in the sense that 
\[
G  (R)=\big\{g\in \GL(C_R) : g\cdot (s_\alpha\otimes 1)=(s_\alpha\otimes 1),\ \forall \alpha   \big\}.
\]
 for all $\Z_{(p)}$-algebras $R$.   In particular, $G(W)=G(K)\cap \GL( C_W)$. 
 In what follows, we take the set of tensors $(s_\alpha)$ to always  include the tensor 
corresponding  to the perfect symplectic  form $\psi$; see the proof of Theorem \ref{mainKim} below.

A choice of field embedding $\bar\Q\hookrightarrow \bar K$ determines a place $v \mid p$ 
of $E$.  The completion $E_v$  is the field of definition of  $\{\mu_h\}$, 
  now regarded as a $G(\bar{K})$-conjugacy class of  cocharacters    $\Gm_{\bar K} \to G_{\bar K}$.

Let  
\[
{\mathrm {Sh}}_{U_p}(G, \mathcal{H}) =\varprojlim\nolimits_{U^p} {\mathrm {Sh}}_{U^pU_p}(G, \mathcal{H})
\]
be the canonical model over $E\subset K$ of the corresponding Shimura variety 
for the hyperspecial subgroup 
\[
U_p=G(\Z_p) \subset G(\Q_p).
\]  
Here the limit is over compact open subgroups $U^p$
of  $G({\mathbb A}^p_f)$.

\subsubsection{} 
For a $\Z_{(p)}$-scheme $S$ and an abelian scheme $A \to S$ we set
\[
 \mathrm{Ta}^p(A) =\varprojlim\nolimits_{p \nmid n} A[n],
 \]
viewed as an \'etale local system on $S$,  and  write  
\[
 \mathrm{Ta}^p(A)_\Q=  \mathrm{Ta}^p(A)  \otimes_{\Z}\Q .
\] 

 Consider the category obtained from the category of abelian schemes over $S$  by tensoring the Hom groups by $ \Z_{(p)}$. 
 An object in this category will be called an  \emph{abelian scheme over $S$ up to prime to $p$-isogeny}. 
 An isomorphism in this category will be called a  \emph{$p'$-quasi-isogeny}. Note that $\mathrm{Ta}^p(A)_\Q$ is functorial for $p'$-quasi-isogenies.

If  $A$ is an abelian scheme up to prime to $p$-isogeny, we write $A^\vee$ for the dual  abelian scheme up to prime to $p$-isogeny. 
A {\sl  weak polarization} of $A$ is an equivalence class of $p'$-quasi-isogenies $\lambda: A\iso  A^\vee$ 
 such that some $\Z_{(p)}^\times$-multiple of $\lambda$ is a polarization.   Here two such 
 $\lambda$ are equivalent if they differ by multiplication by an element of $\Z_{(p)}^\times$.

Let  $U'^p \subset \GSp_{2g}(\A_f^p)$ be any compact open subgroup, let 
\[
U_p' = \GSp(C, \psi)(\Z_p) \subset \GSp_{2g} (\Q_p)
\]
be the hyperspecial subgroup determined by the self-dual symplectic space $(C,\psi)$ over $\Z_{(p)}$, and set
\[
U' = U'^p U'_p \subset \GSp_{2g}(\A_f).
\]
 
 Assume $(A, \lambda)$ is   an abelian scheme up to  prime to $p$-isogeny with a weak polarization.
  A \emph{$U'^p$-level structure} on $(A, \lambda)$  is a global section  
 \[
 \epsilon^p_{U'}  \in \Gamma\big(S, \underline {\rm Isom} \big(C \otimes \A^p_f , \mathrm{Ta}^p(A)_\Q \big)/U'^p \big).
 \]
Here  $\underline {\rm Isom}\big(C \otimes \A^p_f , \mathrm{Ta}^p(A)_\Q\big)/U'^p$
 is the \'etale sheaf on $S$ of  $U'^p$-orbits of isomorphisms 
 \[
 C\otimes \A^p_f \iso  \mathrm{Ta}^p(A)_\Q
 \]
identifying  the symplectic pairings induced by   $\psi$ and $\lambda$,  up to a $(\A_f^p)^\times$-scalar.

For $U'^p$ sufficiently small, the functor that assigns to $S$ the set of isomorphism  
classes of triples  $(A,\lambda, \epsilon^p_{U'})$ as above, is representable by a smooth $\Z_{(p)}$-scheme
${\mathscr A}_{g, U'}$,  whose generic fiber
\begin{equation}\label{siegel generic}
{\mathscr A}_{g, U'}  \otimes_{\Z_{(p)} } \Q \iso {\rm Sh}_{U' }({\rm GSp}_{2g}, {\mathcal H}_{2g} ) 
\end{equation}
 is  identified   with the Siegel Shimura variety; see \cite{KottJAMS}.  We will always assume that  $U'^p$ is sufficiently small in what follows.

\subsubsection{}

Let us denote by $\sS_{U_p}(G, \mathcal{H})$ the canonical smooth model over $\co_{E,(v)}$
of the Shimura variety 
\[
{\mathrm {Sh}}_{U_p}(G, \mathcal{H}) =\varprojlim_{U^p} {\mathrm {Sh}}_{U^pU_p}( G , \mathcal{H})
\]
  constructed in \cite{KisinJAMS} for the hyperspecial subgroup $U_p=G(\Z_p)$.   Thus
\[
{\sS}_{U_p}(G, \mathcal{H})= \varprojlim_{U^p} {\sS}_{U^pU_p}(G, \mathcal{H}),
\]
where for  sufficiently small subgroups $U^p_1\subset U^p_2\subset G({\mathbb A}_f^p)$ the transition morphism
\[
{\sS}_{U^p_1U_p} (  G , \mathcal{H})\to {\sS}_{U^p_2U_p}( G ,  \mathcal{H})
\]
is finite \'etale. In fact, for $U=U_pU^p$ with $U^p$ sufficiently small, the integral model  $\sS_U(G, \mathcal{H})$ is smooth 
over $\co_{E,(v)}$,    and is constructed as the normalization of the Zariski closure 
$\sS_U(G, \mathcal{H})^-$ of  the image of the morphism
\[
\Sh_U(G, \mathcal{H})   \to  \Sh_{U'}({\rm GSp}_{2g}, {\mathcal H}_{2g} ) \otimes_\Q E \to    {\mathscr A}_{g, U'}\otimes_{\Z_{(p)} }\co_{E,(v)}
\]
induced by (\ref{global hodge}) and (\ref{siegel generic}) for a suitable choice of level structure $U'^p$.
  In particular, there are finite morphisms
\begin{equation*}
%\label{iotaMap}
\iota: \sS_U(G, \mathcal{H}) \map{\mathrm{normalization} }  \sS_U (G, \mathcal{H})^- \to {\mathscr A}_{g}\otimes_{\Z_{(p)}}\co_{E,(v)},
\end{equation*}
where now we suppress the level structure $U'$  on the Siegel space from the notation.  
This should not lead to confusion, as the particular choice of level $U'$ will play little role in our arguments.

\subsubsection{} \label{sss:global to local datum}

Now let us pick a point $x_0\in \sS_U(G, \mathcal{H})(k)$. 
Consider the contravariant Dieudonn\'e module $\Db(X_0)(W)$ of the $p$-divisible
group $X_0=A_{x_0}[p^\infty]$ of the abelian scheme $A_{x_0}$ over $k$
determined by the point  $\iota(x_0)\in {\mathscr A}_g(k)$. By \cite[Cor.~(1.4.3)]{KisinJAMS}, there are 
 crystalline tensors    
 \[
 t_{\alpha, 0} =t^{\rm cr}_{\alpha, 0}\in \Db(X_0)(W)^\otimes
 \]
that are fixed by the action of Frobenius on $\Db(X_0)(W)[1/p]$,  and satisfy 
\[
t_{\alpha, 0}(k)\in {\rm Fil}^0(\Db(X_0)^\otimes(k)).
\]

By [\emph{loc.~cit.}], there is an isomorphism of $W$-modules
 \begin{equation*}
 %\label{keylemma}
 \beta_0:  D\otimes_{\Z_p}W\xrightarrow{\sim }\Db(X_0)(W)
 \end{equation*}
 identifying  $s_\alpha\otimes 1$ with  $t_{\alpha, 0}$.    After  choosing such an isomorphism,  
 the Frobenius on $\Db(X_0)(W)$ has the form  $F=b_{x_0}\circ \sigma$ for some $b_{x_0}\in G(K)$.
 
Consider the Hodge filtration \[{\rm Fil}^1(X_0) \subset \Db(X_0)(k)\cong {\rm H}^1_{\rm dR}(A_{x_0}/k).\] By \cite[Cor. (1.4.3) (4)]{KisinJAMS} this filtration is given by a $G_k$-valued cocharacter, and we pick any lift  to a cocharacter 
\begin{equation*}
\mu_{x_0}: \Gm_W\to G_W.
\end{equation*}
By the argument in the proof of \cite[Lemma (1.1.9)]{KisinLR} the $G(W)$-conjugacy class of
$\mu_{x_0}$ is independent of the choice of $\beta_0$.
In fact, any such cocharacter satisfies  
\[
b_{x_0}\in G(W) \mu_{x_0}^\sigma(p) G(W),
\]
and  lies in the   $G(\bar{K})$-conjugacy class defined by $ \mu^{-1}_h $  (see \cite[(1.1.9)]{KisinLR}), whose
local reflex field is   $E_v \subset K$.

In this way, each point  $x_0 \in    \sS_{U_p}(G, \mathcal{H})(k)$ produces a local Shimura-Hodge datum, which, for ease of notation, we abbreviate
to 
\begin{equation}\label{localtoglobaleq}
 (G,b,\mu,C) = (G_{\Z_p}, b_{x_0}, \mu_{x_0},  C_{\Z_p} ).
 \end{equation}
 Note that the $p$-divisible group of Lemma \ref{lemma121} is  $X_0=A_{x_0}[p^\infty]$.

For the rest of \S \ref{ShimuraSect}, we fix the  local Shimura-Hodge datum  (\ref{localtoglobaleq}) given
by a point $x_0  \in    \sS_{U_p}(G, \mathcal{H})(k)$ by the above procedure. 
We also fix the tensors  
$
s_\alpha\in C^\otimes =D^\otimes
$
cutting out the subgroup $G\subset \GL(C)$.  It is essential in what follows that both the local Shimura-Hodge datum $(G,b,\mu,C)$ and the tensors
$(s_\alpha)$ arise from a global point $x_0  \in    \sS_{U_p}(G, \mathcal{H})(k)$, and from tensors defined on $C$, not just on $C_{\Z_p}$.

\subsubsection{}
  \label{universaltensors} 
As the Shimura datum $(G,\mathcal{H}$) will remain fixed, in what follows we will often abbreviate $\sS_U$ or just $\sS$ instead of $\sS_{U}(G, \mathcal{H})$.
We will also often write $\sS_{U_p}$ instead of  $\sS_{U_p}(G, \mathcal{H})$.

Let $f: A\to\sS$ be the universal abelian scheme over $\sS$, and  denote by
\[
\widehat f: \widehat A\to \widehat \sS
\] 
the corresponding morphism of smooth formal schemes over $\Z_p$ obtained by $p$-adic completion.

We have the  crystals $\Db(X^{\rm univ})$ and $\Db(X^{\rm univ})^\otimes$, and Frobenius isocrystals $\Db(X^{\rm univ})[1/p]$ and $\Db(X^{\rm univ})^\otimes[1/p]$,  over $\widehat\sS$, where $X^{\rm univ}=\widehat A[p^{\infty}]$
is the $p$-divisible group of the universal abelian scheme. 
By \cite{BerthelotBreenMessingII}, we have
\[
\Db(X^{\rm univ})= {\rm R}^1\widehat{f}_{{\rm cris}, *}\co_{\widehat{A}/\Z_p},
\]
and there is a natural isomorphism 
\[
{\rm H}^1_{\rm dR}(\widehat A/\widehat \sS)\iso  ({\rm R}^1\widehat{f}_{{\rm cris}, *}\co_{\widehat{A}/\Z_p})_{\widehat\sS}
\]
of coherent sheaves $\co_{\widehat{\sS}}$-modules,    where the left hand side is the first relative de Rham cohomology of $\widehat{A}$, 
and  the right hand side is  the pullback of the crystal to the Zariski site. 

By \cite[Cor. (2.3.9)]{KisinJAMS}, there are de Rham tensors 
\[
t^{\rm univ}_{\alpha, \rm{dR}}: {\bf 1}\to {\rm H}^1_{\rm dR}(\widehat A/\widehat {\sS})^\otimes,
\]
which are horizontal for the Gauss-Manin connection.
Since $\sS$ is smooth over $\Z_p$, 
we obtain $\co_{\widehat\sS/\Z_p}$-morphisms 
\[
t^{\rm univ}_{\alpha}: {\bf 1}\to  ({\rm R}^1\widehat{f}_{{\rm cris}, *}\co_{\widehat{A}/\Z_p})^{\otimes}=\Db(X^{\rm univ})^\otimes
\]
of crystals over $\widehat\sS$.
By the construction in \cite{KisinJAMS}, the tensors $t^{\rm univ}_{\alpha}$ restrict to $t_{\alpha, 0}$ by pulling back via
$x_0:\Spec(k)\to \widehat\sS$, and
\[
t^{\rm univ}_{\alpha}[1/p]: {\bf 1}[1/p]\to  \Db(X^{\rm univ})^\otimes[1/p]
\] 
are Frobenius equivariant.

\subsubsection{}\label{Faltings} 
Recall from \S \ref{general} the homogeneous space 
\[
M_{G, \mu }\otimes_{\co_{E,v}}W\iso  G_{W}/P_{\mu } 
\]
 over $W$,
where $P_{\mu}\subset G_{W}$ is the parabolic subgroup defined by $\mu:\Gm_W\to G_W $.
Denote by $U^\mu=U^\mu_G $ the unipotent radical of the opposite to $P_\mu$ parabolic subgroup of $G_{W}$ and by $U^{\mu,\wedge}_G$ the formal completion of $U^\mu_G  $
at its identity section over $W$.  Then $U^{\mu,\wedge}_G$ can be identified with
the formal completion of $M_{G, \mu}\otimes_{\co_E}W$ at the section given by $1\cdot P_\mu /P_\mu$.

Let 
$R=\widehat\co_{\sS_W, x_0} $ be the completion
of   the local ring of $\sS_W$ at $x_0$. By  \cite[Prop. (2.3.5)]{KisinJAMS} and its proof,  we may identify
${\rm Spf}(R )$  with $U^{\mu_{x_0}, \wedge}_{G}$ and identify  $R= W\lps x_1,\ldots , x_d\rps$ in such a way that the $W$-point  
given by $x_1=\cdots =x_d=0$ has Hodge filtration given by $\mu_{x_0}$.  
  In fact, by a result of Faltings \cite{FaltingsJAMS}, we see, as in \cite[(1.5)]{KisinJAMS} (see  also \cite{Moonen} \S 4), that this identification can be chosen
in such a way that the  Dieudonn\'e crystal  of the universal $p$-divisible group
$X^{\rm univ}$ with its tensors $t_{\alpha}$ over $R $ is 
\[
\Db(X^{\rm univ})(R)=\Db(X_0)(W)\otimes_WR = D\otimes_{\Z_p}R , 
\]
(as $R$-modules) with filtration
\[
 {\rm Fil}^1( X^{\rm univ} ) ={\rm Fil}^1( X_0)  \otimes_{W}R , 
\]
and $t_{\alpha}=t_{\alpha, 0}\otimes 1=s_{\alpha}\otimes 1$, while the semi-linear Frobenius 
\[
F: \Db(X^{\rm univ})(R)\to \Db(X^{\rm univ})(R)
\]
 is given by $F= u\cdot (b_{x_0}\otimes \phi_R)$.  Here $u$ is the  universal  $R $-point of the
completion $U^{\mu_{x_0}, \wedge}_{G}$ of the unipotent subgroup $U^{\mu_{x_0}}_G$ at the identity section and $\phi_{R}$ is the lift of Frobenius 
such that $\phi_R(x_i)=x^p_i$.

\subsection{A global construction of Rapoport-Zink formal schemes}
 \label{newproofpar}

A more general version of the following representability result
appears in work of W. Kim \cite{KimRZ}. Kim does not assume that the local Shimura 
datum 
  $(G, [b], \{\mu\})$ of Hodge type are obtained from a point on a Shimura variety.   Although we were very much inspired by Kim's work, our arguments are quite different 
and independent of \cite{KimRZ}.

\begin{theorem}\label{mainKim} 
Let the unramified local Shimura-Hodge datum $(G, b, \mu, C)$  and  
the finite set of tensors $(s_\alpha)\in C^\otimes$   that cut out $G\subset \GL( C)$ be  as in \ref{localtoglobaleq} above.
Suppose $\xo=\xo(G,b,\mu,  C)$ is the $p$-divisible group
over $k$ defined in Lemma \ref{lemma121}. 
\begin{enumerate}
\item
 There exists a formal scheme   $\RZc_G=\RZc_{ G, b, \mu,  C, (s_\alpha) }$ over $\Spf(W)$,  formally smooth and formally
locally of finite type, that  represents the functor $\RZfsm_G$ on ${\rm ANilp}^{\rm fsm}_W$
defined in \ref{RZGfsm}. 

\item
 The action of $J_b(\Q_p)$ on $\RZfsm_G$ given by (\ref{Jaction}) is  induced 
by a left $J_b(\Q_p)$-action on the formal scheme $\RZc_{G}$.

\item
 The formal scheme $\RZc_G$ is a closed formal subscheme of  $\RZc(\xo)$.
\end{enumerate}
\end{theorem}

\begin{remark}
The formal scheme  $\RZc_G$ is characterized as the unique formally smooth and locally formally 
of finite type over $W$ that represents the functor $\RZfsm_G$, and this functor is determined by $(X_0 , (t_{\alpha,0}))$.
(The definition of the functor also involves the conjugacy class $\{\mu\}$, but this is determined by $X_0$ and its tensors.)
In particular, it follows that $\RZc_G$ agrees, up to isomorphism, with  the Rapoport-Zink formal scheme of Kim  \cite{KimRZ}. 
\end{remark}

\begin{corollary}\label{cor:points}
Under the above assumptions, if $k'/k$ is finitely generated field extension, the 
construction in the proof of Theorem \ref{bijectionRADL} provides a  bijection
\begin{equation*}
\pi: \RZc_G(k')\iso  X_{G, b, \mu^\sigma, \sigma}(k').
\end{equation*}
\end{corollary}

\begin{proof}
 By Theorem \ref{mainKim} (i),  the underlying reduced scheme  $\RZc_G^{\rm red}$ is locally of finite type over $k$,  and so
\[
\RZc_G(k')=\varinjlim\nolimits_R \RZc_G(R)=\varinjlim\nolimits_R\RZnilp_G(R),
\]
where  the limit is as in Proposition \ref{fgfield}. The corollary 
then follows from Theorem \ref{bijectionRADL} after taking into account
Remark \ref{remark:ftpoints}.
\end{proof}

\subsubsection{}
 We will now turn to the proof of  Theorem \ref{mainKim} but, before we do that, we sketch the argument. We first construct the formal scheme
 $\RZc_G$
 over $W$ using the integral model of a Shimura variety and then we show, using also certain results from \cite{KisinLR}, that it represents the functor $\RZfsm_G$.

 Roughly, the construction of $\RZc_G$ is done in two steps: First, we form a fiber product $\RZc^\diamond_G$ of the classical Rapoport-Zink space associated to the $p$-divisible group $X_0$ with the $p$-adic completion of the integral model of the Shimura variety; the fiber product is over a Siegel moduli space. Over this fiber product, we have  crystalline tensors obtained by pulling back the universal crystalline tensors $t_\alpha^{\rm univ}$ over the integral model constructed by Kisin (see \S\ref{universaltensors}); we also have corresponding crystalline tensors $t_{\alpha, 0}$ obtained by pulling back the crystalline tensors on the base point $X_0$ via the universal quasi-isogeny. The second step is to show that the locus where these two  tensors agree, \emph{i.e.} with $t_\alpha^{\rm univ}=t_{\alpha, 0}$, is given by a closed and open formal subscheme of the fiber product. This formal subscheme is the desired formal scheme $\RZc_G$.

\begin{proof}
It follows directly from the definition that the usual Rapoport-Zink formal scheme $\RZc(\xo)$  of  \S\ref{sss:RZ} represents 
\[
\RZfsm_{\GL( C)} \define \RZfsm_{ \GL( C), b, \mu,  C ,\emptyset }.
\] 
Here  we take  the set of tensors  to be empty. 
 Also, we can show that the symplectic Rapoport-Zink formal scheme
 $\RZc(\xo, \lambda_0)$ represents 
 \[
 \RZfsm_{\GSp( C)} \define  \RZfsm_{ \GSp( C, \psi), b, \mu,  C , (s_{\rm sympl})}.
 \]
 Here, $\lambda_0: X_0\xrightarrow{\sim} X_0^\vee$ is the principal polarization deduced from the symplectic form $\psi$ on 
 $C$, and the single tensor $s_{\rm sympl}$ is defined as follows: Denote by $\eta: \GSp( C,\psi) \to \Gm$ the similitude character. The 
 symplectic pairing 
$
\psi:  C\otimes C\xrightarrow{\ } \Z_p(\eta)
$ 
defining $\GSp( C,\psi)$ induces an isomorphism $C\iso   D(\eta)$, which  allows us to view the dual 
\[
\psi^*: \Z_p(\eta^{-1})\to ( C\otimes C)^*\cong  D\otimes D
\]
as a map $\psi^*: \Z_p(\eta)\to  C\otimes C$. Now define  
\[
s_{\rm sympl}\in {\rm End}( C \otimes C )\cong  C^{\otimes 2}\otimes D^{\otimes 2}\subset C^\otimes
\]
 as  the composition 
\[
 C\otimes C\xrightarrow{\psi} \Z_p(\eta)\xrightarrow{\psi^*}  C\otimes C.
\]
The group $\GSp( C, \psi)\subset \GL( C)$
is the  stabilizer of $s_{\rm sympl}$. One can easily shows that 
$\RZc(\xo, \lambda_0)$ represents $\RZfsm_{\GSp( C)}$
by using duality and the full-faithfulness of the 
Dieudonn\'e crystal functor (see \cite[4.1 and 4.3]{BerthelotMessingIII} and \cite{deJongBTIHES}) 
for $p$-divisible groups over bases in ${\rm ANilp}^{\rm fsm}_W$.

In accordance to the above general notation scheme,  we can now denote
\[
\RZc_{\GL( C)}  = \RZc(X_0), \quad  \RZc_{\GSp( C)} = \RZc(X_0,\lambda_0).
\]
If $ C$ is clear from the context we will simply write $\RZc_{\GL}$, $\RZc_{\GSp}$, instead.

As in \cite[Theorem 6.21]{RapZinkBook}, there is a canonical  morphism 
\[
\Theta: \RZc_{\GSp( C) }\xrightarrow{ \ } \widehat{\mathscr{A}}_{g, W}\define\widehat{\mathscr A}_g\widehat{\otimes}_{\Z_p}W.
\]
(Here again, $\widehat{\mathscr A}_g$ denotes the completion of ${\mathscr A}_g\to \Spec(\Z_{(p)})$ along its special fiber.)
This morphism sends a triple $(X, \lambda, \rho)$ to the point corresponding to the unique principally polarized abelian scheme $A$ whose $p$-divisible group is
$X$, and  for which there exists  a quasi-isogeny $A_{x_0} \dasharrow A$  respecting polarizations up to $\Q^\times$-scaling, inducing an isomorphism on $\ell$-divisible groups for all $\ell\neq p$, and inducing the quasi-isogeny 
\[
A_{x_0}[p^\infty] =X_0 \overset{\rho}\dasharrow X=A[p^\infty] .
\]

We now let $\RZc^\dia_G$ be the formal scheme over $\Spf(\bZp)$ defined  by the   fiber product
\begin{equation*}
%\label{fiber}
\xymatrix{
{  \RZc^\dia_G  }  \ar[r] \ar[d]_{\Theta^\dia_G   }  & {  \RZc_{\GSp( C) }  }   \ar[d]^{\Theta}   \\
  {  \widehat\sS_\bZp }    \ar[r]^{\iota}   & \widehat{\mathscr A}_{g,W.}
}
\end{equation*}
Here, $\sS=\sS_{U^pU_p}$, in which we fix a choice of a sufficiently small prime-to-$p$ level $U^p$. We will see eventually that all such choices produce the same Rapoport-Zink formal scheme. 

The formal scheme $\RZc^\dia_G$ represents the functor that assigns to each $W$-scheme $S$ in ${\rm Nilp}_W$
 the set of quadruples 
$
(X, \lambda, \rho, f: S\rightarrow  \sS_W )
$
where $(X, \lambda, \rho)$ is  an $S$-valued point of $\RZc_{\GSp }$  such that the  composition
\[
S\xrightarrow{f} \sS_{\bZp}\xrightarrow{\rm norm}  \sS_{\bZp}^-\hookrightarrow  {\mathscr A}_{g, W}.
\]
gives the corresponding point $\Theta((X, \lambda, \rho))\in  {\mathscr A}_{g, W}(S)$.
 
In this situation, the Dieudonn\'e crystal ${\mathbb D}(X)$ of the $p$-divisible group over $S$ supports 
  tensors $t_{\alpha}:{\bf 1}\to  {\mathbb D}(X)^{\otimes}$ which are obtained as 
$
t_{\alpha}=f^*(t_{\alpha}^{\rm univ}),
$ 
 \emph{i.e.}~by pulling back the universal crystalline tensors over $\widehat\sS$.

\begin{proposition}\label{finite}
The morphism $\RZc^\dia_G\to \RZc_{\GSp( C)}$ is finite and the formal scheme $\RZc^\dia_G$ is formally smooth
and locally formally of finite type over ${\rm Spf}(W)$.
\end{proposition}

\begin{proof}
Since $\iota$ is finite, the same is true for $\RZc^\dia_G\to \RZc_{\GSp( C)}$,  and so $\RZc^\dia_G$ is locally formally of finite type over $W$.
 By \cite[Prop. (2.3.5)]{KisinJAMS}  and the proof of \cite[Theorem (2.3.8)]{KisinJAMS}  (see also \S\ref{Faltings}), the 
morphism $\iota$ induces a closed immersion between the formal completions of
  $\sS_W$ and $\mathscr{A}_{g, \bZp}$ at  each closed point
$s \in \sS_W$. Moreover, the scheme $\sS_W$ is smooth. By the Serre-Tate theorem, 
the formal completion of $\RZc_{\GSp( C)}$ at any closed point can be identified
with the formal completion of ${\mathscr A}_{g,\bZp}$ at 
the corresponding point and formal smoothness follows. In fact, for each closed point $s$ 
of ${\RZc}^\dia_{G}$ the morphism $\Theta^\dia_G$
 gives an isomorphism 
 \[
 {\widehat\RZc}^\dia_{G, s}\iso  \widehat{\sS}_{\bZp, s}
 \]
 between  formal completions.
\end{proof}

  Suppose that $S$ is a scheme in ${\rm Nilp}_W$ and $a\in \RZc^\dia_G(S)$. 
  We have crystals $\Db(X)$ and $\Db(X)^\otimes$, and isocrystals $\Db(X)[1/p]$ and $\Db(X)^\otimes[1/p]$,  over $S$. 
  By pulling back  the universal 
crystalline tensors $t^{\rm univ}_\alpha$   via $\Theta^\dia\circ a: S\to  \sS$, we obtain $t_{\alpha }: {\bf 1}\to \Db(X)^\otimes$ over $S$.

\begin{definition}
 For each $S$ in ${\rm Nilp}_W$,  denote by  
 $
 \RZc_G(S) \subset \RZc_G^\dia(S)
 $ 
 the subset consisting  of all points $a \in \RZc_G^\dia(S)$ with  the following property:
For every field  extension  $k'/k$ and every point $y\in S(k')$, 
the  isomorphism 
\[
\Db(\rho):  \Db(X_y)^{\otimes}(W')[1/p] \iso  \Db(X_0)^{\otimes}[1/p]\otimes_{W[1/p]} W' [1/p] 
\]
 identifies the tensors  $t_{\alpha}(W')$  with $t_{\alpha,  0}(W)\otimes 1$.
In other words
\begin{equation}\label{ident}
\Db(\rho)(t_{\alpha }(W'))=t_{\alpha, 0}(W)\otimes 1.
\end{equation}
(Recall that $W'$ is the Cohen ring of $k'$.)  
\end{definition}

\begin{proposition}\
%\label{subfunctor}
\begin{enumerate}
\item
The subfunctor $\RZc_G$ is represented by a closed and open formal subscheme
of $\RZc^\dia_G$. 

\item
 The formal scheme $\RZc_G$ is  formally smooth and 
locally formally of finite type over ${\rm Spf}(\bZp)$. 
\end{enumerate}
\end{proposition}

\begin{proof}  The proof relies on the following lemma.
 
\begin{lemma}\label{convIsocr}
Assume that $S$ in ${\rm Nilp}_W$ is connected, $a\in  \RZc^\dia_G(S)$,  and that there is a field $k'$ and a point $y\in S(k')$ such that the condition (\ref{ident}) is satisfied at $y$.  Then 
$
\Db(\rho)(t_{\alpha})=t_{\alpha, 0}\otimes 1,
$
where 
\[
\Db(\rho): \Db(X_{\bar S})^{\otimes}[1/p] \iso   \Db(X_0\times_k \bar S)^{\otimes}[1/p] 
\]
is the morphism of Frobenius isocrystals over $\bar S$ induced by the quasi-isogeny $\rho$. In particular,
 condition  (\ref{ident}) is satisfied at  \emph{all} field valued points of $S$,
 and so $a \in \RZc_G(S)$.
\end{lemma}

\begin{proof}
Note that, as  $\RZc^\dia_G$ is  locally formally of finite type, for any $a\in \RZc^\dia_G(S)$,
there exists a locally finite type scheme $S'$ in ${\rm Nilp}_W$, a morphism $\omega: S\to S'$, 
and $b\in \RZc^\dia_G(S')$ such that $a=b\circ \omega$. Since $S$ is connected,
we can assume that $S'$ is connected and so we reduce to showing the
statement above for $S$ connected and locally of finite type. 

We will first show  that (\ref{ident}) holds for all field valued points of $S$. 
All such points  factor through the underlying reduced scheme $S_{\rm red}$; hence,
we can further assume that $S$ is reduced and is actually affine of finite type over $k$.  
Now the argument of \cite[Lemma 5.10]{MadapusiSpin}  implies that 
condition (\ref{ident}) is satisfied for all field-valued points of $S$,
and in particular for all closed points $s\in S$, taking  $k'=k(s)$ to be the residue field. This already gives that $a$ is in $\RZc_G(S)$.

To show the rest, observe that,
by Berthelot's construction \cite[Theorem (2.4.2)]{Berthelot}, the Frobenius crystal $\Db(X)$ over $S$ determines a convergent 
Frobenius isocrystal $M=\Db(X)[1/p]^{\rm an}$ over $S/W$.
Similarly, we have a convergent 
Frobenius isocrystal $M_0$ given by base-changing $\Db(X_0)[1/p]^{\rm an}$ to a
convergent 
Frobenius isocrystal over
 $S/W$.
The quasi-isogeny $\rho$ induces a morphism of convergent 
Frobenius isocrystals 
\[
\Db(\rho)^{\rm an}:  M^{\otimes}\iso  M^{\otimes}_0.
\]
By \cite[Theorem 4.1]{OgusIso}  (see \cite[Remark 2.3.4]{Berthelot}), we have $\Db(\rho)^{\rm an}(t_{\alpha})=t_{\alpha, 0}\otimes 1$ in $M^{\otimes}$
since, by the above, this is true at a closed point.  By  \cite[Theorem (2.4.2)]{Berthelot}  the functor from Frobenius  crystals up to isogeny
over $S$ to convergent Frobenius isocrystals over $S/W$ is faithful, and 
the result follows.
\end{proof}

Consider the union of connected components of $\RZc^\dia_G$ which have 
a field valued point such that (\ref{ident}) is satisfied. We can now see that this 
union represents the functor $\RZc_G$. The second statement now follows from 
the first and the previous proposition.  
\end{proof}

The following gives a main part of Theorem \ref{mainKim}.

\begin{proposition}\label{represent}
The  formal scheme $\RZc_G$ constructed above represents the functor 
\[\RZfsm_G: {\rm  ANilp}_W^{\rm fsm}\to {\rm Sets}
\]
defined in \S \ref{RZGfsm}.
\end{proposition}

\begin{proof} Suppose that $R$ is in ${\rm ANilp}_W^{\rm fsm}$. Denote by $I$ an ideal
of definition of $R$ with $pR\subset I$. We will establish a functorial bijection 
\[
\RZc_G(R)\iso \RZfsm_G(R).
\]

A) Suppose first that we are given   
\[
(X, \lambda, \rho, f: \Spf(R)\xrightarrow{} \widehat\sS_W ) \in \RZc_G(R)
\]
By Remark \ref{remark1} (c),  we  have morphisms of crystals
$t_{\alpha}: {\bf 1}\to \Db(X)^\otimes$ over $\Spec(R)$
obtained by pulling back the universal crystalline tensors on $\widehat\sS$.
Recall, the tensors $(t_\alpha)$ include the crystalline tensor 
that corresponds to the polarization $\lambda$.
By the definition of $\RZc_G$, for any $k'$-valued point of $R$,  where $k'$ is any field, 
we have the identity (\ref{ident}). Property (i) of Definition \ref{defRZG} follows from Lemma \ref{convIsocr}.
To show that 
properties  (ii) and (iii) are satisfied, we  reduce by fppf descent to
the case that $R$ is replaced by its  completion $\widehat R_x$  at an arbitrary closed $k$-valued point $x$.
Since $R$ is formally smooth, we can assume   $R=W/p^m\lps x_1,\ldots , x_n\rps$ 
for some $m$ and $n$.  
Then, the morphism 
\[
f: \Spf(W/p^m\lps x_1,\ldots , x_n\rps)\to  \widehat\sS_W
\] 
factors through the  completion $\Spf(\hat\co_{\sS, f(x)})$; 
this  completion is described in \S\ref{Faltings}, from which properties  (ii) and (iii) follow  for $t_\alpha$
and the Hodge filtration over $\widehat\co_{\sS, f(x)}$,  and therefore also over $R$.

B) Conversely, suppose 
\[
(X, \rho, (t_\alpha))\in \RZfsm_G(R).
\]
Since $(t_\alpha)$
include the  polarization tensor $t_{\rm sympl}$, it follows, by using duality and the full-faithfulness of the 
Dieudonn\'e crystal functor over $R$ (see \cite[4.1 and 4.3]{BerthelotMessingIII}  and also \cite{deJongBTIHES}) that these data also produce a principal polarization $\lambda$ on the $p$-divisible group $X$.
We thus obtain $\Theta( (X, \lambda, \rho ))$, a $\Spf(R)$-valued point of  ${\mathscr A}_{g,W}$
which, by the standard algebraization theorems, corresponds to 
 \begin{equation}\label{y point}
 y: \Spec(R)\to {\mathscr A}_{g,W}.
 \end{equation}
  It is enough to show that $y$ factors    through $\sS_{W}$. 

This true when $R=k$, but this  is already quite deep. It follows from Theorem \ref{bijectionRADL} (with $k'=k$),  together 
with \cite[Proposition (1.4.4)]{KisinLR} and its proof  (this uses the main result of \cite{CKV}).

Let us now deal with more general $R$.
Assume first  $R\iso  W\lps x_1,\ldots, x_n\rps$, for some $n\geq 0$,
with the obvious
extension of a notion of an $R$-valued point of $\RZfsm_G$
(take $I=(p, x_1,\ldots, x_n)$ as an ideal of definition in \S \ref{RZGfsm}). 
As just explained, we know that $\Spec(k)\to {\mathscr A}_W$ given 
$x_1=\cdots =x_n=p=0$
gives a point $x\in \sS_W(k)$.
Use property (ii) of Definition \ref{defRZG} to choose a trivialization 
\[
\Db(X)(R)\cong  D\otimes_{\Z_p}R
\]
that matches the tensors  $t_\alpha(R)$ with the standard tensors $s_\alpha\otimes 1$. 
By Faltings'  construction of the universal deformation of the  polarized
 $p$-divisible group $(X_x, \lambda_x)$ as in \S \ref{Faltings}, and the Serre-Tate theorem, we can identify the completed local ring of ${\mathscr A}_{g, W}$
with the completion $U^{\mu_{\rm std},\wedge}_{\rm GSp( C)}$ of the 
opposite unipotent   at the identity section.
Our conditions now imply that the tensors $t_\alpha=s_\alpha\otimes 1$ are ``Tate tensors'' over $R$,
therefore, by \cite[4.8]{Moonen}, the corresponding $R$-valued point of 
$U^{\mu_{\rm std},\wedge}_{\rm GSp( C)}$ 
factors through 
\[
U^{\mu_x,\wedge}_G\subset U^{\mu_{\rm std},\wedge}_{\rm GSp( C)}.
\]
As in (\ref{Faltings}), $U^{\mu_x,\wedge}_G$ is   identified with the completion of $\sS_W$ at $x$ and the result   for 
  $R\iso  W\lps x_1,\ldots, x_n\rps$  follows.

We now consider the case of a general $R$ in ${\rm ANilp}_W^{\rm fsm}$. 
Evaluate $t_\alpha$ on the PD lift $\widetilde R$ of  \S\ref{PDlift}  to  obtain $t_{\alpha}(\widetilde R)$,  and a 
corresponding $\widetilde R$-scheme $T_{\widetilde R}$ of trivializations of $\big(  \Db(X)(\widetilde R), (t_{\alpha}(\widetilde R)) \big)$.

\begin{lemma}  $T_{\widetilde R}$ is a $G_{\widetilde R}$-torsor.
\end{lemma}

\begin{proof}  Notice that $\widetilde R/p^n\widetilde R\to R$ is a nilpotent PD thickening, for all $n\geq 1$. 
 The claim   follows from the local criterion of flatness and 
the definition of $\RZfsm_G$   (see in particular \S\ref{RZGfsm} and condition (ii) of Definition \ref{defRZG}), by an argument as in the proof of Lemma \ref{Ctorsor}.
 \end{proof}

As in \cite[Prop. (1.1.5)]{KisinJAMS},   the scheme $M_{G,\mu}$ of $G$-split filtrations of type $\mu$ is smooth over $\co_E$.
 Hence, so is its twist 
 \begin{equation}\label{torsor twist}
 M_{G,\mu}^T=T_{\widetilde R}\times^{G_{\widetilde R}}M_{G, \mu}
 \end{equation}
  by the $G_{\widetilde R}$-torsor $T_{\widetilde R}$; 
this  classifies  filtrations in $\Db(X)(\widetilde R)$ which are, locally for the \'etale topology, induced by a cocharacter
which is $G$-conjugate to $\mu$.
Since $\widetilde R$ is $p$-adically complete, we can lift the $R$-valued point 
of (\ref{torsor twist})  corresponding  to the Hodge filtration  ${\rm Fil}^1( X ) \subset \Db(X)(R)$
to an $\widetilde R$-valued point  corresponding to a filtration in $\Db(X)(\widetilde R)$ as above. 
By  Grothendieck-Messing theory, this lift of the Hodge filtration gives   a morphism 
\begin{equation}\label{tilde point}
\widetilde y: \Spec(\widetilde R)\to {\mathscr A}_{g, W}
\end{equation}
extending the point $y$ of (\ref{y point}).   Now suppose that $x$ is a $k$-valued point of $\bar R=\widetilde R/p\widetilde R$
and consider 
\[
y^\wedge_x: \Spec({\widetilde R}^\wedge_x)\to \Spec(\widetilde R)\to {\mathscr A}_{g, W}.
\]
Since ${\widetilde R}^\wedge_x\iso  W\lps x_1,\ldots, x_n\rps$, we obtain, by the result above,
that $y^\wedge_x$ factors through through $\sS_W$.
It follows that (\ref{tilde point}) factors through the Zariski closure $\sS^-_W$ of the generic fiber of $\sS_W$ in 
${\mathscr A}_{g,W}$; since $\widetilde R$ is integrally closed in $\widetilde R[1/p]$, we see that
$\ti y$ factors through the normalization  $\sS_W$. Therefore, the morphism  
  $y: \Spec(R)\to {\mathscr A}_{g,W}$ also factors through $\sS_W$. This completes the proof of Proposition \ref{represent}.
\end{proof}

This completes the proof of (i) and (ii) of Theorem \ref{mainKim}. Indeed, the statement about the action
of $J_b(\Q_p)$ can be easily deduced from the rest.  
Part (iii)   follows from the fact that 
\[
 \RZc_{\GSp }=\RZc(X_0,\lambda_0)\hookrightarrow \RZc_{\GL }=\RZc(X_0)
 \]
  is a closed immersion,  together with:

\begin{proposition}\label{closedimm}
 The morphism $\RZc_G\to \RZc_{\GSp }$
obtained by composing $\RZc_G\hookrightarrow\RZc^\dia_G$ and $\RZc^\dia_G\to \RZc_{\GSp }$
is a closed immersion.
\end{proposition}

\begin{proof} Recall that by Proposition \ref{finite},  $\RZc_G\to \RZc_{\GSp }$ is finite.
By the previous proposition, we can identify $\RZfsm_G(k)=\RZc_G(k)$. Let $(X,\lambda, \rho)\in  \RZc_{\GSp }(k)$,
and  let   
\[
x=(X, \rho, (t_\alpha))\in \RZfsm_G(k)
\] 
be a preimage. 
The crystalline tensors $t_\alpha: {\bf 1}\to \Db(X)^\otimes$ are uniquely determined by
$t_\alpha(W)\in \Db(X)(W)^{\otimes}$. The condition  (\ref{ident}) shows that these are then
also uniquely determined by the rest of the data, so 
\[
\RZfsm_G(k)=\RZc_G(k)\to \RZc_{\GSp }(k)
\] 
is injective.  The proof of Proposition \ref{finite} now implies that $\RZc_G\to \RZc_{\GSp }$ induces 
a closed immersion $\widehat{\RZc}_{G,x}\hookrightarrow \widehat{\RZc}_{\GSp ,x}$ on formal completions.
Since $\RZc_G\to \RZc_{\GSp }$ is finite the result easily follows, for example by using Nakayama's
lemma.
\end{proof}

This completes the proof of Theorem \ref{mainKim}. \end{proof}

\begin{proposition}\label{prop:tensor independence} 
The formal scheme $\RZc_G$ depends only on the local Shimura-Hodge datum $(G, b, \mu,  C)$ and not on the  choice of the tensors $(s_\alpha)\subset C^{\otimes}$, as in \ref{sss:Shimura set up}, that cut out $G$. 
\end{proposition}

\begin{proof}
By Proposition \ref{closedimm}, $\RZc_G$ is a closed formal subscheme of 
the undecorated Rapoport-Zink formal scheme
$\RZc_{\GL }=\RZc(X_0)$. By Theorem \ref{bijectionRADL},
 the choice of the base point $x_0=(X_0,  (t_{\alpha, 0}))$, 
together with an isomorphism of its Dieudonne module with  $( D_W, b \circ \sigma , (s_\alpha\otimes 1))$, determines bijections
$$
\RZc_G(k)=\RZfsm_G(k)\iso  X_{G, b_0, \mu^\sigma_0}(k),\quad 
\RZc_{\GL }(k)\iso  X_{\GL , i(b_0),   i(\mu^\sigma_0)}(k).
$$
In fact, these bijections are compatible with the maps $\RZc_G\hookrightarrow \RZc_{\GL }$
and \[ X_{G, b_0, \mu^\sigma_0}(k)\hookrightarrow  X_{\GL , i(b_0),   i(\mu^\sigma_0)}(k)\]
determined  by $i: G\hookrightarrow \GL( C) $. Moreover, for each $x\in  \RZc_G(k)$,
the formal completions $\widehat\RZc_{G,x}\subset \widehat\RZc_{\GL , x}$ at $x$ 
can be identified with $U^{\mu_x,\wedge}_{G}
\subset U^{i(\mu_x), \wedge}_{\GL }$ where $\mu_x:\Gm_W\to G_W$ 
 gives a filtration that
 lifts the Hodge filtration
for $x$. Therefore,  both 
the set of $k$-valued points and the formal completions
at each point of the closed formal subscheme $\RZc_G\subset \RZc_{\GL } $   do not depend
on the choice of tensors $(s_\alpha)\subset C^{\otimes}$.   Hence, we deduce that 
 the closed formal subscheme $\RZc_G\subset \RZc_{\GL } $ also does not depend on 
 the choice of tensors $(s_\alpha)$. Combining with the above, this now implies that $\RZc_G$ depends only on the local Shimura-Hodge datum $(G, b, \mu,  C)$. 
 \end{proof}
 
  \begin{remark}
 According to \cite{KimRZ}, the Rapoport-Zink formal  scheme   only depends, up to isomorphism,
on the datum $(G, [b], \{\mu\})$. Then by the above, $\RZc_G$ also only 
depends, up to isomorphism, on $(G, [b], \{\mu\})$ and not on the local Hodge embedding. However,  this independence does not follow directly from our construction without appealing to [\emph{loc.~cit.}]. 
\end{remark}

\begin{remark}\label{comparemark_level}
Define the $W$-morphism
\[
\Theta: \RZc_G\to \widehat {\sS}_W=\widehat {\sS}_{U, W}
\]
to be
the composition of $\RZc_G\hookrightarrow\RZc^\dia_G$ with $\Theta^\dia_G: 
\RZc^\dia_G\to \widehat {\sS}_W$. The morphisms \[\Theta: \RZc_G\to  \widehat {\sS}_{  W}=\widehat {\sS}_{U, W}\]
commute with the projections $\sS_{U^p_1U_p, W}\to \sS_{U^p_2U_p, W}$ for $U^p_1\subset U^p_2$. 
Hence, they  combine to also provide a morphism
\[
\Theta: \RZc_G\to  \widehat\sS_{U_p, W}\define \varprojlim_{U^p}\widehat\sS_{U^pU_p, W}.
\]
\end{remark}

\begin{remark} \label{rem:point switch}
Recall from  \S\ref{sss:global to local datum}  that the local Shimura-Hodge datum $(G,b,\mu,C)$ was constructed  from a point $x_0\in \sS(k)$.
Fix $g\in G(K)$ and  $h\in G(W)$, set 
\[
b'=g^{-1}b\sigma(g),\quad \mu'=h\mu h^{-1},
\] 
and assume that 
$
b'\in G(W)\mu'^\sigma(p)G(W).
$
Then there is a point $x_0'\in \sS(k)$ such that the unramified local Shimura-Hodge datum $(G, b', \mu', C)$
is constructed, in the sense of  \S\ref{sss:global to local datum}, from $x_0'$.

To see this, notice that the above condition on $b'$ implies that  $g \in X_{G, b, \mu^\sigma}(k)$,
and we may take $x_0'$ to be the image of $g$ under the composition
\[
X_{G, b, \mu^\sigma}(k)\xrightarrow{\pi^{-1}} \RZc_G(k)\xrightarrow{\Theta} \sS(k).
\]
Note that we then obtain an isomorphism 
\[
\RZc_{G,b,\mu,C, (s_\alpha)}\xrightarrow{\sim} \RZc_{G,b',\mu', C, (s_\alpha)}
\]
 by composing the quasi-isogeny $\rho$ in the definition of the Rapoport-Zink functor \ref{defRZG} with
the quasi-isogeny 
\[
X_0(G, b',\mu', C)\dasharrow X_0(G, b,\mu, C)
\] 
determined  by $g$. 
\end{remark}

\subsection{Formal uniformization of the basic locus}
 \label{uniformpar}

By our construction, $\RZc_G$ 
comes   with a $W$-morphism
\[
\Theta: \RZc_G\to \widehat {\sS}_W
\]
where $\sS=\sS_U$ is the integral model of the Shimura variety
given by our choice of a global Shimura datum.  Such a morphism
is one of the main ingredients of the  uniformization theorems of 
 \cite[Theorem 6.2]{RapZinkBook} and  \cite{KimUnif}. In our approach, $\Theta$ is 
essentially part of the definition of $\RZc_G$. We can directly show  a version of   
 uniformization (Theorem \ref{uniformThm})  
by combining the above with results of Kisin \cite{KisinLR}.

\subsubsection{}
For simplicity, we will only discuss the uniformization when $b\in G(K)$ is {\sl basic.} 
We assume this is the case for the rest of this section.

We fix a sufficiently small
compact open subgroup $U^p$ of $ G(\A^p_f)$,  again set $U=U^pU_p$, and again abbreviate
$\sS=\sS_{U}(G, \mathcal{H})$ for the smooth integral
model over $\co_{E,(v)}$ of the Shimura variety ${\rm Sh}_{U}(G, \mathcal{H})$.

We continue as in \S\ref{sss:global to local datum}. In particular, we assume that $X_0=X_0(G, b, \mu,  C)$
arises as the $p$-divisible group $A_{x_0}[p^\infty]$ with tensors attached  to a point $x_0\in \sS_{U_p}(k)$.
We will denote also by $x_0$ the image of $x_0$ in $\sS(k)$.
Denote by 
\begin{equation}\label{newt stratum}
\sS_{b} \subset \sS\otimes_{\co_{E,(v)}}k
\end{equation}
the Newton stratum determined by  $b$   in  the geometric special fiber of $\sS$.   
By definition,   $\sS_{b}$ consists of all points $x$ such that  there is a quasi-isogeny 
\[
X_0\otimes_k k(x)=A_{x_0}[p^\infty]\otimes_k k(x)\dashrightarrow A_x[p^\infty]
\]
of $p$-divisible groups whose  corresponding morphism of contravariant Dieudonn\'e isocrystals identifies
$
t_{\alpha,x}\define i_x^*(t_\alpha^{\rm univ})
$
with  $t_{\alpha, 0}\otimes 1$.   Obviously,  $x_0\in \sS_b$.

By \cite{RapoportRich} and our assumption that $b$ is basic, the stratum (\ref{newt stratum}) is closed.
  Denote by   $(\widehat \sS_W)_{/\sS_b}$ the completion  of $\sS_W$ along $\sS_b$.

\begin{theorem}\label{uniformThm}(Kim \cite{KimUnif})
The morphism $\Theta$ extends to a $G({\mathbb A}^p_f)$-equivariant morphism
\begin{equation*}
\Theta: \RZc_G\times G({\mathbb A}^p_f) \xrightarrow{ \ \ } \widehat\sS_{U_p, W}\define \varprojlim\nolimits_{U^p}\widehat\sS_{U^pU_p, W}.
\end{equation*}
which   induces an isomorphism of formal schemes
\[
\Theta^b: I (\Q)\backslash  \RZc_G\times  G({\mathbb A}^p_f)/U^p  \iso  (\widehat \sS_{W})_{/\sS_b}.
\]

Here $I$ is a reductive group over $\Q$, which is an inner form of $G$, and is such that $I_\R$ is anisotropic modulo center.   
Moreover there are natural identifications 
\[
I_{\Q_\ell} = 
\begin{cases}
J_b & \mbox{if }\ell=p \\
G_{\Q_\ell} & \mbox{otherwise.}
\end{cases}
\]
The quotient  is for the action of $I(\Q)$ obtained by combining the (discrete) embedding  
$
I(\Q)\subset J_b(\Q_p)\times G({\mathbb A}^p_f)
$ 
given by the above  identifications, with  the actions of $J_b(\Q_p)$ on $\RZc_G$,  
and of $G({\mathbb A}^p_f)$ on $G({\mathbb A}^p_f)/K_p$ by left multiplication.
\end{theorem}

There is a more general result for non-basic $b$, which is more complicated to state. 
Compare, for example, with \cite[Theorem 6.23]{RapZinkBook} or \cite{KimUnif}.
Also, the uniformization isomorphism descends to an isomorphism 
over a finite unramified extension of $\Q_p$; again, we omit this discussion.

 \begin{proof} Given the existence of the morphism $\Theta$, this closely follows \cite{KimUnif}
 and  \cite{RapZinkBook}; for the convenience of the reader we sketch 
the proof here.

The morphism  
\[
\Theta: \RZc_G\to \widehat \sS_{U_p, W}
\]  
is given using Remark \ref{comparemark_level}  and, by its construction,
sends the base point $(X_0, {\rm id}, (t_{\alpha, 0}))$
 to the point $x_0$. Note that $G({\mathbb A}^p_f)$ acts on the projective system
 $\sS_{U_p}\define\varprojlim_{U^p}\sS_{U_pU^p}$ on the right (this is the prime-to-$p$ Hecke action) and so this gives a morphism
 \begin{equation*}
\Theta: \RZc_G\times G({\mathbb A}^p_f) \xrightarrow{ }  \widehat{\sS}_{U_p, W}.
\end{equation*}
After taking the quotient by $U^p$ we obtain  
  \begin{equation*}
\Theta: \RZc_G\times G({\mathbb A}^p_f)/U^p\xrightarrow{ }   \widehat{\sS}_W.
\end{equation*}
The prime-to-$p$ Hecke action on $ \sS_{U_p }$ preserves the $p$-divisible groups,
and we can easily see that this morphism  factors through $(\widehat \sS_W)_{/\sS_b}$.

By their construction (see also \S \ref{sss:Shimura set up}), the integral model $\sS $ supports a 
 lisse ${\mathbb A}^p_f$-sheaf $\mathrm{Ta}^p(A)_{\Q}$ given by the  
 Tate ${\mathbb A}^p_f$-module of the universal abelian scheme $A$,  and sections 
 \[
 t^{ p}_{\alpha, \et}: {\mathbb A}^p_f\to  \mathrm{Ta}^p(A)_{\Q}^{\otimes}.
 \] 
  
 We say that two  points $x$ and $x'$   
 of $\sS(k)$ are in the same isogeny class if there is a quasi-isogeny $f: A_x\dashrightarrow A_{x'}$ of the 
 corresponding abelian schemes,  respecting weak polarizations, such that the induced maps 
 \[
 \Db(A_{x'})[1/p]\iso  \Db(A_{x})[1/p]
 \]
 and 
 \[
 \mathrm{Ta}^p(A_{x})_{\Q}\iso   \mathrm{Ta}^p(A_{x'})_{\Q}
 \]
   send   $t_{\alpha, x'} \mapsto t_{\alpha, x}$ and $t^p_{\alpha, \et, x} \mapsto t^p_{\alpha, \et, x'}$,
  respectively; compare with   \cite[Prop. (1.4.15)]{KisinLR}. 
 
Kisin \cite[(2.1)]{KisinLR} associates to the isogeny class 
$
\phi \subset \mathscr{S}(k)
$ 
of $x_0$ an algebraic group $I=I_\phi$ 
with rational points $I(\Q)$ given by the self-quasi-isogenies 
$A_{x_0 }\dashrightarrow A_{x_0 }$ that preserve $t^p_{\alpha, \et, x_0}$ and $t_{\alpha, x_0}$. 
By its very definition, $I(\Q)$ is a subgroup of $J_b(\Q_p)\times G({\mathbb A}^p_f)$, and an 
 argument as in \cite[p. 289]{RapZinkBook} shows that this subgroup
 is discrete.

 Moreover, as in \cite{RapZinkBook},  $\Theta$ factors as
 \begin{equation}\label{factor unif}
 \Theta : I (\Q)\backslash  \RZc_G\times  G({\mathbb A}^p_f)/U^p \xrightarrow{ \ } (\widehat \sS_W)_{/\sS_b}.
 \end{equation}
 We continue to assume that $U^p$ is sufficiently small.  Using 
Corollary \ref{cor:points} and  \cite[Prop. (2.3.1)]{KisinLR}, we can see  that (\ref{factor unif}) 
 gives an injection on $k$-valued points.
By the proof of Proposition \ref{finite}, we then see that $\Theta$ induces an isomorphism between the formal completions at such points.
It then also follows that, for $U^p$ sufficiently small, the quotient 
\[
 I(\Q)\backslash \RZc_G\times  G({\mathbb A}^p_f)/U^p 
 \]
 is representable by a formal scheme over $W$. 
 
 It remains to show that  the group $I$ has the properties in the statement of  Theorem  \ref{uniformThm},
 and that (\ref{factor unif}) is an isomorphism. 
 
Note that the point  $x_0\in \sS (k)$  is actually defined over a finite field of cardinality $p^r$.  By \cite[(2.3) and Cor.~(2.3.5)]{KisinLR}, one sees that there is an element $\gamma_0\in G(\Q)$ such that $I$ is an inner form of the centralizer $I_0$ of a sufficiently divisible power of $\gamma_0$. In fact, $\gamma_0$ is a part of a so-called \emph{Kottwitz triple} 
 \[
 {\mathfrak k}=(\gamma_0, (\gamma_\ell)_{\ell\neq p}, \delta)
 \] 
 as in    [\emph{loc.~cit.}, (4.3) and  (4.4.6)]. Here, $\gamma_\ell$ belongs to $G(\Q_\ell)$, for all $\ell\neq p$.  
 Also $\delta$ belongs to $G(\Q_{p^r})$, with $\Q_{p^r}\subset K$ a finite unramified 
 field extension of $\Q_p$, and is $\sigma$-conjugate to $b$.  There are  reductive groups 
 $I_p$ over $\Q_p$, and $I_\ell$  over $\Q_\ell$ for  $\ell\neq p$,  associated to $\mathfrak k$.  
 By \cite[Cor.~(2.3.2)]{KisinLR}  we have  isomorphisms 
 \[
 I\otimes_\Q\Q_\ell\iso  I_\ell,\quad I\otimes_\Q\Q_p\iso  I_p.
 \]
As we assume that $b$ is basic,   a power of the element $\gamma_p\define \delta\sigma(\delta)\cdots \sigma^{r-1}(\delta)$ is central.
Therefore $I_p=J_\delta=J_b$ and they are both inner forms of $G$. It follows from the definition of Kottwitz triple that a power of $\gamma_0$ 
 is also central,   and hence $I_0=G$, and $I_\ell=G_{\Q_l}$  for $\ell\neq p$. The statements about the group $I$ follow from this
 and the results of Kisin \cite[(2.3)]{KisinLR} mentioned above.
  
 In fact, the isogeny class $\phi$ is  independent of our choice of (basic) point $x_0$.   More precisely, we have 
 
 \begin{proposition}\label{basicIsogeny}
 Suppose that  $x\in \sS_b (k)$; in other words, assume that $\Db(A_x[p^\infty])(W)$ has 
 Frobenius $F=b_x\circ \sigma$, where $b_x\in G(K)$   is $\sigma$-conjugate to  $b=b_{x_0}$ in $G(K)$.
 Then $x$ and $x_0$ are in the same isogeny class. 
 \end{proposition}
 
 \begin{proof} This again follows from \cite{KisinLR}. As above, a power of $\gamma_{p, x}$ obtained  from $\delta_x$ as above is central. Using this we can easily see that there is a unique equivalence class of Kottwitz triples ${\mathfrak k}=(\gamma_0, (\gamma_\ell)_{\ell\neq p}, \delta_x)$ with $\delta_x$ $\sigma$-conjugate to $b$. (See  
 \cite[4.3.1]{KisinLR}  for the definition of   the equivalence relation.)
 Now by \cite[Proposition 4.4.13]{KisinLR},  the set of isogeny classes which produce the same Kottwitz triple 
 ${\mathfrak k}$ is in bijection with the abelian group $\Sha_{G}(\Q, I )$. However, since $b$ is basic, $I_0=G$ as above, and we can see   from its definition [\emph{loc.~cit.}, \S 4.4.9 and \S 4.4.7]  that  $\Sha_{G}(\Q, I )$ is  trivial. This concludes the proof.
 \end{proof}
 
   Since the image of (\ref{factor unif}) on $k$-points is the isogeny class of $x_0$, Proposition \ref{basicIsogeny}
implies that (\ref{factor unif}) surjects onto $\sS_b (k)$. Given the above, the remaining claim that (\ref{factor unif}) 
is an isomorphism can be proven quickly by following  the arguments in \cite[Chapter 6]{RapZinkBook}.
\end{proof}

%%%%%%%%%%%%%%%%%%%%%%%%%%%%%%%%%%%%

%%%%%%%%%%%%%%%%%%%%%%%%%%%%%%%%%%%%

\section{Rapoport-Zink spaces for spinor similitude groups}
\label{s:GSpin RZ space}

%%%%%%%%%%%%%%%%%%%%%%%%%%%%%%%%%%%%

Now we turn to the description of a special class of Hodge type Rapoport-Zink formal schemes: those associated 
with the Shimura varieties for spinor similitude groups.  
Throughout \S\ref{s:GSpin RZ space}, \S\ref{s:lattices}, and \S\ref{s:RZ structure} we work purely locally.  
The  Shimura varieties themselves will not appear until \S\ref{s:global GSpin}.

Fix a nondegenerate  quadratic space $(V,Q)$ of rank $n\ge 3$ over $\Z_p$  and define a bilinear form 
on $V$ by  (\ref{bilinear}).
We assume that $V$ is self-dual, in the sense that the bilinear form  induces an isomorphism $V\iso \Hom(V,\Z_p).$
The space $V$ will remain fixed throughout \S\ref{s:GSpin RZ space}, \S\ref{s:lattices}, and \S\ref{s:RZ structure}.

%%%%%%%%%%%%%%%%%%%%%%%%%%%%

\subsection{Quadratic spaces, Clifford algebras, and spinor similitudes}
%\label{ss:GSpin}

%%%%%%%%%%%%%%%%%%%%%%%%%%%%%

For details on quadratic spaces, Clifford algebras,  and spinor similitude groups  we refer the reader to 
\cite{BassClifford}, \cite{MadapusiSpin}, and \cite{ShimuraQuadratic}.

\subsubsection{}\label{sss:quadratic invariants}

The \emph{Hasse invariant} of $V_{\Q_p}$ is the product of Hilbert symbols
\[
\epsilon( V_{\Q_p}) = \prod\nolimits_{  i<j } (a_i,a_j)_p,
\] 
where $e_1,\ldots, e_n \in V_{\Q_p}$ is an orthogonal basis and  $a_i=Q(e_i)$.
The \emph{determinant}  
\[
\det(V_{\Q_p}) = 2^n a_1\cdots a_n
\] 
 is  the determinant of the matrix of inner products $[e_i,e_j]$.  
It is  well-defined up to multiplication by a square in $\Q_p^\times$.  
The self-duality  hypothesis on $V$ implies  that  $\epsilon( V_{\Q_p})=1$ and 
\[
\ord_p(\det(V_{\Q_p})) \equiv 0 \pmod{2}.
\]

 \subsubsection{}\label{sss:azumaya}

The Clifford algebra  of $V$ is a $\Z/2\Z$-graded $\Z_p$-algebra denoted
\[
C(V)=C^+(V) \oplus C^- (V). 
\]
It is free of rank $2^n$ over $\Z_p$, generated as an algebra by the image of a canonical injection
$V\hookrightarrow C^-(V)$ satisfying $v\cdot v=Q(v)$.  The \emph{canonical involution} on $C(V)$ is the    
$\Z_p$-linear endomorphism $c\mapsto c^*$ characterized by 
 $
 (v_1 \cdots v_d)^*=v_d\cdots v_1
 $
 for all $v_1,\ldots, v_d\in V$.

For some faithfully flat $\Z_p$-algebra $R$ there is an isomorphism
\[
C(V_R) \iso \begin{cases}
M_{2^k }(R) & \mbox{if }n=2k  \\
M_{2^k}(R) \times M_{2^k}(R) & \mbox{if } n=2k+1.
\end{cases}
\]
The \emph{reduced trace} $\mathrm{Trd}:C(V)\to \Z_p$ is the unique $\Z_p$-linear map which induces,
under any such isomorphism, the usual trace on $M_{2^k }(R)$ when $n=2k$, and the sum of the usual traces
when $n=2k+1$.

 The center  $Z(V) \subset C(V)$ is easy to determine: if $n$ is even then $Z(V)=\Z_p$, while 
if  $n$  is odd then $Z(V)$ is either $\Z_{p^2}$ or $\Z_p\times \Z_p$, depending on the determinant of $V_{\Q_p}$.
In all cases  the natural map
\[
C(V) \otimes_{Z(V)}  C(V)^{op} \to \End_{Z(V)} (C(V) )
\]
is an isomorphism.

 \subsubsection{}\label{sss:G}
For a $\Z_p$-algebra $R$,  the tensor product   $V_R=V\otimes_{\Z_p} R$ is a nondegenerate quadratic space over $R$ 
with  Clifford algebra  $C(V_R)=C(V)\otimes_{\Z_p}R$.
  The \emph{spinor similitude group}   $G=\GSpin(V)$   is the reductive group over $\Z_p$ with $R$-points
\[
G(R) = \{ g\in C^+(V_R)^\times : g V_R g^{-1} =V_R,\, g^*g \in R^\times \},
\]
 and the \emph{spinor similitude}  
 $
 \eta_G:G \to \Gm
 $ 
 is the character  $\eta_G(g) = g^* g$.

The conjugation action of $G$ on $C(V)$ leaves invariant the $\Z_p$-submodule $V$, and this action of $G$ on $V$ is denoted
$g\action v = gvg^{-1}$.  There is a short exact sequence of group schemes
\[
1 \to \Gm \to G \map{g\mapsto g \action} \SO(V) \to 1
\]
over $\Z_p$, and the restriction of $\eta_G$  to the central $\Gm$ is  $z\mapsto z^2$.

\subsubsection{}\label{sss:symplectic rep}

If we fix any $\delta\in C(V)^\times$ with $\delta^*=-\delta$,  then 
\[
\psi_\delta(c_1,c_2) \define \mathrm{Trd}(c_1\delta c_2^*)
\]
is a perfect symplectic form on $C(V)$.   The group $G$, being a subgroup of $C(V)^\times$,  
acts on $C(V)$ by  left multiplication, yielding a closed immersion
$
 G \hookrightarrow \GSp(C(V) ,\psi_\delta) .
$
Under this embedding the symplectic similitude character  restricts to the spinor similitude on $G$.

\subsubsection{}\label{sss:D-module}

As in (\ref{contra}), we denote by  
$
D = \Hom_{\Z_p} (C(V) , \Z_p)
$
the contragredient representation. It follows from \S \ref{sss:azumaya} that there is an isomorphism
\[
C(V)^{op} \otimes_{Z(V)}  C(V) \iso \End_{Z(V)} (D)
\]
defined by $( (c_1 \otimes c_2) d ) ( c ) =  d( c_1 c c_2)$.  Note that the contragredient action of $G$ on $D$ commutes 
with the action of $C(V)$, but not with that of $C(V)^{op}$.

However, the inclusion $V\subset C(V)^{op}$ allows us to view 
\begin{equation}\label{simple special}
V \subset \End_{\Z_p} (D).
\end{equation}
These are the \emph{special endomorphisms} of $D$.  Again, they do not commute with the $G$ action; rather, they satisfy the relation
$
g\circ v \circ g^{-1} = g\action v
$
as endomorphisms of $D$, for any  $g\in G(\Z_p)$ and $v\in V$.

%%%%%%%%%%%%%%%%%%%%%%%%%%%%%%%%%

\subsection{The GSpin local Shimura datum}
\label{ss:gspin local datum}

%%%%%%%%%%%%%%%%%%%%%%%%%%%%%%%%%

From the quadratic space $V$ we will construct  an \emph{unramified local Shimura-Hodge} datum $(G, b  , \mu ,C(V) )$  in the sense of 
Definition \ref{def:shimura-hodge}.

\subsubsection{}\label{sss:mu basis}

Fix a $\Z_p$-basis  $x_1,\ldots,x_n \in V$ for which  the matrix of inner products has the form
\[
( [ x_i,x_j ] ) = \left(\begin{matrix}
0 & 1 \\
1 & 0 \\
& & * \\
& & & *  \\
& & &  & \ddots \\
& & & & & *
\end{matrix}\right)
\]
(the matrix is diagonal except for the upper left $2\times 2$ block.)
This choice of basis determines a  cocharacter $\mu:\Gm \to G$  by 
\[
\mu(t) = t^{-1}  x_1 x_2 +  x_2 x_1 ,
\]
where the arithmetic on the right hand side takes place in $C(V)$.

Under the  representation $G \to \SO(V)$, we have 
\[
\mu(t) \action x_i = \begin{cases}
t^{-1} x_i & \mbox{if }i=1 \\
t x_i & \mbox{if }i=2 \\
x_i & \mbox{if } 3 \le i \le n.
\end{cases}
\]
   The relation $x_1 x_2 + x_2x_1 =[x_1,x_2]=1$  implies that 
 $
 C(V) =x_1C(V) \oplus x_2 C(V),
 $ 
 and under the representation $G\to \GSp(C(V),\psi_\delta)$ we have 
\[
\mu(t) \cdot z = \begin{cases}
t ^{-1} z & \mbox{if } z\in x_1 C(V) \\
z & \mbox{if } z\in x_2 C(V).
\end{cases}
\]

The following lemma will be needed in the proof of Proposition \ref{prop:special lattice bijection}.

\begin{lemma}\label{lem:tech2}
For  self-dual $W$-lattices $A,A^\sharp\subset V_K$, the following are equivalent: 
 \begin{enumerate}
 \item
 $(A+A^\sharp)/A \iso W / pW$,
  \item
 there is a  $g\in G(K)$ such that  $ A^\sharp  =  g  \action V_W$ and  $A = g\mu(p^{-1} )  \action V_W.$
 \end{enumerate}
 \end{lemma}
 
 \begin{proof}
First assume  (i) holds.  As self-dual lattices are necessarily  maximal,
Theorem \ref{thm:elementary divisors} implies that $A$ and $A^\sharp$ have the form  
\[
A = W e \oplus W f  \oplus B_1 ,\quad 
A^\sharp =  W p e \oplus W p^{-1} f  \oplus B_1
\]
for some isotropic $e,f\in V_K$ with $[e,f]=1$, and some $W$-submodule $B_1\subset V_K$ orthogonal to both $e$ and $f$.

Now consider the self-dual $W$-lattice $V_W \subset V_K$.  The calculations of \S \ref{sss:mu basis} imply that
\begin{equation}\label{nearby base}
 \big( \mu(p^{-1})  \action V_W  + V_W \big) / \big( \mu(p^{-1})  \action V_W \big)  \iso W / pW ,
\end{equation}
 and  so there is a similar decomposition 
 \begin{align*}
\mu(p^{-1}) \action V_W  &= W \widetilde{e} \oplus W  \widetilde{f}  \oplus B_2 \\
V_W  &=  W p \widetilde{e} \oplus W p^{-1} \widetilde{f}  \oplus B_2
\end{align*}
 for some $W$-submodule $B_2 \subset V_K$.  
 
 Witt's cancellation theorem implies that 
 $B_{1 K} \iso B_{2K}$ as $K$-quadratic spaces.  As $B_1$ and $B_2$ are self-dual (hence maximal),  Theorem \ref{thm:eichler}
 implies that  $B_1$ and $B_2$ are isomorphic  as  $W$-quadratic spaces.   It follows that there is a 
 $g\in \SO( V_K )$  such that  $g \widetilde{e} = e$, $g \widetilde{f} = f$, and  $g B_2 =  B_1$.
Choosing any lift $g \in G(K)$ yields the element required in (ii).

  The reverse  implication is clear from (\ref{nearby base}).
\end{proof}

\subsubsection{}
%\label{sss:isocrystals}

As $G$ acts on both $V$ and $D$, any  $b \in  G(K)$ determines  isocrystals 
\[
( V_K, \Phi=b \circ  \sigma ) \quad \mbox{and}\quad (  D_K, F=b\circ  \sigma).
\]  
These will play a  central role in everything that follows.

Recall from \S \ref{sss:slope} that  $b\in G(K)$ is  \emph{basic}    if its slope cocharacter  $\nu_b:\mathbb{T}_K \to G_K$ factors through  
the center of $G_{ K}$.  As $\mathbb{T}_K$ is an inverse limit of connected group schemes, this is equivalent to the slope cocharacter factoring  
through the connected component of the center, which is
\[
\Gm =\ker\big(  G \to \SO(V) \big).
\]
Thus any basic  $b\in G(K)$ determines a rational number
\[
\nu_b \in \Hom(\mathbb{T}_K ,\Gm ) = \Q,
\]
which depends only on the $\sigma$-conjugacy class of $b$.

\begin{lemma}\label{lem:basic criterion}
For each $b\in G(K)$  the following are equivalent: 
\begin{enumerate}
\item
$b$ is basic,
\item
 the isocrystal $( V_K, \Phi=b\circ \sigma)$ is isoclinic of slope $0$,
\item
the isocrystal $( D_K , F= b\circ \sigma)$ is isoclinic.
\end{enumerate}
When these equivalent conditions hold, the slope of the isocrystal $D_K$ is $-\nu_b$.
\end{lemma}

\begin{proof}
The equivalence of the first two statements follows from the fact that the central $\Gm \subset G_K$ 
is  the kernel of the representation $G_K \to \SO(V_K)$.   The equivalence of the first and third follows from 
the observation that the representation $G_K \to \GL(D_K)$  identifies the  central  $\Gm \subset G_K$ with the 
torus of scalars in $\GL(D_K)$.   Moreover, $\Gm \subset G_K$ acts on $D_K$ through 
the  character $t\mapsto t^{-1}$, proving the final claim.
\end{proof}

\begin{proposition}\label{prop:twisted space}
  Every  basic $b\in G(K)$ satisfies
\begin{equation}\label{slope formula}
 \nu_b=   \frac{ \ord_p ( \eta_G(b) ) }{2}  ,
\end{equation}
and  $b \mapsto \nu_b$ establishes a bijection 
\begin{equation}\label{basic classes}
\{\mbox{basic } b\in G(K)  \} /  \mbox{$\sigma$-conjugacy} \iso   \frac{ 1 }{2} \Z  .
\end{equation}
Moreover, for any basic $b\in G(K)$  the $\Q_p$-quadratic space 
\[
V_K^\Phi = \{ x\in V_K : \Phi x=x\}
\]
  has the same dimension and  determinant as $V_{\Q_p}$, and has Hasse invariant 
$
\epsilon( V_K^\Phi) =   (-1)^{2  \nu_b} .
$
\end{proposition}

\begin{proof}
As in \cite[\S 2]{asgari}, the derived group of $G$ is the kernel of the spinor similitude, which is just the usual spin double cover 
of $\SO(V)$.  
In particular the derived group is simply connected, and results of Kottwitz (combine \cite[Proposition 5.4]{KottIso}  and
\cite[(2.4.1)]{KottIso}) imply that 
\[
\frac{\ord_p\circ \eta_G}{2} : G(K) \to \frac{1}{2} \Z
\]
 induces a bijection 
\[
\{\mbox{basic } b\in G(K)  \} /  \mbox{$\sigma$-conjugacy} \iso  \frac{1}{2} \Z.
\]

Recalling the basis  $x_1,\ldots, x_n$ of $V$  of \S \ref{sss:mu basis},  we now set
\begin{equation}\label{nice b}
b = x_3 ( p^{-1} x_1+x_2) \in G(\Q_p).
\end{equation}
A simple calculation gives 
\begin{equation}\label{b square}
b^2 = -p^{-1} Q(x_3) \in p^{-1}\cdot \Z_p^\times.
\end{equation}
As $\Gm \subset G$ acts on $D_K$ via $t\mapsto t^{-1}$,  the relation (\ref{b square}) 
implies that  $b^k$ makes $D_K$ into an isoclinic isocrystal
of slope $k/2$.  Thus $b^k$ is basic with  $\nu_{b^k} = -k/2$ by Lemma \ref{lem:basic criterion}.

As $\eta_G(p)=p^2$, the relation (\ref{b square})   implies $\ord_p( \eta_G(b) )  =-1$, and so 
the powers of $b$ form a complete set of representatives  for the basic $\sigma$-conjugacy classes.  As  these satisfy
\[
\nu_{b^k} = -\frac{k}{2}=  \frac{ \ord_p ( \eta_G(b^k) ) }{ 2 },
\]
we have now proved both (\ref{slope formula}) and (\ref{basic classes}).

As $b^2$ is a scalar, it lies in the kernel of $G(K) \to \SO(V)(K)$.  Thus  the  isocrystal structure on $V_K$
defined by $\Phi=b^k\circ \sigma$ depends only on  $k\pmod{2}$. 
 If $k=0$ then   $V_K^\Phi=V_{\Q_p}$
as subspaces of $V_K$, and so they have the same dimension, determinant,  and Hasse invariant.

On the other hand,  if $k=1$ then direct calculation shows that the isocrystal  $V_K$ defined by  $\Phi=b\circ \sigma$ satisfies 
\begin{equation}\label{b on V}
\Phi x_1  = -p x_2, \quad
\Phi x_2  = -p^{-1} x_1, \quad \Phi x_3 = - x_3,
\end{equation}
and $\Phi x_i =x_i$ for $i>3$.   If we define subspaces 
$M=\Q_p x_1 +\Q_p x_2 + \Q_p x_3$ and $N=M^\perp$ in $V_{\Q_p}$,  then there are orthogonal  decompositions
\[
V_{\Q_p}= M\oplus N,\quad 
V_K^\Phi = M_K^\Phi  \oplus N,
\] 
and an elementary calculation  (as in the proof of \cite[Proposition 2.6]{HP}) 
shows that  $M$ and $M_K^\Phi$ have the same dimension and determinant, but different Hasse invariants.  
Hence the same is true of $V_{\Q_p}$ and $V_K^\Phi$.
 \end{proof}

\begin{proposition}\label{prop:local GSpin datum}
If we let 
\begin{itemize}
\item
$\mu: \Gm \to G$ be as in \S \ref{sss:mu basis},
\item
 $b\in G(K)$ be defined by (\ref{nice b}), so that $\nu_b=-1/2$, 
 \item 
  $G \to \GSp( C(V) , \psi_\delta)$ be the representation of \S \ref{sss:symplectic rep},
  \end{itemize}
then   $b\in G(W) \mu^\sigma(p) G(W),$ 
and the action of $\Gm$ on $C(V)$ determined by 
\[
\Gm \map{\mu} G \to \GSp(C(V) ,\psi_\delta)
\] 
has the form
\[
t\mapsto \left(\begin{matrix}
t^{-1} I_{2^{n-1}} \\ 
& I_{2^{n-1}} 
\end{matrix} \right)
\]
for some choice of basis of $C(V)$.  In particular,  $(G, b,\mu,C(V))$ is a local unramified Shimura-Hodge datum in the sense of  
Definition \ref{def:shimura-hodge}. 
\end{proposition}

\begin{proof}
The calculations of \S \ref{sss:mu basis} show that the action of $\Gm$ on $C(V)$ has the stated form.
 Thus, as $\mu=\mu^\sigma$,  it suffices  to prove  $b\in \mu  (p) G(W)$.  
Comparing the calculations of \S \ref{sss:mu basis} with  (\ref{b on V}) shows that
\[
b\action V_W = \Phi (V_W ) = \mu(p) \action V_W
\]
as lattices in $V_K$.  Thus $\mu(p^{-1}) b$ lies in $p^\Z G(W)$, the stabilizer  in $G(K)$ of the lattice $V_W$. But 
\[
\ord_p( \eta_G(b) ) = -1 = \ord_p( \eta_G( \mu(p) )),
\]
and so  in fact $\mu(p^{-1}) b\in G(W)$.
\end{proof}

\begin{remark}
In general, given $G$ and $\{\mu\}$ as in \S\ref{general} with $\{\mu\}$ minuscule, there is a unique 
basic $\sigma$-conjugacy class $[b]$ such that $(G, [b], \{\mu\})$ is a local unramified Shimura datum.
This follows from \cite[Theorem 1.15 (i)]{RapoportRich}
 and the description of the set $B(G_{\Q_p}, \{\mu\})$ given there; see also   \cite[\S 5.2]{Wortmann}. 
\end{remark}

%%%%%%%%%%%%%%%%%%%%%%%%%%%%%%

\subsection{The GSpin Rapoport-Zink space}
%\label{ss:GSpin space}

%%%%%%%%%%%%%%%%%%%%%%%%%%%%%%

The  local Hodge-Shimura datum  $(G, b, \mu, C(V))$ of Proposition \ref{prop:local GSpin datum}   will remain fixed throughout
the remainder of \S \ref{s:GSpin RZ space},  and throughout \S \ref{s:lattices} and \S \ref{s:RZ structure}.

  \subsubsection{}\label{sss:special quasi-endomorphisms}

By Lemma \ref{lem:basic criterion}, the  isocrystals 
  \[
  (V_K , \Phi = b\circ \sigma) , \quad (D_K , F=b\circ \sigma)
  \]
have  slopes $0$ and $1/2$, respectively, and Lemma \ref{lemma121} implies that 
there is a $p$-divisible group 
\[
X_0=X_0( G,b,\mu,C(V))
\]
 over $k$  whose contravariant  Dieudonn\'e module is the lattice $D_W\subset D_K$.  
 The perfect symplectic form $\psi_\delta$ on $C(V)$
 determines a  principal  polarization 
 $
 \lambda_0: X_0 \to X_0^\vee,
 $ 
 and the inclusion $C(V)^{op}\subset \End(C(V))$ by right multiplication defines an action of $C(V)^{op}$ on $X_0$.
 Let $\Db(X_0)$ be the contravariant crystal of $X_0$, so that 
$\Db(X_0)(W)=D_W.$

Tensoring (\ref{simple special}) with $K$ yields a subspace
$
V_K\subset \End_K(D_K)
$ 
of special endomorphisms, on which the operators $\Phi$ and $F$  are related by 
$
\Phi x = F \circ x\circ F^{-1}.
$   
In particular, the $\Phi$-fixed vectors commute with $F$, and so determine  a distinguished  $\Q_p$-subspace
\[
V_K^\Phi \subset \End (X_0)_\Q
\]
of \emph{special quasi-endomorphisms} of $X_0$.   The restriction to $V_K^\Phi$ of the $K$-valued quadratic form on $V_K$
then satisfies
$
x\circ x = Q(x) \cdot \mathrm{id}.
$  
By Proposition \ref{prop:twisted space} the space $V_K^\Phi$
has the same dimension and determinant as $V_{\Q_p}$, but has Hasse invariant 
$
\epsilon ( V_K^\Phi ) =  - \epsilon(V_{\Q_p} ) = -1 .
$

\subsubsection{}\label{sss:RZ space}

Now we have all the ingredients needed to attach a Rapoport-Zink formal scheme to the quadratic space $(V,Q)$, 
using  the general  constructions of \S \ref{s:General RZ space} and \S \ref{ShimuraSect}.
By Theorem \ref{mainKim},  the  quadruple  $(G,b,\mu,C(V))$  determines a formal $W$-scheme
\begin{equation}\label{GSpin RZ}
\RZc=\RZc_{ G,b,\mu,C(V)  ,(s_\alpha) } 
\end{equation}
together with a closed immersion  $\RZc   \hookrightarrow  \RZc(X_0,\lambda_0).$   
Here $ \RZc(X_0,\lambda_0)$ is the symplectic Rapoport-Zink space as in \S \ref{sss:RZ}.
By Proposition \ref{prop:tensor independence}, the formal scheme (\ref{GSpin RZ}) does not depend on the collection of tensors  
$
(s_\alpha) \subset C(V)^\otimes = D^\otimes
$
  that cut out the subgroup $G\subset \GL(C(V)$. 
 
For this to make sense, we must explain why the datum $(G,b,\mu,C(V))$ has the form (\ref{localtoglobaleq}).   In other words, why this quadruple 
 (along with some choice of  tensors $s_\alpha$) agrees with the one coming from a $k$-point on an integral canonical model of a Shimura variety.  
The quadratic space $(V,Q)$ over $\Z_p$ can be realized as the $p$-adic completion of a quadratic space over $\Z_{(p)}$ of signature $(n-2,2)$,
and hence the reductive group $G$ and its representation   $C(V)$ also arise from analogous objects over $\Z_{(p)}$.  
The tensors $(s_\alpha)$ may then be chosen to come from this $\Z_{(p)}$-model of $C(V)$, and the existence of the desired point $x$ on a global Shimura variety then follows from Proposition \ref{switchpoint} below.

An explicit   list of  tensors $(s_\alpha)$ that cut out $G\subset \GL(C(V))$ can be found in \cite[Lemma 1.4(3)]{MadapusiSpin}.  
It will be convenient to fix, once and for all, such a list, and assume that it includes  the tensor induced by the symplectic form
$\psi_\delta$ on $C(V)$, as in the proof of Theorem \ref{mainKim}, and that it include a set of $\Z_p$-algebra generators for the subring
\[
C(V)^{op} \subset \End(C(V)) = C(V) \otimes C(V)^*
\] 
defined by  right multiplication.  (Equivalently, a set of generators for the subring $C(V) \subset \End(D)$ of \S\ref{sss:D-module}.)

 \subsubsection{}\label{sss:universal object}
The  restriction of the universal object  via $\RZc \hookrightarrow \RZc(X_0,\lambda_0)$  is a   pair $(X, \rho)$,
in which $X$ is a $p$-divisible group over $\RZc$, and 
\[
\rho:  X_0 \times_{\Spf(k) }  \overline{\RZc}   \dashrightarrow X \times_{\RZc} \overline{\RZc}
\] 
is a quasi-isogeny of $p$-divisible groups over 
\[
\overline{\RZc} = \RZc \times_{\Spf(W)} \Spf(k).
\]

As we have chosen our list of tensors to include  generators of the subalgebra
$C(V)^{op} \subset \End(C(V))$, the universal $X$ is endowed not only with a principal polarization $\lambda : X\to X^\vee$,
but also with an action of $C(V)^{op}$.  The universal quasi-isogeny is $C(V)^{op}$-linear.
The action of $C(V)^{op}$ on the universal object will play little role in what follows; it will be used only in the proof of
Proposition \ref{prop:projective} below.

The universal quasi-isogeny also  respects  the polarizations $\lambda$ and $\lambda_0$ up to scaling, and hence,
 Zariski locally on $\RZc$, we have
 $
 \rho^*\lambda = c(\rho)^{-1} \lambda_0
 $  
 for some $c(\rho) \in \Q_p^\times$.  For each $\ell\in \Z$ let $\RZc^{ (\ell)} \subset \RZc$   be the open and closed formal subscheme 
 on which $\ord_p( c(\rho)) =\ell$, so that
\[
\RZc = \bigsqcup_{\ell \in \Z }\RZc ^{ (\ell)}.
\]

\subsubsection{}\label{sss:J action}

The algebraic group $J_b = \GSpin( V_K^\Phi)$ has $\Q_p$-points
\[
J_b(\Q_p) = \{ g \in G(K) : g b = b \sigma(g) \},
\]
and acts   as automorphisms of the isocrystal $D_K$. This realizes 
$J_b(\Q_p) \subset \End(X_0)_{\Q}^\times$, and, as in  (\ref{Jaction}), there is an  induced  action of $J_b(\Q_p)$ on  $\RZc$.   
Each $g\in J_b(\Q_p)$ restricts to  an isomorphism
\[
g: \RZc^{(\ell)} \to \RZc^{ (\ell + \ord_p ( \eta_b (g) )) }.
\]
In particular, the subgroup  $p^\Z\subset J_b(\Q_p)$ acts on  $\RZc$, and, as $\eta_b(p)=p^2$,
\begin{equation}\label{fundamental domain}
p^\Z \backslash \RZc  \iso \RZc^{(0)} \sqcup \RZc^{(1)}.
\end{equation}
The surjectivity of  the spinor similitude $\eta_b : J_b(\Q_p)  \to \Q_p^\times$ implies that  the   $\RZc^{(\ell)}$ for various $\ell$ are  (non-canonically) isomorphic.

\subsubsection{}\label{sss:period domain} In this paper, we do not discuss the very interesting general theory of the $p$-adic symmetric domain and the period morphism for the $p$-analytic spaces associated to the formal schemes $\RZc_G$. See \cite[Chapter 5]{RapZinkBook}, and for the Hodge type case \cite{KimRZ}. 

We will just mention briefly, and without details, an elegant description of the $p$-adic symmetric domains that relates to the basic Rapoport-Zink spaces for $\GSpin(V)$ treated here.

In this case, the corresponding flag variety  is simply the quadric $\mathcal{Q}\subset \mathbb{P}(V)$ of isotropic lines $L\subset V$. 
If $F$ is a finite extension of $K$, the $F$-valued points  $\mathcal{Q}^{\rm wa}(F)$
of the (weakly) admissible locus $\mathcal{Q}^{\rm wa}\subset \mathcal{Q}^{\rm rig}_K$ in the rigid analytic quadric are those isotropic $F$-lines $L\subset V_F$
which are not contained in any  isotropic $F$-subspace of $V_F$ which is $\Q_p$-rational; here ``rational'' is  for the $\Q_p$-vector space structure on $V_F$ given by $V_F=V^\Phi_K\otimes_{\Q_p}F$.

 This description  can be obtained  by first
reducing consideration to the corresponding $p$-adic symmetric domain for the 
  group ${\rm SO}(V)=\GSpin(V)/\Gm$ 
(for example, see \cite[Cor. 9.2.22]{RapoportDatOrlik}),
and then by working through the definitions of \cite[Chapter 1]{RapZinkBook}
for ${\rm SO}(V)$. We leave the details to the reader.

%%%%%%%%%%%%%%%%%%%%%%%%%%%%%%%

\section{Vertex lattices and special lattices}
\label{s:lattices}

%%%%%%%%%%%%%%%%%%%%%%%%%%%%%%%%%

The section contains mostly linear algebra.  We study the family of \emph{vertex lattices}  $\Lambda \subset V_K^\Phi$,
and the family of \emph{special lattices} $L\subset V_K$.

%%%%%%%%%%%%%%%%%%%%%%%%%%%%%%

\subsection{Vertex lattices}

%%%%%%%%%%%%%%%%%%%%%%%%%%%%%%

In this subsection we introduce the  vertex lattices  and study their
 combinatorial properties (compare with  \cite{HP, VollaardSS, VollaardWedhorn}). Later, 
 in  \S \ref{s:RZ structure}, we will  express the reduced scheme underlying the spinor similitude Rapoport-Zink 
formal scheme (\ref{GSpin RZ}) as a union  of    closed subschemes  indexed by these vertex lattices.

\begin{definition}\label{def:vertex lattice}
A \emph{vertex lattice} is a $\Z_p$-lattice $\Lambda \subset V_K^\Phi$ satisfying
\[
p\Lambda \subset \Lambda^\vee \subset \Lambda,
\]
where $\Lambda^\vee$ is the dual lattice in the sense of Definition \ref{def:q-maximal}.
The \emph{type} of $\Lambda$ is \[t_\Lambda= \dim_{\F_p} ( \Lambda/\Lambda^\vee).\]
\end{definition}

%Define a positive even integer 
%\[
%t_\mathrm{max} =   \begin{cases}
%n-2 & \mbox{if $n$ is even and $\det(V_{\Q_p}) = (-1)^{\frac{n}{2}} $} \\
%n-1 & \mbox{if $n$ is odd} \\
%n & \mbox{if $n$ is even and $\det(V_{\Q_p}) \neq (-1)^{\frac{n}{2}} $}.
%\end{cases}
%\]

\begin{proposition}\label{prop:possible types}
Let $\Lambda$ be a vertex lattice, and recall the integer $t_\mathrm{max}$
 of (\ref{tmax}).   The type $t_\Lambda$  is even and satisfies 
$
2\le t_\Lambda \le t_\mathrm{max}.
$
Furthermore,  every vertex lattice is contained in a vertex lattice of type $t_\mathrm{max}$.
\end{proposition}

\begin{proof}
Recall from \S \ref{sss:quadratic invariants} and  Proposition \ref{prop:twisted space} that  
\begin{equation}\label{even dets}
\ord_p(\det(V_K^\Phi))  \equiv \ord_p( \det(V_{\Q_p}) ) \equiv 0 \pmod{2}.
\end{equation}
It follows that the type   of a  vertex lattice is even.  
The type cannot be $0$, for then $V_K^\Phi$ would contain a self-dual lattice, contradicting the 
Hasse invariant calculation $\epsilon(V_K^\Phi) = -1$.

Let $\mathbf{Lat}$ be the set of all $\Z_p$-lattices  $\Lambda \subset V_K^\Phi$ satisfying $[ \Lambda , \Lambda] \subset p^{-1}\Z_p$.  In other words,
$\mathbf{Lat}$ is the set of all lattices for which $p\Lambda \subset \Lambda^\vee$, and so $\mathbf{Lat}$ contains all vertex lattices.  
Let $\Lambda$ be any lattice which is maximal (with respect to inclusion) among all elements of $\mathbf{Lat}$.  We will prove that $\Lambda$ is a vertex
lattice of type $t_\mathrm{max}$, from which the proposition follows immediately.

The lattice $\Lambda$ is a maximal lattice (in the sense of Definition \ref{def:q-maximal}) with respect to the  rescaled quadratic  form $pQ$ on $V_K^\Phi$,
and so by Theorems \ref{thm:eichler} and  \ref{thm:elementary divisors} there is a decomposition
\[
\Lambda = \mathrm{Span}_{\Z_p}\{  e_1 ,  f_1, \ldots , e_r ,  f_r \}  \oplus Z
\]
in which $Z_{\Q_p}$ is anisotropic, 
\begin{equation}\label{Z space}
Z= \{ x\in Z_{\Q_p} : Q(x) \in p^{-1} \Z_p \},
\end{equation}
and the vectors $e_i$ and $f_j$ satisfy 
\[
[Z,e_i] = [Z,f_i] =0, \quad   [e_i,e_j]=[f_i,f_j ] =0 ,
\] 
and $[e_i,f_j] = p^{-1}\delta_{i,j}$.
  As every quadratic space over $\Q_p$ of dimension greater than $4$ contains an isotropic vector, we also  have 
 $\mathrm{dim} ( Z_{\Q_p} )  \le 4$.

The relation  $p \Lambda \subset \Lambda^\vee$ implies    $pZ\subset Z^\vee$.  We cannot have $Z^\vee =Z$, for then 
\[
 \mathrm{Span}_{\Z_p}\{  p e_1 ,  f_1, \ldots , p e_r ,  f_r \}  \oplus Z \subset V_K^\Phi
\]
would be a self-dual lattice,  contradicting $\epsilon(V_K^\Phi)=-1$.  In particular, $Z\neq 0$.

If  $\dim(Z_{\Q_p})=1$  then  $Z_{\Q_p}\iso \Q_p$ with the quadratic form
$Q(x)= c x^2$ for some $c\in \Q_p^\times/(\Q_p^\times)^2$.  We cannot have  $\ord_p(c)$ even, for  then (\ref{Z space}) implies
$Z=Z^\vee$,  contradicting what was said above.  But we also cannot have $\ord_p(c)$ odd, for  then 
\[
\ord_p(\det(V_{\Q_p}))=\ord_p( \det(V_K^\Phi) ) = 2r + \ord_p( \det(Z_{\Q_p} )) 
\]
is odd,  contradicting (\ref{even dets}).  Thus  $\dim(Z_{\Q_p}) \in \{ 2,3,4\}$. 

Suppose   $\dim(Z_{\Q_p})=2$.  Let $\Q_{p^2}$ be the unramified quadratic extension of $\Q_p$, and let $x\mapsto \bar{x}$
be its nontrivial Galois automorphism.  For some $c\in \Q_p^\times /\mathrm{Nm}(\Q_{p^2}^\times)$ there is an isomorphism
$Z_{\Q_p}\iso \Q_{p^2}$ identifying the quadratic form $Q$ with $Q(x) = c x\bar{x}$.  
If $\ord_p(c)$ is even then, as above,  $Z=Z^\vee$ yields a contradiction.  Thus $\ord_p(c)$ is odd, and simple calculation shows that
$Z^\vee\subset Z$, $\dim(Z/Z^\vee)=2$  and $\det(Z_{\Q_p}) = -u$ for a nonsquare $u\in \Q_p^\times$.
This implies that $\Lambda^\vee \subset \Lambda$ with  $\dim(\Lambda/\Lambda^\vee) = 2r+2 = n,$ and 
\[
\det(V_{\Q_p}) = \det(V_K^\Phi) = (-1)^r \det(Z_{\Q_p}) = (-1)^{\frac{n}{2}} u \neq (-1)^{\frac{n}{2}}.
\]
Thus  $\Lambda$ is a vertex lattice of type $n=t_\mathrm{max}$.

Suppose that   $\dim( Z_{\Q_p} ) =3$.  Let $B$ denote the quaternion division algebra over $\Q_p$, with its main involution $x\mapsto \bar{x}$.  
The subspace of traceless elements $B^0= \{ x\in B : x+\bar{x}=0\}$ has dimension $3$, and the 
reduced norm $\mathrm{Nrd}(x)=x\bar{x}$  restricts to an anisotropic quadratic form on $B^0$ with 
$\ord_p( \det(B^0))$ even.  In fact
\[
(B^0 , \mathrm{Nrd}) \iso (\Q_p^3 , -ux_1^2 -p x_2^2 + upx_3^2 )
\]
for any nonsquare $u\in \Z_p^\times$.
There are exactly four anisotropic quadratic spaces over $\Q_p$ of dimension $3$, and they are the spaces $(B_0, c \mathrm{Nrd})$
with $c\in \Q_p^\times / (\Q_p^\times)^2$.  If $\ord_p(c)$ is odd then $\ord_p(\det(V_K^\Phi))$ is also odd, contradicting 
(\ref{even dets}). Thus $\ord_p(c)$ is even, and one easily checks from (\ref{Z space}) that $Z^\vee \subset Z$ with 
$Z/Z^\vee$ of dimension $2$.  Thus $\Lambda^\vee \subset \Lambda$ and  
\[
\mathrm{dim} ( \Lambda/\Lambda^\vee )  = 2r+2 = n-1 . 
\]
In other words, $\Lambda$ is a vertex lattice of type $n-1=t_\mathrm{max}$.

Finally,  suppose $\dim(Z_{\Q_p})=4$.  By  \cite[Corollary IV.2.3]{SerreAcourse}, the only anisotropic quadratic space of dimension $4$ is $Z_{\Q_p} \iso B$
with its reduced norm form.  In particular  $\det(Z_{\Q_p})=1$, and $Z\iso \mathfrak{m}^{-1}$, where $\co_B\subset B$ is  the unique maximal order and  $\mathfrak{m}\subset \co_B$ is its unique maximal ideal.  The dual lattice is $Z^\vee =\co_B$, and 
it  follows that  $Z^\vee \subset Z$ with   $\dim(Z/Z^\vee) =2$.  This implies that $\Lambda^\vee \subset \Lambda$, 
$
\dim(\Lambda/\Lambda^\vee) = 2r+2 = n-2,
$
and 
\[
\det(V_{\Q_p}) = \det(V_K^\Phi) = (-1)^r \det(Z_{\Q_p}) = (-1)^{\frac{n}{2}} .
\]
  Thus  $\Lambda$ is a vertex lattice of type $n-2=t_\mathrm{max}$.
\end{proof}

There is a natural notion of adjacency between  vertex lattices, which makes them into the vertices
of a connected graph, as we now explain.

\begin{definition}
Two vertex lattices  $\Lambda_1 ,\Lambda_2 \subset V_K^\Phi$ are \emph{adjacent}
if either $\Lambda_1 \subsetneq \Lambda_2$ or $\Lambda_2 \subsetneq \Lambda_1$.
\end{definition}

We write $\Lambda_1 \sim \Lambda_2$ to indicate that $\Lambda_1$ and $\Lambda_2$ are adjacent.
Adjacent lattices have different types, 
and the inclusion between them  is always the lattice of smaller type inside the  lattice of larger type.

\begin{proposition}\label{prop:all types}
Let $\Lambda \subset V_K^\Phi$ be a vertex lattice of type $t_\Lambda$ and suppose $t\neq  t_\Lambda$ is any 
even integer with  $2 \le t \le t_\mathrm{max}$.  There is a vertex lattice of type $t$ adjacent  to $\Lambda$.
\end{proposition}

\begin{proof}
First suppose that $t< t_\Lambda$.  The quadratic form $q(x) = pQ(x)$ makes   $\Lambda/\Lambda^\vee$ into
a nondegenerate quadratic space over $\F_p$ of rank  $t_\Lambda \ge 4$, and Corollary  1.2.2 of \cite{SerreAcourse} implies the
existence of  an isotropic line $\ell \subset \Lambda/\Lambda^\vee$.
The orthogonal $\ell^\perp \subset \Lambda/\Lambda^\vee$ determines a vertex lattice $\Lambda' = \Lambda^\vee + \ell^\perp \subset \Lambda$
of type $t_{\Lambda'} = t_\Lambda - 2$, and repeating this process yields a vertex lattice of type $t$  contained in $\Lambda$.

Now suppose that $t_\Lambda <t$.  By Proposition \ref{prop:possible types} 
there is a vertex lattice $\Lambda_\mathrm{max}$ of maximal type $t_\mathrm{max}$ satisfying
$
\Lambda_\mathrm{max}^\vee  \subsetneq \Lambda^\vee \subset \Lambda \subsetneq \Lambda_\mathrm{max}. 
$
The subspace $\Lambda^\vee / \Lambda^\vee_\mathrm{max} \subset \Lambda_\mathrm{max}/ \Lambda^\vee_\mathrm{max}$ is 
totally isotropic, and for any codimension one subspace $\ell \subset \Lambda^\vee / \Lambda^\vee_\mathrm{max}$ the orthogonal
\[
\Lambda / \Lambda^\vee_\mathrm{max} \subset \ell^\perp \subset \Lambda_\mathrm{max}/ \Lambda^\vee_\mathrm{max}
\]
determines a vertex lattice $\Lambda' = \Lambda^\vee_\mathrm{max} + \ell^\perp$ of type $t_\Lambda+2$ containing $\Lambda$.  Repeating  this 
process   yields a vertex lattice of type $t$  containing  $\Lambda$.
\end{proof}

The following proposition, which proves the connectedness of the graph of vertex lattices,  
will be used in the proof of Theorem \ref{thm:final}   to show that $\RZc^{(\ell)}$ is connected.

\begin{proposition}\label{prop:connected lattices}
Given any two vertex lattices $\Lambda', \Lambda'' \subset V_K^\Phi$, there is a sequence of adjacent
vertex lattices
\[
\Lambda'\sim  \Lambda_1 \sim \Lambda_2 \sim \cdots\sim \Lambda_s \sim \Lambda''.
\]
\end{proposition}

\begin{proof}
As in the proof of Proposition \ref{prop:possible types}, let $\mathbf{Lat}$ be the set of all $\Z_p$-lattices
$\Lambda \subset V_K^\Phi$ satisfying $[\Lambda, \Lambda] \subset p^{-1} \Z_p$.    
Recall that $\mathbf{Lat}$ contains all vertex lattices.  Recall also that any maximal   (with respect to inclusion) element 
$\Lambda\in \mathbf{Lat}$   is necessarily a vertex lattice of type $t_\mathrm{max}$, and is a maximal lattice 
with respect to the rescaled quadratic form $pQ$.

Pick  maximal elements $\underline{\Lambda}' , \underline{\Lambda}'' \in \mathbf{Lat}$ with $\Lambda'\subset \underline{\Lambda}'$
and $\Lambda'' \subset \underline{\Lambda}''$.  In particular $\Lambda' \sim \underline{\Lambda}'$ and 
$\Lambda'' \sim \underline{\Lambda}''$.  Using the  maximality of $\underline{\Lambda}'$ and  $\underline{\Lambda}''$
with respect to $pQ$, Theorem \ref{thm:elementary divisors} implies that there are decompositions
\[
\underline{\Lambda}'' =  \mathrm{Span}_{\Z_p}\{   e_1,  f_1 , \ldots,  e_r , f_r \}  \oplus Z
\]
and
\[
\underline{\Lambda}' = \mathrm{Span}_{\Z_p}\{   p^{a_1} e_1,   p^{-a_1} f_1 , \ldots , p^{a_r} e_r ,  p^{-a_r}  f_r  \}  \oplus Z
\]
where all $e_i$ and $f_i$ are isotropic, $[e_i,f_j]=p^{-1} \delta_{ij}$, each $a_i \ge 0$, and 
 $Z$ is orthogonal to all  $e_i$ and $f_i$ and satisfies  $pZ\subset Z^\vee \subset Z$.

From these decompositions it is elementary to construct a chain of adjacent vertex lattices 
\[
\Lambda'\sim \underline{\Lambda}' \sim \Lambda_1 \sim \Lambda_2 \sim \cdots\sim \Lambda_s\sim \underline{\Lambda}''  \sim \Lambda''.
\]
For example,  set
\[
\Lambda_1 =   \mathrm{Span}_{\Z_p} \{   p^{a_1} e_1 ,  p^{-a_1+1} f_1 ,  p^{a_2} e_2,   p^{-a_2} f_2,
\ldots , p^{a_r} e_r, p^{-a_r}  f_r \}  \oplus Z
\]
so that $\Lambda_1 \subsetneq \underline{\Lambda}'$ is a vertex lattice of type $t_\mathrm{max}-2$, and then set
\[
\Lambda_2 =  \mathrm{Span}_{\Z_p} \{  p^{a_1-1} e_1,  p^{-a_1+1} f_1 , p^{a_2} e_2,   p^{-a_2} f_2,
\ldots ,  p^{a_r} e_r ,  p^{-a_r}  f_r \} \oplus Z
\]
so that $\Lambda_2\supsetneq \Lambda_1$ is a vertex lattice of type $t_\mathrm{max}$.  Repeat until all the 
exponents reach $0$.
\end{proof}

%%%%%%%%%%%%%%%%%%%%%%%%%%%%%%%

\subsection{Special lattices}
\label{ss:special lattices}

%%%%%%%%%%%%%%%%%%%%%%%%%%%%%%%%%

We now define a family of special lattices in $V_K$.  In \S \ref{s:RZ structure} we will show that these special lattices
are in bijection with the set $p^\Z \backslash \RZc(k)$.

 In fact, we will need a similar result for any finitely generated  extension $k'$ of $k$. 
Let  $W'$ be the Cohen ring of $k'$, let $K'=W'[1/p]$ be its fraction field, and let    $\sigma: K' \to K'$  be any  lift of Frobenius.
Define a $\sigma$-linear operator  $\Phi = b\circ \sigma$ on $V_{K'}$.
 If $L \subset V_{K'}$ is any  $W'$-submodule, let $\Phi_*(L) \subset V_{K'}$ be the $W'$-submodule generated by $\Phi(L)$.

\begin{definition}
A \emph{special lattice} $L\subset V_{K'}$ is a self-dual $W'$-lattice such that 
\[
  (L+\Phi_*(L))/L  \iso W' / pW'.
\]
\end{definition}

The following proposition implies  that for every special lattice $L\subset V_K$
there is a vertex lattice $\Lambda \subset V_K^\Phi$ with 
$\Lambda_W^\vee  \subset L \subset \Lambda_W.$
In fact, there is a unique  minimal  such $\Lambda$, denoted $\Lambda(L)$.
The proof is identical to that of  \cite[Proposition 4.1]{RTW} and  \cite[Lemma 2.1]{VollaardSS}, and so is omitted here.

\begin{proposition}\label{prop:optimal vertex}
Let $L\subset V_K$ be a special lattice.  If we define 
\[
L^{(r)}= L + \Phi(L) + \cdots + \Phi^r(L),
\]  
then there is a (necessarily unique) integer $1\le d \le  t_\mathrm{max}/2 $ such that 
\[
L=L^{(0)} \subsetneq L^{(1)} \subsetneq \cdots\subsetneq L^{(d)} =L^{(d+1)}.
\]
Moreover, the $W$-module $L^{(r+1)}/L^{(r)}$ has length $1$  for all $ r<d$, and 
\[
\Lambda(L) = \{ x\in L^{(d)}  : \Phi(x) = x\} \subset V_K^\Phi
\]
is a vertex lattice of type $2d$  satisfying  $\Lambda(L)^\vee =\{ x\in L : \Phi(x) = x \}.$  
\end{proposition}

\subsection{The variety $S_\Lambda$}

We next  attach to a vertex lattice 
$
\Lambda \subset V_K^\Phi
$ 
a $k$-variety $S_\Lambda$  parametrizing certain special lattices. 

\subsubsection{}\label{sss:Xlambda}

Define an $\F_p$-vector space  $\Omega_0=\Lambda/\Lambda^\vee$ 
of dimension $t_\Lambda$.  The quadratic form $pQ$ on $\Lambda$ is $\Z_p$-valued, and its reduction modulo $p$ makes $\Omega_0$ into a nondegenerate quadratic space over $\F_p$.  Set 
\[
\Omega \define \Omega_0\otimes_{\F_p} k \iso \Lambda_W / \Lambda_W^\vee
\] 
with its Frobenius operator  $\mathrm{id}\otimes \sigma = \Phi$. 
  Note that $\Omega_0$ cannot admit a Lagrangian ($=$ totally isotropic of dimension $t_\Lambda/2$) subspace.   
  Indeed, if such a subspace $\mathscr{L} \subset \Omega_0$ existed, then 
$\Lambda^\vee+ \mathscr{L}\subset V_K^\Phi$ would be a vertex lattice of type $0$, 
contradicting Proposition \ref{prop:possible types}.
In fact, $\Omega_0$ is characterized up to isomorphism as the unique nondegenerate 
quadratic space of dimension $t_\Lambda$ that does \emph{not} admit a Lagrangian subspace.

 The \emph{orthogonal Grassmannian} $\OGr(\Omega)$  is the moduli 
space of  Lagrangian subspaces $\mathscr{L}\subset \Omega$.    More precisely, 
an $R$-point of $\OGr(\Omega)$ is a totally isotropic local direct summand $\mathscr{L} \subset \Omega\otimes_k R$
of rank $t_\Lambda/2$.
Denote by  $S_\Lambda \subset \OGr(\Omega)$  the reduced closed subscheme 
with  $k$-points
\begin{align*}
S_\Lambda (k) & = \left\{ \mbox{Lagrangians }\mathscr{L} \subset \Omega : \dim_k( \mathscr{L}+\Phi(\mathscr{L}) ) = \frac{t_\Lambda}{2}+1 \right\} \\
& \iso  \{ \mbox{special lattices } L \subset V_K :   \Lambda^\vee_W \subset L \subset \Lambda_W \}.
\end{align*}

\begin{proposition}\label{prop:lagrangian}
The $k$-scheme $S_\Lambda$ has two connected components  
$
S_\Lambda = S_\Lambda^+ \sqcup S_\Lambda^-.
$
The two components are isomorphic, and each  is  projective and smooth of dimension $( t_\Lambda / 2 ) -1$.
\end{proposition}

\begin{proof}
All of the claims are included in  \cite[Proposition 3.6]{HP}, except for the isomorphism 
$S_\Lambda^+\iso S_\Lambda^-$.
Pick any  $g\in \mathrm{O}(\Omega_0)(\F_p)$ with $\det(g)=-1$.  
The natural action of $g$ on $\mathrm{OGr}(\Omega)$ leaves $S_\Lambda$ invariant, and the
 discussion  of \cite[\S 3.2]{HP} shows that $g$ interchanges $S_\Lambda^+$ with  $S_\Lambda^-$.
\end{proof}

%%%%%%%%%%%%%%%%%%%%%%%%%%%%%%%

\section{Structure of the spinor similitude Rapoport-Zink space}
\label{s:RZ structure}

%%%%%%%%%%%%%%%%%%%%%%%%%%%%%%%%%

We will determine explicitly the  structure of the reduced $k$-scheme $\RZc^\red$ underlying the formal $W$-scheme
$\RZc$ of  \S \ref{sss:RZ space}.  More precisely, we will  express $\RZc^\red$ as a union of closed subschemes 
$\RZc_\Lambda^\red$ indexed by vertex lattices, and then relate each $\RZc_\Lambda^\red$ to the variety
$S_\Lambda$ of \S  \ref{sss:Xlambda}.

\subsection{Closed subschemes defined by vertex lattices}

Recall from \S \ref{sss:special quasi-endomorphisms} and \S \ref{sss:universal object} the $\Q_p$-quadratic space of special quasi-endomorphisms
$
V_K^\Phi \subset \End ( X_0 )_\Q ,
$
and the universal quasi-isogeny 
\[
\rho:  X_0 \times_{\Spf(k) }  \overline{\RZc}   \dashrightarrow X \times_{\RZc} \overline{\RZc}.
\]

\subsubsection{}

Fix a vertex lattice $\Lambda \subset V_K^\Phi$, and   denote by
$\RZc_\Lambda\subset \RZc $  the closed \cite[Proposition 2.9]{RapZinkBook} formal subscheme defined by the condition
\[
\rho  \circ \Lambda^\vee\circ  \rho^{-1}  \subset \End(X).
\] 
In other words, $\RZc_\Lambda$ is the locus where the quasi-endomorphisms $\rho  \circ \Lambda^\vee\circ  \rho^{-1}$  of $X$ are actually integral.  
As in  \S \ref{sss:J action}, the subgroup  $p^\Z\subset J_b(\Q_p)$ acts on  $\RZc_\Lambda$.  Set
\[
\RZc^{(\ell)}_\Lambda= \RZc_\Lambda \cap \RZc^{(\ell)},
\]
so that $p^\Z \backslash \RZc_\Lambda  \iso \RZc_\Lambda^{(0)} \sqcup \RZc_\Lambda^{(1)}$, exactly as in (\ref{fundamental domain}).

\begin{proposition}\label{prop:projective}
The  reduced $k$-scheme underlying   $\RZc^{(\ell)}_\Lambda$  is  projective.
\end{proposition}

\begin{proof}
Abbreviate $Z=Z(V)$ for the center of $C(V)$.  Using the isomorphism 
\begin{equation}\label{azumaya action}
C(V)^{op} \otimes_{Z} C(V) \iso \End_{Z} ( D )
\end{equation}
of  \S\ref{sss:symplectic rep}, and  the inclusion
$
\Lambda^\vee \subset V_K \subset C(V_K)^{op},
$
we denote by 
\[
R  \subset \End_{Z_K}(D_K)
\] 
the $W$-subalgebra generated by  
$
\Lambda^\vee  \otimes_{\Z_p} C(V)  \subset \End_{Z_K} ( D_K ).
$   
The isomorphism (\ref{azumaya action}) implies that $R$ generates $\End_{Z_K} ( D_K )$ as a $K$-vector space.
Fix any maximal $Z_W$-order $\tilde{R}$ with 
\[
R\subset \tilde{R}  \subset \End_{Z_K}(D_K).
\]
As $R$ and $\tilde{R}$ are both $W$-lattices in $\End_K(D_K)$,  we  have $p^m   \tilde{R} \subset R$   for some positive integer $m$.
Fix a  $W$-lattice  $\tilde{M} \subset D_K$  stable under the action of $\tilde{R}$.  It  is  unique up to scaling by $Z_K^\times$.

Suppose  $y\in \RZc^{(\ell)}_\Lambda(k)$,  and use the quasi-isogeny $\rho_y: X_0 \dashrightarrow X_y$ to view 
\[
M_y=\Db(X_y)(W)
\] 
as a $W$-lattice in $D_K=\Db(X_0)(W) [1/p]$. On one  hand, $M_y$ is stable under the action of $C(V)$ defined by (\ref{azumaya action}).
Indeed, this action in precisely the action on $M_y$ induced by the action of $C(V)^{op}$ on $X_y$ and contravariant functoriality;
see \S \ref{sss:universal object}.   On the other hand,  $\Lambda^\vee \subset \End  (  X_y  )$ by the very definition of 
$\RZc^{(\ell)}_\Lambda$.   Combining these, we see that $M_y$ is stable under the action of $R$, and so we may define the 
$\tilde{R}$-stable lattice $\tilde{M}_y = \tilde{R} \cdot M_y$.    
By the uniqueness of $\tilde{M}$ up to scaling, there is an $a(y) \in Z_K^\times$ such that 
$\tilde{M}_y = a(y) \tilde{M}$, and so
\[
p^m a(y) \tilde{M} \subset  M_y \subset a(y) \tilde{M} .
\]

First suppose that $n$ is even, so that $Z_K=K^\times$.  The perfect symplectic form
$\psi_\delta$ on $C(V_W)$ induces a dual form on $D_W$, which satisfies
\[
p^\ell \psi_\delta( D_W, D_W) = \psi_\delta( \tilde{M}_y , \tilde{M}_y) = a(y)^2 \psi_\delta(  \tilde{M} ,  \tilde{M}).
\]
Thus the $p$-adic valuation of $a(y)$ is constant as $y$ varies, and we may choose $a=a(y)$ to be independent of $y$.

Now suppose that $n$ is odd, so that $Z_K = K \times K$.  Let $\epsilon_1, \epsilon_2\in Z_K$ be the orthogonal
idempotents.  In this case the dual form on $D_W$ satisfies
\[
p^\ell \psi_\delta( \epsilon_i D_W,  \epsilon_i D_W) =
 \psi_\delta( \epsilon_i \tilde{M}_y , \epsilon_i\tilde{M}_y) = a_i(y)^2 \psi_\delta(  \epsilon_i\tilde{M} , \epsilon_i \tilde{M})
\]
where $a_i(y) = \epsilon_i(y) a(y)$.  Again this shows that $a_i(y)$ has constant $p$-adic valuation as $y$ varies,
and we may take $a=a(y)$ to be independent of $y$.

In either case
$
a p^m   \tilde{M} \subset  M_y \subset a \tilde{M}
$  
for all $y$.   Combining this  bound on $M_y$ and the projectivity result of   \cite[Corollary 2.29]{RapZinkBook}, 
we see that  the closed immersion    
$
\RZc\hookrightarrow \RZc(X_0,\lambda_0)
$ 
 realizes the reduced scheme underlying  $\RZc^{(\ell)}_\Lambda$ as a closed subscheme of a projective $k$-scheme.
\end{proof}

\subsection{Special lattices and the points of $\RZc_\Lambda$}

Let $k'/k$ be a finitely generated field extension.  Let $W'$ be the Cohen ring of $k'$, set $K'=W'[1/p]$, and 
let $\sigma : W' \to W'$ be a lift of Frobenius chosen as in Proposition \ref{fgfield}.  This choice of $\sigma$ determines
an operator $F=b\circ \sigma$ on $D_{K'}$, and an operator $\Phi=b\circ \sigma$ on $V_{K'}$.

\subsubsection{}\label{sss:the specials}

As in the proof of Proposition \ref{prop:projective}, each $y\in \RZc(k')$ determines a $W'$-lattice
\[
M_y = \mathbb{D}(X_y)(W') \subset D_{K'},
\]
and a $W'$-submodule
$
M_{1,y} =F^{-1}(pM_y) 
$
as in \S \ref{Zinkpar}.
Using the inclusion  $V_{K'} \subset \End_{K'}(D_{K'})$ obtained from  (\ref{simple special}), 
define   $W'$-lattices
\begin{align*}
L_y  &= \{ x\in V_{K'} : xM_{1,y} \subset M_{1,y} \} \\
L^\sharp_y &= \{ x\in V_{K'} : xM_y\subset M_y \}. \nonumber \\
L^{\sharp\sharp}_y & = \{ x\in V_{K'} : xM_{1,y} \subset M_y \}. \nonumber
\end{align*}
The action of $p^\Z$ on $\RZc(k')$ rescales the lattices $M_y$ and $M_{1,y}$, and hence 
the three lattices defined above depend only on the image of $y$ in $p^\Z \backslash \RZc(k')$.

\begin{proposition}\label{prop:special lattice bijection}
For every $y\in  \RZc(k')$ the lattice $L_y$ is special (in the sense of \S \ref{ss:special lattices}),  and satisfies
$
\Phi_*(L_y) = L_y^\sharp
$  
and  $L_y + L^\sharp_y= L_y^{\sharp\sharp} . $  Moreover,  $y\mapsto L_y$ establishes  bijections
\begin{align*}
p^\Z \backslash \RZc(k') & \iso \{ \mbox{special lattices  } L \subset  V_{K'} \} \\
p^\Z \backslash \RZc_\Lambda(k')  & \iso  \{ \mbox{special lattices }L \subset V_{K'}  :  \Lambda^\vee_{W'} \subset L  \subset \Lambda_{W'} \}.
\end{align*}
\end{proposition}

\begin{proof}
As in \S\ref{par:ADL} we have the refined affine Deligne-Lusztig set 
\[
X_{G,b,\mu^\sigma, \sigma}(k')=\big\{ g\in G(K') : g^{-1} b\sigma(g) \mu^\sigma(p)^{-1} \in G(W') \big\} / Q(W'),
\]
 where
\[
Q(W') = G(W') \cap \mu (p^{-1}) G(W') \mu(p).
\]
Recalling the action  $G \to \SO(V)$  defined by  $g\bullet v = gvg^{-1}$, for each $g\in X_{G,b,\mu^\sigma, \sigma}(k')$
 define  self-dual $W'$-lattices 
 \[
 L_g^\sharp = g\action V_{W'}\quad\mbox{and}\quad L_g = g\mu(p^{-1} ) \bullet V_{W'}.
 \]
As the action of $p \bullet$ is trivial, these lattices  depend only on the image of $g$  modulo $p^\Z$.

First we show that  $g\mapsto L_g$ establishes a bijection
\[
p^\Z \backslash X_{G,b,\mu^\sigma, \sigma}(k') \iso \{ \mbox{special lattices in } V_{K'} \}.
\]
Given a $g\in  X_{G,b,\mu^\sigma, \sigma}(k')$,  Lemma \ref{lem:tech2}  (which holds with $W$ replaced by $W'$ throughout) implies   
\[
( L_g + L_g^\sharp)/L_g\iso W' / pW'.
\] 
  Moreover,  $g^{-1} b\sigma(g) \mu^\sigma(p^{-1}) \in G(W')$ implies
\begin{equation}\label{special shift}
\Phi_*(L_g) = L_g^\sharp,
\end{equation}
and so   $L_g$ is special.  
To prove  injectivity,  assume $L_g=L_h$.   Applying $\Phi_*$ to both sides and using  (\ref{special shift}) shows that 
$L_g^\sharp=  L_h^\sharp$.    It follows that $h^{-1} g$ lies in the   intersection in $G(K')$
of the stabilizers of $V_{W'}$ and $\mu(p^{-1})\action V_{W'}$,  which is $p^\Z Q(W')$.  Thus
  $g=h$ in $p^\Z \backslash X_{G,b,\mu^\sigma, \sigma}(k')$.
   For surjectivity, suppose  $L$ is a special lattice.   
 Lemma \ref{lem:tech2}  implies the existence of a $g\in G(K')$ such that 
\[
(  \Phi_*(L) , L )= (  g \action V_{W'} ,  g\mu( p^{-1}) \action V_{W'} ).
\]
This equality implies that $g^{-1} b \sigma(g)\mu^\sigma(p^{-1})$ stabilizes $V_{W'}$, and so 
lies in $p^\Z G(W')$.  The relation 
$
b\in G(W) \mu^\sigma(p) G(W)
$
of Proposition \ref{prop:local GSpin datum}  implies that  
\[
 \eta_G\big( g^{-1} b \sigma(g)\mu^\sigma(p^{-1}) \big) \in (W')^\times,
\]
 and so in fact $g^{-1} b \sigma(g)\mu^\sigma(p^{-1})\in G(W')$.
Thus we have found a  $g\in  X_{G,b,\mu^\sigma, \sigma}(k')$ with  $L=L_g$.

By Corollary \ref{cor:points},  there is bijection
$
\RZc(k') \iso X_{G,b,\mu^\sigma, \sigma}(k'),
$
defined by sending  the point $y\in \RZc(k')$ to the unique  $g\in X_{G,b,\mu^\sigma, \sigma}(k')$ satisfying both
\[
M_y=g\cdot D_{W'}\quad\mbox{and}\quad M_{1,y} = g\cdot p\mu(p^{-1}) D_{W'}.
\]
Assuming that $y$ and $g$ are related in this way, we claim that  
\begin{equation}\label{lattices match}
(L_y^\sharp , L_y) = (  L_g^\sharp , L_g).
\end{equation}
To prove this,  let $B =  \{ x\in V_{K'} : x D_{W'} \subset D_{W'}\}$.
The inclusion $V_{W'} \subset B$ is obvious.  For the other inclusion
note that any $x\in B$ must have $Q(x) = x\circ x \in W'$, and so 
$V_{W'} \subset B\subset B^\vee \subset ( V_{W'})^\vee.$  The self-duality of $V_{W'}$ implies that equality 
 holds throughout, and so
\[
V_{W'} = \{ x\in V_{K'} : x D_{W'} \subset D_{W'}\}.
\]
Applying $g \action$ to both sides of this equality proves  $L_y^\sharp=L_g^\sharp$, while applying  $g\mu(p^{-1}) \action$  to
both sides proves  $L_y=L_g$.

We have now  established  bijections
\[
p^\Z \backslash \RZc(k') \iso  p^\Z \backslash X_{G,b,\mu^\sigma, \sigma}(k') \iso  \{ \mbox{special lattices  } L \subset  V_{K'} \}.
\]
The relation  $\Phi_*( L_y^\sharp ) =  L_y$  follows from  (\ref{special shift}) and (\ref{lattices match}).  
We verify $L_y + L_y^\sharp   = L_y^{\sharp\sharp}$ as follows:
Using the calculations of \S \ref{sss:mu basis},  one can show 
\begin{equation}\label{special stabilizer}
 \mu(p^{-1}) \action V_{W'} + V_{W'} = \{ x\in V_{K'} :  x   \mu(p^{-1})  D_{W'} \subset   D_{W'} \}.
\end{equation}
If $y\in \RZc(k')$ corresponds to  $g\in X_{G,b,\mu^\sigma, \sigma}(k')$ under the bijection above, 
then applying $g\action$  to both sides of (\ref{special stabilizer})  yields
\begin{align*}
L_y^\sharp+ L_y  = L_g^\sharp+ L_g 
& = \{ x\in V_{K'} :   (g^{-1}  x   g) \mu(p^{-1}) \cdot D_{W'}  \subset    D_{W'} \} \\
& = \{ x\in V_{K'} :   x  M_y^1 \subset    M_y \} \\
& = L_y^{\sharp\sharp}.
\end{align*}

Finally, a point $y \in p^\Z \backslash \RZc(k')$ lies in  the subset  $p^\Z \backslash \RZc_\Lambda(k')$ if and only if the quasi-endomorphisms
$\Lambda^\vee \subset \End (D_{K'})$ stabilize both lattices $M_{1,y} \subset M_y$.  This is  equivalent to the condition
$\Lambda^\vee \subset L_y \cap L_y^\sharp$, and so 
\begin{align*}
p^\Z \backslash \RZc_\Lambda (k') & \iso  \{ \mbox{special lattices  } L \subset  V_{K'}  :  \Lambda^\vee \subset L \cap \Phi_*(L)\} \\
& =  \{ \mbox{special lattices  } L \subset  V_{K'}  :  \Lambda^\vee \subset L \} \\
& = \{ \mbox{special lattices  } L \subset  V_{K'}  :  \Lambda^\vee_{W'} \subset L \subset \Lambda_{W'} \} .
\end{align*}
Here we have used first the fact that all elements of $\Lambda^\vee$ are fixed by $\Phi$, and then the fact that special lattices are  self-dual.
This completes the proof of Proposition \ref{prop:special lattice bijection}.
\end{proof}

\begin{corollary}\label{cor:vertex cover}
We have
\[
\RZc (k) = \bigcup_{ \substack{  \Lambda \\ t_\Lambda = t_\mathrm{max} }} \RZc_\Lambda (k).
\]
\end{corollary}

\begin{proof}
Suppose $y\in \RZc(k)$.  Let $L_y \subset V_K$ be the corresponding 
special lattice of  Proposition  \ref{prop:special lattice bijection}, and let  $\Lambda(L_y)$ be the vertex lattice of   
Proposition \ref{prop:optimal vertex}.  By  Proposition \ref{prop:possible types} there is a vertex lattice 
$\Lambda \supset \Lambda(L_y)$ with $t_\Lambda=t_\mathrm{max}$, and clearly
\[
\Lambda^\vee \subset \Lambda(L_y)^\vee = \{ x\in L_y : \Phi(x) =x \}  \subset L_y.
\]
The self-duality of $L_y$   implies $\Lambda_W^\vee \subset L_y \subset \Lambda_W$, and so
$y\in \RZc_\Lambda(k)$.
\end{proof}

\begin{corollary}\label{cor:stratum intersection}
For any vertex lattices $\Lambda_1$ and $\Lambda_2$ we have 
\[
\RZc_{\Lambda_1} (k) \cap \RZc_{\Lambda_2}(k)
= \begin{cases}
\RZc_{\Lambda_1 \cap \Lambda_2 }(k) & \mbox{if $\Lambda_1\cap\Lambda_2$ is a vertex lattice} \\
\emptyset & \mbox{otherwise.}
\end{cases}
\]
\end{corollary}

\begin{proof}
The proof is the same as \cite[Proposition 4.3(ii)]{RTW}.
\end{proof}

%%%%%%%%%%%%%%%%%%%%%%%%%%%%%%%%%%%%

\subsection{Comparison of $\RZc_\Lambda$ and $S_\Lambda$}

%%%%%%%%%%%%%%%%%%%%%%%%%%%%%%%%%%%%

Fix a vertex lattice $\Lambda\subset V_K^\Phi$.
Comparing Proposition \ref{prop:special lattice bijection} with the bijection of  \S \ref{sss:Xlambda} yields bijections
\[
p^\Z \backslash \RZc_\Lambda(k) \iso   \{ \mbox{special lattices } L\subset V_K :   \Lambda_W^\vee \subset L \subset \Lambda_W \} \iso S_\Lambda(k),
\]
and similarly for any finitely generated field  extension $k'/k$.

\begin{theorem}\label{thm:main isomorphism}
Let $\RZc_\Lambda^\red$ be the reduced $k$-scheme underlying $\RZc_\Lambda$.
There is a unique isomorphism of $k$-schemes  
\[
p^\Z \backslash \RZc_\Lambda^\red \iso S_\Lambda
\]
inducing  the above bijection on $k$-points.
\end{theorem}

\begin{proof}
First we construct a morphism $\RZc^\red_\Lambda \to S_\Lambda$.  Suppose we are given an 
$R$-point $y\in \RZc^\red_\Lambda(R)$ for some  reduced $k$-algebra $R$
of finite type.  Pulling back the universal object of \S \ref{sss:universal object} yields  a triple $(X_y   , \rho_y , \lambda_y )$ over $R$
in which $X_y$ is a $p$-divisible group, $\lambda_y$ is a principal polarization,  and  $\rho_y : X_{0/R}  \dashrightarrow X_y$ is a quasi-isogeny. 
Moreover, $x\mapsto \rho \circ x \circ \rho^{-1}$ defines a $\Z_p$-module map
\[
\Lambda^\vee \to \rho \circ \Lambda^\vee \circ \rho^{-1} \subset  \End (X_y).
\]

Let $\mathscr{D}_y=\Db(X_y)(R)$ be the  contravariant crystal of 
$X_y$ evaluated at the trivial divided power thickening  $R\to R$,  and let $\Fil^1(\mathscr{D}_y) \subset \mathscr{D}_y$ be the Hodge filtration.  
The locally free $R$-modules  $ \Fil^1(\mathscr{D}_y) \subset \mathscr{D}_y$  depend functorially on  $X_y$,
and so $\Lambda^\vee \to \End(X_y)$  induces  $R$-module maps
\[
\phi^\sharp : (\Lambda^\vee / p \Lambda^\vee) \otimes_{\F_p} R \to \End_R(\mathscr{D}_y)
\]
and
\[
\phi^{\sharp\sharp} : (\Lambda^\vee / p \Lambda^\vee) \otimes_{\F_p} R \to \End_R(\mathscr{D}_y^1)
\]
with $\ker(\phi^\sharp) \subset \ker(\phi^{\sharp\sharp})$.

The bilinear form on $\Lambda^\vee$ induces an $R$-valued bilinear form on 
$(\Lambda^\vee / p \Lambda^\vee)  \otimes_{\F_p} R$,
and any $x_1,x_2\in (\Lambda^\vee / p \Lambda^\vee)  \otimes_{\F_p} R$ satisfy 
\[
x_1\circ x_2+x_2\circ x_1 = [x_1,x_2]\in R 
\]
 as endomorphisms of $\mathscr{D}_y$.
In particular,  if   $x_1\in \ker(\phi^{\sharp\sharp})$ then $[x_1 , x_2]=0$, as the value of the scalar $[x_1,x_2]$
can be computed from its action  on $\Fil^1 ( \mathscr{D}_y)$, which is obviously trivial.
This shows that $\ker(\phi^{\sharp\sharp})$ is contained in the radical of the quadratic space $(\Lambda^\vee / p \Lambda^\vee)  \otimes_{\F_p} R$,
which is $(p\Lambda / p \Lambda^\vee)  \otimes_{\F_p} R$.  Recalling the $k$-quadratic space $\Omega=(\Lambda/\Lambda^\vee)\otimes_{\F_p} k$
from \S \ref{sss:Xlambda}, let 
\[
\mathscr{L}_y^\sharp \subset \mathscr{L}^{\sharp\sharp}_y\subset  \Omega\otimes_k R
\]
 be the images of $\ker(\phi^\sharp) \subset \ker(\phi^{\sharp\sharp})$ under the isomorphism
\[
(p\Lambda / p \Lambda^\vee)  \otimes_{\F_p} R \map{ p^{-1} \otimes\mathrm{id} }   \Omega\otimes_k R.
\]

Suppose for the moment that $R=k$.  Recalling from \S \ref{sss:the specials} (with $k'=k$) the $W$-modules 
$M_{1,y}\subset M_y$, there is an isomorphism $\mathscr{D}_y \iso M_y/pM_y$ identifying
$\Fil^1(\mathscr{D}_y) \iso M_{1,y} /pM_y$.  The subspaces
\[
\mathscr{L}_y^\sharp \subset \mathscr{L}^{\sharp\sharp}_y \subset  \Omega \iso (\Lambda_W / \Lambda^\vee_W)
\]
correspond to lattices $\Lambda_W^\vee \subset L_y^\sharp \subset L^{\sharp\sharp}_y \subset \Lambda_W$,
and tracing through the  definitions shows that these are none other than the lattices
\[
L_y^\sharp = \{ x\in V_K : x M_y = M_y \} \quad\mbox{and}\quad
L^{\sharp\sharp}_y  = \{ x\in V_K : x M_{1,y} \subset M_y \}
\]
appearing in \S \ref{sss:the specials}.
Comparison with Proposition \ref{prop:special lattice bijection} shows that 
$
L^{\sharp\sharp}_y = L_y + L_y^\sharp,
$
where $L_y$ is the special lattice
\[
L_y = \{x \in V_K : x M_y^1 \subset M_y^1 \}
\]
satisfying  $\Phi(L_y) =L_y^\sharp  $.    Noting that $\Lambda_W^\vee  \subset L_y \subset \Lambda_W,$
 we denote by  $\mathscr{L}_y \subset \Omega$ the $k$-subspace corresponding to $L_y$.

 The self-duality of  the $W$-lattices $L_y$ and $L_y^\sharp$ implies that the  corresponding $k$-subspaces 
$\mathscr{L}_y$ and $\mathscr{L}_y^\sharp$ of   $\Omega$ are maximal isotropic,  and so have dimension 
$t_\Lambda/2$.  The specialness of $L_y$ also implies that  $\mathscr{L}^{\sharp\sharp}_y=\mathscr{L}_y+ \mathscr{L}_y^\sharp$ 
has dimension $(t_\Lambda/2)+1$.   
 It follows  that $( \mathscr{L}^{\sharp\sharp}_y )^\perp \subset \mathscr{L}^{\sharp\sharp}_y$ with codimension $2$, 
 and that the quotient  $\mathscr{L}^{\sharp\sharp}_y / (\mathscr{L}^{\sharp\sharp}_y)^\perp$ is a hyperbolic plane over $k$.
The subspaces $\mathscr{L}_y/ (\mathscr{L}^{\sharp\sharp}_y)^\perp$ and 
$\mathscr{L}_y^\sharp/ (  \mathscr{L}^{\sharp\sharp}_y)^\perp$ are its unique isotropic lines.

Now  return to a general reduced $R$ of finite type.   The submodule
$
\mathscr{L}^\sharp_y\subset \Omega\otimes_k R
$ 
is a totally isotropic local direct summand of rank $t_\Lambda/2$, and 
$
\mathscr{L}^{\sharp\sharp}_y \subset \Omega\otimes_k R
$ 
is a local direct summand of rank $(t_\Lambda/2)+1$.  Indeed, by
 \cite[Exercise X.16]{LangAlgebra} it suffices to check these properties fiber-by-fiber at the closed points of $\Spec(R)$,
which is precisely what we did in the $R=k$ case above.
 
By similar reasoning the quotient $\mathscr{L}^{\sharp\sharp}_y/  ( \mathscr{L}^{\sharp\sharp}_y)^\perp$ is a hyperbolic plane over $R$, and so
contains exactly two isotropic local direct summands of rank one.  One of them is $\mathscr{L}^\sharp_y/ (\mathscr{L}^{\sharp\sharp}_y)^\perp$,
and the other has the form  $\mathscr{L}_y/ (\mathscr{L}_y^{\sharp\sharp})^\perp$ for a uniquely determined Lagrangian 
$\mathscr{L}_y \subset \Omega\otimes_k R$.
 By again reducing to the  $R=k$ case  treated above, we see that  $\Phi(\mathscr{L}_y) = \mathscr{L}_y^\sharp$,  and so 
 \[
 \mathscr{L}_y + \Phi(\mathscr{L}_y)= \mathscr{L}_y +  \mathscr{L}_y^\sharp  = \mathscr{L}_y^{\sharp\sharp} 
 \] 
 is a local direct summand of rank $(t_\Lambda/2)+1$.  In other words,  $\mathscr{L}_y\in S_\Lambda(R)$.  

The $k$-scheme  $\RZc^\red_\Lambda$ is itself reduced and locally of finite type, and so the rule $y\mapsto \mathscr{L}_y$ 
defines (at last) the promised  morphism 
$\RZc^\red_\Lambda \to S_\Lambda$.  It is clear from the construction that the morphism descends to 
\begin{equation}\label{stratum iso}
p^\Z \backslash \RZc^\red_\Lambda \to S_\Lambda
\end{equation}
and  induces the desired bijection  on $k$-points.
In fact, the generality of Proposition \ref{prop:special lattice bijection}  shows that this morphism induces a bijection 
\[
p^\Z \backslash \RZc^\red_\Lambda(k') \iso S_\Lambda(k')
\]
for any extension field $k'/k$.  In particular (\ref{stratum iso}) is birational and quasi-finite. 
It is  a proper morphism, as  Proposition \ref{prop:projective} implies that 
$p^\Z \backslash \RZc^\red_\Lambda$ is projective.
The variety  $S_\Lambda$ is smooth by Proposition \ref{prop:lagrangian}, and so
Zariski's main theorem implies that (\ref{stratum iso}) is an isomorphism.
\end{proof}

Recall  from Proposition \ref{prop:lagrangian} that $S_\Lambda$ has two connected components.  
The two components are isomorphic, and are labelled (arbitrarily) as $S_\Lambda^+$ and $S_\Lambda^-$.

\begin{corollary}\label{cor:connected stratum}
The reduced scheme $\RZc_\Lambda^{(\ell),\red}$ underlying $\RZc_\Lambda^{(\ell)}$ is connected and nonempty, and is isomorphic to $S_\Lambda^\pm$.
\end{corollary}

\begin{proof}
The action of $p^\Z$ on $\RZc_\Lambda$ identifies $\RZc_\Lambda^{(\ell)} \iso \RZc_\Lambda^{(\ell +2 )}$, and so it suffices to 
assume $\ell \in \{0,1\}$.  Moreover, we know from Proposition \ref{prop:special lattice bijection} that
\[ 
\RZc_\Lambda^{(0),\red } \sqcup \RZc_\Lambda^{(1),\red }   \iso  p^\Z \backslash \RZc_\Lambda ^\red
\iso S_\Lambda^+ \sqcup S_\Lambda^-.
\]
These leaves two possibilities:  either each of  $\RZc_\Lambda^{(0),\red }$ and  
$\RZc_\Lambda^{(1),\red }$  is connected and isomorphic to $S_\Lambda^\pm$, or one of them is empty 
and the other has two connected components.   To complete the proof of the corollary, it therefore suffices to show that
 $\RZc_\Lambda^{(0),\red }$ and  $\RZc_\Lambda^{(1),\red }$  are nonempty.

First suppose that $\Lambda$ has type $t_\Lambda=2$.  In this case one can easily check that $S_\Lambda$
consists of two points, and so the same is true of $p^\Z \backslash  \RZc_\Lambda^\red$.  There is a $W$-basis $e_1,\ldots, e_n$ 
of $\Lambda$ such that the matrix of the
bilinear form is
\[
\left( \begin{smallmatrix}
u_1 p^{-1} \\
& u_2 p^{-1} \\
& & u_3\, \\
& & & \ddots\,  \\
& & & & u_n\,
\end{smallmatrix}\right)
\]
for some $u_1,\ldots, u_n \in \Z_p^\times$.

Let  $r_i\in \mathrm{O}(V_K^\Phi)(\Q_p)$ be the reflection with $e_i \mapsto -e_i$  and $e_j \mapsto e_j$ for all $j\neq i$.  
 The \emph{spinor norm} of $r_1 r_3$, in the sense of  \cite{Kitaoka}, is 
 \[
 \frac{ u_1 }{2p} \cdot  \frac{ u_3 }{2} =Q(e_1)Q(e_3) \in \Q_p^\times/(\Q_p^\times)^2.
 \]
The spinor norm of \cite{Kitaoka} is compatible with the spinor similitude of \S \ref{sss:G}, in the sense that 
any lift of  $r_1r_3\in \SO(V_K^\Phi)(\Q_p)$ to 
\[
g\in \GSpin(V_K^\Phi)(\Q_p) \iso J_b(\Q_p)  
\] 
satisfies  $\eta_b(g) = Q(e_1)Q(e_3)$ up to scaling by $(\Q_p^\times)^2$.  Thus
\[
\ord_p( \eta_b (g)) \equiv 1\pmod{2}.
\]  
By this calculation and the discussion of   \S \ref{sss:J action}, $g$ acts on 
\[
p^\Z \backslash \RZc_\Lambda^\red   \iso \RZc_\Lambda^{(0) , \red }  \sqcup \RZc_\Lambda^{(1) ,\red}  ,
\]
and  interchanges the two subsets on the right.  Thus each is nonempty, and in fact each is a single reduced point.

For general $\Lambda$, Proposition  \ref{prop:all types} allows us to pick a type $2$ vertex lattice  
$\Lambda_2 \subset \Lambda$.   Combining Corollary \ref{cor:stratum intersection} with the paragraph above shows that 
$
\emptyset\neq\RZc^{(\ell) , \red } _{\Lambda_2} \subset \RZc^{(\ell) , \red }_\Lambda.
$
\end{proof}

%%%%%%%%%%%%%%%%%%%%%%%%%%%%%%%%%%%%

\subsection{The main result}
\label{ss:mainRZGSPin}

%%%%%%%%%%%%%%%%%%%%%%%%%%%%%%%%%%%%

%\subsubsection{}
We can now prove our main result on the structure of the reduced  scheme 
\[
\RZc^\red =  \bigsqcup_{\ell\in \Z} \RZc^{(\ell), \red}.
\]
For each  vertex lattice $\Lambda$,  recall that  $\RZc_\Lambda^{(\ell) , \red}$ is  the  reduced $k$-scheme underlying the formal $W$-scheme
$
\RZc_\Lambda^{(\ell)} = \RZc_\Lambda \cap \RZc^{(\ell)}.
$

\begin{theorem}\label{thm:final}
For each $\ell$ the $k$-scheme $\RZc^{(\ell),\red}$  is connected.  Each closed subscheme
$\RZc^{(\ell),\red}_\Lambda$ is projective and smooth of  dimension 
$(t_\Lambda/2) - 1$, and is isomorphic to $S_\Lambda^\pm$. The  irreducible  components of $\RZc^{(\ell),\red}$
are precisely the closed subschemes $\RZc_\Lambda^{(\ell),\red}$  as $\Lambda$ runs over the 
vertex lattices of  maximal type $t_\Lambda=t_\mathrm{max}$, and in particular $\RZc^\red$ is equidimensional with
\[
\dim(\RZc^\red) = \frac{1}{2}
\begin{cases}
n-4 & \mbox{if $n$ is even and $\det(V_{\Q_p})= (-1)^{\frac{n}{2}} $} \\
n-3 & \mbox{if $n$ is odd } \\
n-2& \mbox{if $n$ is even and $\det( V_{\Q_p} ) \neq (-1)^{\frac{n}{2}} $}.
\end{cases}
\]
\end{theorem}

\begin{proof}
For any vertex lattice $\Lambda \subset V_K^\Phi$,  Corollary \ref{cor:connected stratum} and Proposition \ref{prop:lagrangian}
tell us that \[ \RZc_\Lambda^{( \ell ),\red}\iso S_\Lambda^\pm\] is  
irreducible, projective, and smooth of dimension $(t_\Lambda/2) - 1$.

Corollary  \ref{cor:vertex cover} implies that 
\begin{equation}\label{final cover}
\RZc^{(\ell),\red} = \bigcup_{ \substack{ \Lambda \\ t_\Lambda = t_\mathrm{max} } } \RZc^{(\ell),\red}_\Lambda,
\end{equation}
and so the irreducible components of $\RZc^{(\ell),\red}$ are precisely the 
$\RZc^{(\ell),\red}_\Lambda$ with $\Lambda$ of maximal type $t_\mathrm{max}$.  This proves all parts of the claim,
except for the connectedness of $\RZc^{(\ell),\red}$.

Suppose that $\Lambda_1\sim \Lambda_2$ are adjacent vertex lattices.  If 
$\Lambda_1\subset \Lambda_2$ then Corollary \ref{cor:stratum intersection} implies that 
$ \RZc^{(\ell),\red}_{\Lambda_1}$ and  $\RZc^{(\ell),\red}_{\Lambda_2}$ 
lie on the same connected component of 
$\RZc^{(\ell)}$.  Of course similar remarks hold if $\Lambda_2\subset \Lambda_1$.
Proposition \ref{prop:connected lattices} shows that any two vertex lattices are connected by a chain of adjacent 
vertex lattices, and so all of the closed subschemes $\RZc_\Lambda^{(\ell),\red}$ 
lie on  the  same connected component of $\RZc^{(\ell),\red}$.  
The equality (\ref{final cover}) now shows that  $\RZc^{(\ell),\red}$
is connected.
\end{proof}

\begin{remark}
 When $n$ is odd the center of $C^+(V)$ is $\Z_p$. 
 When $n$ is even the center of  $C^+(V)$  is the maximal  $\Z_p$-order in $F=\Q_p[x]/( x^2 - \Delta)$, 
where $\Delta = (-1)^{\frac{n}{2}} \det(V_{\Q_p})$.  Thus for $n$ even
\[
\dim(\RZc^\red) = 
\begin{cases}
(n/2)-2  & \mbox{if }F\iso \Q_p\times\Q_p \\
(n/2)-1  & \mbox{if }F\iso \Q_{p^2}.
\end{cases}
\]
\end{remark}

\begin{remark}%\label{conjrapoport}
The dimension formula of Theorem \ref{thm:final} verifies a case of a conjecture of Chai and of Rapoport \cite{GHKR,RapoportGuide}. 
According to this conjecture, we should have 
\[
\dim(\RZc^\red) = \langle\rho, \mu-\nu_b\rangle-\frac{1}{2}{\rm def}_G(b).
\]
Here, $\mu$ is assumed to be a dominant representative of the conjugacy class $\{\mu\}$, $\rho$ is the half sum of all absolute positive roots of $G$ and by definition, 
\[{\rm def}_G(b)={\rm rank}_{\Q_p}(G)-{\rm rank}_{\Q_p}(J_b).\] 
In our case,
$ \langle \rho, \mu-\nu_b\rangle= \langle \rho, \mu\rangle=(n-2)/2$, while
we have
${\rm def}_G(b)=2$, $1$, or  $0$, in the three cases
listed in the Theorem (in that order).
Indeed, ${\rm def}_G(b)$ is the difference between the Witt indices of $V_{\Q_p}$ and $V^\Phi_K$
and this can be determined as in the proof of Proposition \ref{prop:possible types}. The above dimension formula
has recently been shown, for the all (unramified) Rapoport-Zink spaces of Hodge type defined in this paper, 
 by Hamacher \cite{HamacherProduct1} and by Zhang \cite{ZhangCao}.  
\end{remark}

%%%%%%%%%%%%%%%%%%%%%%%%%%%%%%

\subsection{The Bruhat-Tits stratification}
\label{ss:BT}

%%%%%%%%%%%%%%%%%%%%%%%%%%%%%%

Using the collection of closed subschemes $\RZc_\Lambda^\red$ of $\RZc^\red$, we explain how to define a stratification of $\RZc^\red$,
in which each  stratum is the Deligne-Lusztig variety determined by a Coxeter element in a special orthogonal group over $\F_p$.

\subsubsection{}
Recall from Corollary \ref{cor:stratum intersection} that $\Lambda' \subset \Lambda$ implies 
$
\RZc_{\Lambda'}^\red  \subset \RZc_\Lambda^\red.
$
For each vertex lattice $\Lambda$ define   the \emph{Bruhat-Tits stratum} 
\[
\mathrm{BT}_\Lambda = \RZc_\Lambda^\red \smallsetminus \bigcup_{ \Lambda'\subsetneq \Lambda } \RZc_{\Lambda'}^\red .
\]
It is an open and dense subscheme  of  $\RZc^\red_\Lambda$, and 
\[
 \RZc_\Lambda^\red = \bigcup_{\Lambda'\subset \Lambda} \mathrm{BT}_\Lambda
\] 
defines a  stratification of $ \RZc_\Lambda^\red $ as a disjoint union of  locally closed subschemes.

\subsubsection{}
\label{sss:BTstratEO}
Similarly  
\[
 \RZc^\red  = \bigcup_{\mathrm{all\ }\Lambda} \mathrm{BT}_\Lambda
\]
defines a stratification of  $\RZc^\red$ as a disjoint union of  locally closed subschemes. 
This is the $\GSpin$ analogue of the Bruhat-Tits stratification for unitary Rapoport-Zink spaces found in
\cite{VollaardWedhorn} and \cite{RTW}.  However,   this terminology should be taken  with a grain of salt: unlike in [\emph{loc.~cit.}] the strata here are not in bijection with the vertices in the Bruhat-Tits building of the group $J^\mathrm{der}_b$.  See \cite[\S 2.7]{HP} for more details in the special case $n=6$.

\subsubsection{}\label{sss:BT points}
For a special lattice $L \subset V_K$, recall from Proposition \ref{prop:optimal vertex} the vertex lattice  $\Lambda(L)$ characterized by
\[
\Lambda(L)^\vee =\{ x\in   L : \Phi(x)=x \}.
\]
If we rewrite the bijections  of  Theorem \ref{thm:main isomorphism} and   \S \ref{sss:Xlambda} as
\begin{align*}
p^\Z \backslash \RZc _\Lambda^\red(k) & \iso S_\Lambda (k)  \\
& \iso  \{ \mbox{special lattices }L \subset V_K  : \Lambda^\vee  \subset L \}  \\
& =  \{ \mbox{special lattices }L  \subset V_K : \Lambda(L) \subset \Lambda \},
\end{align*}
the inclusion $ \mathrm{BT}_\Lambda \subset \RZc_\Lambda^\red$ identifies 
\begin{align*}
p^\Z \backslash \mathrm{BT}_\Lambda (k) 
& \iso \{ \mbox{special lattices }L \subset V_K    : \Lambda(L)=\Lambda \} \\
& = \{ \mbox{special lattices }L  \subset V_K   : L + \Phi(L) + \cdots + \Phi^d(L) =\Lambda_W \} .
\end{align*}

\subsubsection{}\label{DLsection}
Fix a vertex lattice $\Lambda$ of type $t_\Lambda=2d$, and recall from  \S \ref{sss:Xlambda}
the  $2d$-dimensional  $\F_p$-quadratic space $\Omega_0= \Lambda/\Lambda^\vee$.
Set 
\[
\Omega=\Omega_0\otimes_{\F_p} k \iso \Lambda_W / \Lambda_W^\vee,
\] 
and let $\Phi=\mathrm{id}\otimes \sigma$ be the Frobenius on $\Omega$. 

We recall the set-up of  \cite[\S 3.2]{HP}.
Fix a basis $\{e_1,\ldots,e_d , f_1,\ldots, f_d\}$ of $\Omega$ in such a way that $\mathrm{Span}_k\{ e_1,\ldots, e_d\}$
and $\mathrm{Span}_k\{ f_1,\ldots, f_d\}$ are totally isotropic, $[ e_i ,f_j] =\delta_{i,j}$, and the Frobenius $\Phi$
fixes $e_1,\ldots, e_{d-1}, f_1,\ldots, f_{d-1}$ but interchanges $e_d \leftrightarrow f_d$.
This choice of basis  determines a maximal $\Phi$-stable torus $T \subset \SO(\Omega)$. 

The isotropic flags  $\FF^+_\bullet$ and $\FF^-_\bullet$ in $\Omega$ defined by
\begin{align*}
\FF_i^\pm & = \mathrm{Span}_k\{ e_1, \ldots , e_i \} \mbox{ for } 1\leq i\leq d-1 \\
\FF_d^+  & = \mathrm{Span}_k\{  e_1, \ldots, e_{d-1}, e_d \} \nonumber \\
\FF_d^-  & =  \mathrm{Span}_k\{   e_1, \ldots, e_{d-1}, f_d \}  \nonumber .
\end{align*}
satisfy $\FF^\pm_\bullet=\Phi(\FF^\mp_\bullet) $, and have   the same stabilizer  $B\subset \SO(\Omega)$.  
It is a $\Phi$-stable Borel subgroup containing $T$.
The corresponding set of simple reflections in the Weyl group $W=N(T)/T$ is 
$
\{s_1, \ldots, s_{d-2}, t^+, t^-\}
$
where
\begin{itemize}
\item $s_i$  interchanges $e_i \leftrightarrow e_{i+1}$ and  $f_i \leftrightarrow f_{i+1}$, and fixes the other basis elements.

\item $t^+$ interchanges $e_{d-1} \leftrightarrow e_d$ and $f_{d-1} \leftrightarrow f_d$, and fixes the other basis elements.

\item $t^-$ interchanges $e_{d-1} \leftrightarrow f_d$ and  $f_{d-1} \leftrightarrow e_d$, and fixes the other basis elements.
\end{itemize}
Notice that  $\Phi(s_i)=s_i$, and $\Phi(t^\pm)=t^{\mp}$, and so the products
\[
w^\pm=t^{\mp}s_{d-2}\cdots s_2s_1 \in W
\] 
 are  \emph{Coxeter elements}: products of exactly one representative from each $\Phi$-orbit 
in the set of simple reflections above.

\subsubsection{}
The \emph{Deligne-Lusztig variety}
\[
X_B(w^\pm)  = \{   g\in \SO(\Omega) /B : \mathrm{inv}( g, \Phi(g) ) = w^\pm   \} 
\]
is a smooth quasi-projective $k$-variety of dimension $d-1$.  Here  $\mathrm{inv}$ is the relative position invariant 
\[
\SO(\Omega) /B \times \SO(\Omega) /B  \map{(g_1,g_2) \mapsto g_1^{-1} g_2}   B\backslash \SO(\Omega) /B \iso W.
\]

\begin{theorem}
%\label{thm:coxeterDL}
There are isomorphisms $X_B(w^+) \iso X_B(w^-)$, and 
\[
p^\Z \backslash \mathrm{BT}_\Lambda \iso X_B(w^+) \sqcup X_B(w^-).
\]
\end{theorem}

\begin{proof}
Recall that the $k$-variety
\[
S_\Lambda (k)= \{ \mbox{Lagrangians }\mathscr{L} \subset \Omega : \dim_k( \mathscr{L}+\Phi(\mathscr{L}) ) = d +1 \}.
\]
has two connected components $X^+_\Lambda$ and $X^-_\Lambda $, interchanged
by the action of any $g\in \mathrm{O}(\Omega_0)(\F_p)$ with $\det(g)=-1$.  

  After possibly relabeling $\FF^+_\bullet$ and $\FF^-_\bullet$,  \cite[Proposition 3.8]{HP} gives an open immersion
$X_B(w^\pm) \to S_\Lambda^\pm$ defined by $g\mapsto g\FF_d^\pm$.    Thus 
\[
X_B(w^+) \sqcup X_B(w^-) \subset S_\Lambda
\]
as an open subset with $k$-points
\begin{align*}
X_B(w^\pm)(k)  
& = 
\{ \mathscr{L} \in S_\Lambda^\pm (k) : \mathscr{L} \cap  \Phi(\mathscr{L}) \cap \Phi^2(\mathscr{L}) \cap  \cdots \cap \Phi^d(\mathscr{L}) = 0 \} \\
& = 
\{ \mathscr{L} \in S_\Lambda^\pm (k) : \mathscr{L} + \Phi(\mathscr{L}) +\Phi^2(\mathscr{L}) + \cdots + \Phi^d(\mathscr{L}) = \Omega \} .
\end{align*}
The action of any $g$ as above interchanges $X_B(w^+)$ with $X_B(w^-)$.

By  \S \ref{sss:BT points}, we have bijections
\begin{align*}
p^\Z \backslash \mathrm{BT}_\Lambda(k)  & \iso 
\{ \mathscr{L} \in S_\Lambda (k) : \mathscr{L} + \Phi(\mathscr{L}) +\Phi^2(\mathscr{L}) + \cdots + \Phi^d(\mathscr{L}) = \Omega \} \\
& \iso X_B(w^+)(k)  \sqcup X_B(w^-) (k).
\end{align*}
This is nothing more than the restriction of the isomorphism 
$p^\Z \backslash \RZc^\red_\Lambda \iso S_\Lambda$ of Theorem \ref{thm:main isomorphism},
and hence arises from an isomorphism of varieties
\[
p^\Z \backslash \mathrm{BT}_\Lambda \iso X_B(w^+)   \sqcup X_B(w^-) .
\]
\end{proof}

\begin{remark}
The quotient $p^\Z \backslash \RZc_\Lambda^\red$ is itself isomorphic to a  disjoint union of two Deligne-Lusztig varieties.
Indeed, if  $P^\pm \subset \SO(\Omega)$ denotes the maximal parabolic subgroup stabilizing $\FF_d^\pm$, then 
\cite[Proposition 3.6]{HP} shows that  
\[
p^\Z \backslash \RZc_\Lambda \iso S_\Lambda \iso X_{P^+}(1) \sqcup X_{P^-}(1).
\]
\end{remark}

%%%%%%%%%%%%%%%%%%%%%%%%%%%%%%%%%%%%

\section{Shimura varieties for spinor similitude groups}
\label{s:global GSpin}

%%%%%%%%%%%%%%%%%%%%%%%%%%%%%%%%%%%%

Finally, we apply our results to study the supersingular loci of  Shimura varieties of type $\GSpin$.
Throughout \S \ref{s:global GSpin} we fix  a quadratic space $(V,Q)$ of signature $(n-2,2)$ over $\Z_{(p)}$.
We always assume that $n\ge 3$, and that the corresponding  bilinear form $[x,y]$  induces an isomorphism from $V$ to its $\Z_{(p)}$-linear dual.

%%%%%%%%%%%%%%%%%%%%%%%%%%%%%%

\subsection{The GSpin Shimura variety}

%%%%%%%%%%%%%%%%%%%%%%%%%%%%%%

First, we attach to the quadratic space $V$ a Shimura variety of Hodge type.

\subsubsection{}

As in the local set-up of \S \ref{sss:azumaya}, the Clifford algebra $C(V)$ is endowed with a $\Z/2\Z$-grading 
$
C(V)=C^+(V)\oplus C^-(V)
$
and a canonical involution $c\mapsto c^*$.
The group of spinor similitudes $G=\GSpin(V)$  is the reductive group over $\Z_{(p)}$ defined by
\[
G(R) = \{  g\in C^+(V_R)^\times : g V_R g^{-1} = V_R,\, g^* g \in R^\times \}
\]
for any $\Z_{(p)}$-algebra $R$.  As before, the spinor similitude $\eta_G: G \to \Gm$ is  defined by 
$\eta_G(g) =g^* g$, and there is a representation $G \to \SO(V)$ defined by  $g\action v = gvg^{-1}$.
By slight abuse of notation, we denote again by $G$ the generic fiber of the $\Z_{(p)}$-group scheme $G$ just defined.

\subsubsection{}

As in (\ref{hermitian domain}), define a hermitian symmetric domain 
\[
 \mathcal{H} = \{ z\in V_\C : [z,z]=0,\,  [z,\bar{z}]<0 \} /\C^\times
\]
 of dimension $n-2$. The group $G(\R)$ acts on $\mathcal{H}$ through the representation $G \to \SO(V)$,
and the action of any $g\in G(\R)$ with $\eta_G(g) < 0$ interchanges the two connected components of $\mathcal{H}$.

Writing   $z\in \mathcal{H}$  as  $z=u+iv$ with $u,v\in V_\R$, the subspace $\mathrm{Span}_\R\{u,v\}$
is a negative definite plane in $V_\R$, oriented by the ordered orthogonal basis $u,v$.  There are natural $\R$-algebra maps 
\[
\C \iso  C^+(\mathrm{Span}_\R\{u,v\}) \to C^+(V_\R).
\]
The first is  determined by 
\[
i\mapsto  \frac{uv} { \sqrt{ Q(u) Q(v)} },
\]
and the second is induced by  the inclusion $\mathrm{Span}_\R\{u,v\} \subset V_\R$.  The above composition restricts to an injection
$\bm{h}_z: \C^\times \to G(\R)$, which arises from a morphism $\bm{h}_z : \mathbb{S} \to G_\R$ of real algebraic groups.
Here  $\mathbb{S} = \mathrm{Res}_{\C/\R} \Gm$ is  Deligne's torus.  
The construction $z\mapsto \bm{h}_z$  realizes 
$
\mathcal{H} \subset \Hom(\mathbb{S} , G_\R)
$
as a $G(\R)$-conjugacy class.

Using the conventions of \cite{DeligneCorvallis}, the Hodge structure on $V$ determined by $\bm{h}_z$ is
\begin{equation}\label{gspin hodge}
V_\C^{(1,-1)} = \C z ,\quad
V_\C^{(0,0)} = ( \C z + \C \bar{z} )^\perp,\quad
V_\C^{(-1,1)} = \C \bar{z}.
\end{equation}

\subsubsection{}
\label{sss:ortho hodge}

The $\Z_{(p)}$-quadratic space $V$ admits an orthogonal basis, and so  one can choose  orthogonal vectors 
$e,f\in V$  of negative length with $Q(e), Q(f)  \in \Z_{(p)}^\times$.  If we set
$\delta = ef \in C(V)^\times$ then, exactly as in \S \ref{sss:symplectic rep}, $\delta$ determines  a perfect $G$-equivariant symplectic form 
\[
\psi_\delta : C(V) \otimes_{\Z_{(p)}} C(V) \to \Z_{(p)} (\eta_G),
\]
where $G$ acts  on $C(V)$ via left multiplication.  For any $z\in \mathcal{H}$ the bilinear form
$\psi_\delta( \bm{h}_z(i) c_1 , c_2)$ on $C(V_\R)$ is either positive definite or negative definite, depending on the 
connected component of $\mathcal{H}$ containing  $z$.

The Hodge structure on $C(V_\Q)$ determined by $\bm{h}_z \in \Hom(\mathbb{S} ,G_\R)$ is 
\[
C(V_\C)^{( 0,-1) } =  z C(V_\C) ,\quad
C(V_\C)^{ (-1,0) } = \bar{z}  C(V_\C) .
\]
From this it follows that the faithful representation
\[
G\to \GSp(C(V),\psi_\delta)
\] 
defines a morphism of Shimura data from $(G ,\mathcal{H})$ to the Siegel Shimura datum determined by 
the symplectic space $(C(V),\psi_\delta)$.

\subsubsection{}
Define a  hyperspecial subgroup  $U_p=G(\Z_p)$  of $G(\Q_p)$, and choose any sufficiently small compact open subgroup 
$U^p \subset G( \A_f^p)$.
Setting $U=U_p U^p$, there  is an associated Shimura variety $\Sh_{U} (   G , \mathcal{H}   )$ over $\Q$ with complex points 
\[
\Sh_{ U } (   G  , \mathcal{H}   ) (\C)  = G(\Q) \backslash \mathcal{H} \times G(\A_f) / U .
\]

Let $2g=\dim C(V_\Q) = 2^n$ so that, as in \S \ref{ss:integral models}, the morphism of Shimura data 
$(G,\mathcal{H}) \to (\GSp_{2g} ,\mathcal{H}_{2g})$ constructed in \S \ref{sss:ortho hodge} determines a morphism from 
$\Sh_{U} (   G , \mathcal{H}   )$ to a moduli space of polarized abelian varieties up to prime to $p$-isogeny.
Pulling back the universal object over this moduli space yields an abelian scheme up to prime to $p$-isogeny
\[
A \to \Sh_{U} (   G , \mathcal{H}   ),
\]
often called the \emph{Kuga-Satake abelian scheme}; see \cite{MadapusiSpin} for more information.

The fiber of the Kuga-Satake abelian scheme at a point  $(z,g) \in \mathcal{H} \times G(\A_f)$ can be made very explicit: 
it is the abelian variety up to  prime-to-$p$-isogeny $A_{( z ,g )}$   whose   Betti homology is the $\Z_{(p)}$-module
\[
{\rm H}_1( A_{( z,g)}  (\C) , \Z_{(p)} )  =  g  \cdot C(V) \subset C(V_\Q)
\]
 with  the  Hodge structure  $\bm{h}_z$ defined above.  Note that $A_{(z,g)}$  carries a prime-to-$p$ polarization $\lambda$ inherited from the 
 symplectic form $\psi_\delta$, and an action of $C(V)^{op}$ induced by the 
 right multiplication action of $C(V)$ on itself.    
 \bigskip
 \bigskip

%%%%%%%%%%%%%%%%%%%%%%%%%%%%%%

\subsection{Uniformization of the supersingular locus}

%%%%%%%%%%%%%%%%%%%%%%%%%%%%%%

As in \S \ref{ss:integral models}, 
let $\mathscr{S} =\mathscr{S}_U(G , \mathcal{H})$ be Kisin's \cite{KisinJAMS} smooth  integral model of 
$\Sh_{U } (G , \mathcal{H} )$ over  $\Z_{(p)}$, and let
\[
\mathscr{S}_{U_p} =  \varprojlim_{U^p} {\sS}_{U^pU_p}(G, \mathcal{H}).
\]   
By the very construction of the integral model,
 the Kuga-Satake abelian scheme extends to an abelian scheme up to prime to $p$-isogeny
$A \to \mathscr{S}$.

\subsubsection{}

We denote by 
\[
\mathscr{S}_{ss} \subset \mathscr{S} \otimes_{\Z_{(p)} } k
\]
 the supersingular locus:  the largest reduced closed subscheme  over which the Kuga-Satake abelian scheme is supersingular.  
The fiber of $A$ at any point of $\mathscr{S} \otimes_{\Z_{(p)} }k$  is supersingular if and only if 
its $p$-divisible group is isoclinic.  Thus  Lemma \ref{lem:basic criterion} implies that the supersingular locus is precisely the basic locus.
Moreover, along the supersingular locus the slope of the universal $p$-divisible group must by $1/2$, and so the classification of basic elements in Proposition \ref{prop:twisted space} tells us that $\mathscr{S}_{ss}$ must be the Newton stratum $\sS_b$
for the basic $b$ appearing in  Proposition  \ref{prop:local GSpin datum}.

Denote by $ (  \widehat{\mathscr{S}}_W )_{ /  \mathscr{S}_{ss}  }$ the formal completion of $\mathscr{S}_W$ along $\mathscr{S}_{ss}$.

\begin{lemma}\label{nonempty}
The supersingular locus $\mathscr{S}_{ss}$ is nonempty.
\end{lemma}

\begin{proof} This can  be understood as a special case of recent results on the non-emptiness of the basic locus in Hodge type Shimura varieties; see Remark \ref{rem:nonemptiness}.  Here we  give a direct argument.

Let $V_{0\Q} \subset V_\Q$ be a rational $2$-plane on which $Q$ is negative definite.  The $\Z/2\Z$-grading on the  Clifford algebra of $V_{0\Q}$
has the simple form $C(V_{0\Q}) = F \oplus V_{0\Q}$, where the even part is the quadratic imaginary field
\[
F= \Q\left(\sqrt{ - \mathrm{det}(V_{0\Q}) }  \right ).
\]  
We leave it as an exercise to the reader to check that one may  choose $V_{0\Q}$ so that $p$ is inert in $F$
(reduce to the case $n=3$, and use the classification of quadratic forms from \cite{SerreAcourse}).

 The action of $F$ by left multiplication makes $V_{0\Q}$ into an $F$-vector space
of dimension $1$.  The  $\C$-quadratic space $V_{0\C}$ is a hyperbolic plane, and its two isotropic lines are distinguished 
by the two embeddings $F\to \C$: on one line $F$ acts through one embedding, and on the other $F$ acts through the conjugate embedding.
These two lines determine two points of  $\mathcal{H}$, and we pick one of them $z_0 \in \mathcal{H}$.  For any $g\in G(\A_f)$,
an exercise in linear algebra shows that  $A_{ ( z_0,g) }$ is isogenous to a product of elliptic 
curves with complex multiplication by $F$.

Let $x_\C  \in \mathscr{S}(\C) $ be the point defined by $(z_0,g)$.  This is a special point in the sense of Deligne,
and so the underlying point  $x\in \mathscr{S}$   has residue field   a finite extension of $\Q$. 
By completing the residue field at a prime above $p$ and passing to its maximal unramified extension,  we obtain a finite extension 
$\Phi /  \Q^\mathrm{unr}_p$  and a point $x_\Phi \in \mathscr{S}(\Phi)$ above $x$.    
As  the Kuga-Satake abelian scheme $A_{x_\Phi}$ has complex multiplication,
the criterion  of N\'eron-Ogg-Shafarevich guarantees that we may replace $\Phi$ by a finite extension 
so that the $\ell$-adic Tate module of $A_{x_\Phi}$ is unramified for all $\ell\not=p$.  
The extension property of Kisin's integral models now gives an extension of $x_\Phi$ to a point 
of $\mathscr{S}(\co_\Phi)$, whose reduction to $\mathscr{S}(k)$ is necessarily supersingular
(as $p$ is inert in the CM field $F$).
\end{proof}

\begin{proposition}\label{switchpoint}
There exists a point $x\in \sS_{U_p}(k)$
such that the local Hodge-Shimura datum $(G_{\Z_p} , b_x ,\mu_x , C(V_{\Z_p}) )$
obtained from $x$ (by the procedure of  \S\ref{sss:global to local datum}) agrees with the local Hodge-Shimura datum of 
Proposition \ref{prop:local GSpin datum}.
\end{proposition}

\begin{proof}
Let $b$ and $\mu$ be as in Proposition \ref{prop:local GSpin datum}.  Using Lemma \ref{nonempty}, we can find a point  
\[
x_0 \in \mathscr{S}_{ss}(k)=\sS_b(k) ,
\]
which  determines  a local Shimura-Hodge datum $(G_{\Z_p} , b_{x_0} ,\mu_{x_0} , C(V_{\Z_p}) )$ as in  \S\ref{sss:global to local datum}.

The cocharacters $\mu_{x_0}$ and  $\mu$ are $G(W)$-conjugate.  Indeed, using (\ref{gspin hodge}), one can see that the conjugacy class of both $\mu_{x_0}$ and  $\mu$ is characterized as the set of all characters $\mathbb{G}_{m W} \to G_W$ such that the composition 
\[
\mathbb{G}_{m W} \to G_W \map{\nu_G} \mathbb{G}_{m W}
\]
is $z\mapsto z^{-1}$, and such that the induced grading on $V_W$ has the form
$
V_W = F_1\oplus F_0 \oplus F_{-1},
$
in which $F_1$ and $F_{-1}$ are isotropic lines orthogonal to $F_0$.

The results of \S \ref{ss:gspin local datum} now show that there is a unique $\sigma$-conjugacy class of basic elements in $G(K)$ making $D_K=\Hom(C(V_K) , K)$ into an isocrystal of slope $1/2$, and hence the basic elements $b_{x_0}$ and $b$ are $\sigma$-conjugate.  Thus the claim follows from Remark \ref{rem:point switch}.
 \end{proof}

Let $V'_\Q$ be the unique positive definite quadratic space over $\Q$ with the same dimension and determinant as $V_\Q$, but with 
Hasse invariant 
\[
\epsilon ( V'_{\Q_\ell}) = \begin{cases}
\ \ \epsilon ( V_{\Q_\ell}) & \mbox{if }\ell \neq p \\
- \epsilon ( V_{\Q_\ell}) & \mbox{if }\ell=p
\end{cases}
\] 
for all finite primes $\ell$.  Let  $I' = \GSpin( V'_\Q)$  be the corresponding spinor similitude group over $\Q$, 
and let $\eta_{I'} : I'  \to \Gm$ be the spinor similitude.

The uniformization Theorem \ref{uniformThm}  now gives the following result.

\begin{theorem}\label{thm:GSpin uniform}
There is an isomorphism of formal $W$-schemes
\[
I'(\Q) \backslash \RZc \times G(\A^p) / U^p  \iso (  \widehat{\mathscr{S}}_W )_{ /  \mathscr{S}_{ss}  }
\]
for suitable isomorphisms $I'(\Q_p) \iso J_b(\Q_p)$,  and $I'(\Q_\ell) \iso G(\Q_\ell)$ for $\ell \not=p$.
\end{theorem}

\begin{proof}  Since $\sS_{ss}=\sS_b$, this will follow from Theorem \ref{uniformThm} after we 
show that the group $I$ in the statement of Theorem \ref{uniformThm} can be identified with the group $I'= \GSpin( V'_\Q)$ above.

The group $I'= \GSpin( V'_\Q)$ is   an inner form of $G$: Since $V_\Q$ and $V'_\Q$ have 
the same dimension and determinant, we can find an isomorphism of quadratic spaces 
\[
\psi: V'_\Q\otimes_\Q\bar\Q\xrightarrow{\sim} V_\Q\otimes_{\Q}\bar\Q
\] which produces a Galois cocycle $\sigma\mapsto \psi \sigma(\psi)^{-1}$ with values in ${\rm SO}(V_\Q)(\bar\Q)$.
Composing this with ${\rm SO}(V_{\Q})\to G^{\rm ad}$ gives a  class $c'\in {\rm H}^1(\Q, G^{\rm ad} )$ which defines $I'= \GSpin( V'_\Q)$. Notice that $V'_{\Q_\ell}$ is isomorphic to $V_{\Q_\ell}$,
for all $\ell\neq p$, and $V'_{\Q_p}$ is isomorphic to $V^\Phi_K$. 

The group $I$ is also
an inner form of $G$. In fact, by the remarks that follow part (iv) of the definition of a Kottwitz triple in  \cite[(4.3)]{KisinLR},  $I$ is uniquely determined {\sl as an inner form} 
 (or more correctly {\sl an inner twist}) of $G$ by the local inner twisting
isomorphisms at finite places, and the fact that $I_{\R}$ is anisotropic 
modulo center. 
(This uses the Hasse principle for adjoint
groups, see \cite[\S 6.5, Theorem 6.22]{PlatonovRapinchuk},
and the fact that there is a unique element of ${\rm H}^1({\mathbb R}, G^{\rm ad})$
which corresponds to the compact modulo center form of $G_{\mathbb R}$, see \cite[p.~423]{KottJAMS}). Therefore, $I$ is given by a well-defined cohomology class $c\in {\rm H}^1(\Q, G^{\rm ad} )$
with prescribed localizations $c_v$ in ${\rm H}^1(\Q_v, G^{\rm ad} )$, for all places $v$ of $\Q$.

By the definition of the classes $c_v$, as provided by the inner twists coming from 
the Kottwitz triple given by $x_0$, $c_\ell$ is trivial for $\ell\neq p$,
while $c_p$ corresponds to the inner twist $\GSpin(V^\Phi_K)$; as
$
V^\Phi_K\iso V'_{\Q_p},
$
 we have $c_p=c'_p$. Hence, $c_v=c'_v$ for all finite places $v$ of $\Q$. Also since $V'_\R$ is positive definite, we have as above $c_{\infty}=c'_{\infty}$. 
The result then follows as above, by the Hasse principle for adjoint
groups.
 \end{proof}

From here  on we identify
\[
I=I'=\GSpin(V'_\Q).
\]

\begin{remark}
As in the proof of Theorem \ref{uniformThm}, the group $I = \GSpin (V'_\Q)$ acts as quasi-endomorphisms 
of the fiber $A_{x_0}$ of the Kuga-Satake abelian scheme at the base point $x_0 \in \mathscr{S}_{U_p}(k)$.
The action $I\subset \End(A_{x_0})_\Q^\times$ can be explained as follows: the fiber $A_{x_0}$, like every fiber of the Kuga-Satake abelian scheme,
comes endowed with a collection of \emph{special quasi-endomorphisms} $V(A_{x_0}) \subset \End(A_{x_0})_\Q$ as in \cite[\S 5]{MadapusiSpin}.  
This is a quadratic space over $\Q$, with quadratic form determined by $v\circ v = Q(v) \cdot \mathrm{id}$.
For any fiber the space of special endomorphisms has dimension $\le \dim(V_\Q)$, and equality holds precisely at supersingular points.
In fact, using \cite[Theorem 6.4]{MadapusiK3}, the supersingularity of $A_{x_0}$ implies that $V(A_{x_0}) \iso V'_\Q$.  After fixing such an isomorphism, we obtain an injection 
$V'_\Q \to \End(A_{x_0})_\Q$, which, by the universal property of Clifford algebras, extends to ring homomorphism
$C(V_\Q') \to \End(A_{x_0})_\Q$.  This homomorphism then restricts to a homomorphism of groups
\[
\GSpin(V_\Q') \to \End(A_{x_0})_\Q^\times.
\]
\end{remark}

\subsubsection{}
Recalling the decomposition
$
\RZc = \bigsqcup_{\ell\in \Z} \RZc^{(\ell)}
$
of $\RZc$ into its connected components,
  we may rewrite the uniformization of Theorem \ref{thm:GSpin uniform} as an isomorphism
\[
 (  \widehat{\mathscr{S}}_W )_{ /  \mathscr{S}_{ss}  }  \iso I(\Q)_0  \backslash \RZc^{(0)}  \times G(\A^p) / U^p ,
\]
where 
\[
I(\Q)_0 = \mathrm{ker} \big( I(\Q)  \map{\eta_I } \Q^\times \map{ \ord_p}    \Z  \big)
\]
is the common stabilizer in $I(\Q)$ of the connected components of $\RZc$.  This may be further rewritten as
\begin{equation}\label{second GSpin uniform}
(  \widehat{\mathscr{S}}_W )_{ /  \mathscr{S}_{ss}  }   \iso \bigsqcup_{ g\in I(\Q)_0 \backslash G(\A_f^p) / U^p } \Gamma_g \backslash \RZc^{(0)},
\end{equation}
where  $\Gamma_g = I(\Q)_0 \cap g U^p g^{-1}.$

\subsection{Structure of the supersingular locus}

As in the proof of Theorem \ref{thm:GSpin uniform}, we may identify \[V_{\Q_p}' \iso V_K^\Phi\] as $\Q_p$-quadratic spaces.
In particular,  we obtain from Definition \ref{def:vertex lattice} the notion of a \emph{vertex lattice} $\Lambda \subset V_{\Q_p}'$,
whose \emph{type} $t_\Lambda = \dim_{\F_p}(\Lambda/ \Lambda^\vee)$ is a positive  even integer less than or equal to the integer
$t_\mathrm{max}$ of (\ref{tmax}).

\subsubsection{}

Fix a vertex lattice $\Lambda  \subset V'_{\Q_p}$. Exactly as in \S \ref{sss:Xlambda},  endow  $\Omega_0=\Lambda/ \Lambda^\vee$ with the rescaled $\F_p$-valued
quadratic form $pQ$.  Set $\Omega= \Omega_0 \otimes_{\F_p} k$, and let $S_\Lambda$ be the reduced $k$-scheme with $k$-points
\[
S_\Lambda(k) = \left\{ \mbox{Lagrangians }\mathscr{L} \subset \Omega : \dim_k( \mathscr{L}+\Phi(\mathscr{L}) ) = \frac{t_\Lambda}{2} + 1  \right\} 
\]
where  $\Phi=\mathrm{id} \otimes \sigma$ is the  absolute Frobenius on $\Omega$.
Recall from Proposition \ref{prop:lagrangian} that $S_\Lambda =S^+_\Lambda \sqcup S^-_\Lambda$ has two connected components.
The components are isomorphic, and each is projective and smooth of dimension $(t_\Lambda/2) -1$.
Up to isomorphism, $S_\Lambda$ depends only on the type $t_\Lambda$.

\subsubsection{}
Taking the reduced scheme underlying both sides of (\ref{second GSpin uniform}) yields an isomorphism
\[
\mathscr{S}_{ss}  \iso \bigsqcup_{ g\in I(\Q)_0 \backslash G(\A_f^p) / U^p } \Gamma_g \backslash \RZc^{(0),\red}.
\]
From this, the description of $ \RZc^{(0),\red}$ of Theorem \ref{thm:final}, and an argument
as in the proof of \cite[Theorem 6.1]{VollaardSS} we deduce the following result.

\begin{theorem}
For all $U^p \subset G(\A_f^p)$ sufficiently small, the following hold:
\begin{enumerate}
\item
Each of the $k$-schemes $\Gamma_g\backslash   \RZc^{(0),\red}$ is connected.
\item
The irreducible components of $  \Gamma_g\backslash   \RZc^{(0),\red}$ are in bijection with the set of orbits 
\[
\Gamma_g \backslash \{ \mbox{vertex lattices of type } t_\mathrm{max} \},
\]
and the irreducible component indexed by a vertex lattice $\Lambda$ is   isomorphic to  $S_\Lambda^\pm$.
In particular, all irreducible components are isomorphic to one another, and are projective and smooth of dimension
\[
\dim(\mathscr{S}_{ss}) = \frac{ t_\mathrm{max}}{2} -1.
\]
\end{enumerate}\qed
\end{theorem}

%%%%%%%%%%%%%%%%%%%%%%%%%%%%%%%%%%%%%%%%%%%%%%%%%%%%

\begin{appendix}

\section{Maximal lattices}

For the reader's convenience we recall two results on maximal lattices in quadratic spaces,  used in several places in the body of the article.

\subsection{Eichler's theorem and the Elementary Divisor Theorem}
Let $F$ be a field, complete with respect to a  discrete valuation.  Denote by $\co$ the valuation ring of $F$.
Suppose $V$ is a finite dimensional $F$-vector space endowed with a nondegenerate quadratic form $q:V \to F$,
and let 
\[
[x,y]=q(x+y) -q(x) -q(y)
\] 
be the bilinear form determined by $q$.

\begin{definition}\label{def:q-maximal}
By a \emph{lattice} in $V$ we mean a free $\co$-submodule $M\subset V$ with $\mathrm{rank}_{\co}(M) = \mathrm{dim}_F(V)$.
 A lattice $M$ is \emph{maximal} (with respect to $q$) if $q(M)\subset \co$, and if $M$ is not properly contained in any other lattice with this property.
The \emph{dual} $M^\vee$ of the lattice $M$ is   
\[
M^\vee=\{x\in V :  [x, m]\in \co, \ \forall m\in M\}.
\]
 The lattice $M$ is called \emph{self-dual} if $M=M^\vee$.
\end{definition}

For a proof of the following, see \cite[Theorem 8.8]{Ger}.

\begin{theorem}[Eichler]\label{thm:eichler}
All  maximal lattices in $V$ are isomorphic as $\co$-quadratic spaces.  If $V$ is anisotropic, then it has a unique maximal lattice 
\[
M = \{ x\in V : q(x) \in \co \}.
\]
\end{theorem}

\begin{theorem}[Elementary divisor theorem]\label{thm:elementary divisors}
Suppose $A$ and $B$ are maximal lattices in $V$.  There is a decomposition
\[
V = F e_1 \oplus F f_1 \oplus \cdots \oplus  F e_r \oplus F f_r \oplus V_0,
\]
in which   $V_0$ is anisotropic and orthogonal to all $e_i$ and all $f_i$,
\[
[e_i,e_j] = 0,\quad [f_i,f_j]=0, \quad [e_i,f_j] =\delta_{i,j},
\]
 and 
\begin{align*}
A &= \co e_1 \oplus \co f_1 \oplus \cdots \oplus \co e_r \oplus \co f_r \oplus M_0 \\
B  &= ( \beta_1)  e_1 \oplus (\beta_1^{-1})  f_1 \oplus \cdots \oplus  (\beta_r) e_r \oplus (\beta_r^{-1})  f_r \oplus M_0 ,
\end{align*}
for some  $\beta_1,\ldots, \beta_r \in F^\times$.  Here $(x) =\co  x$, and $M_0 \subset V_0$ is the unique maximal lattice in $V_0$.
\end{theorem}

\begin{proof}
By applying \cite[Lemma 6.36]{Ger} inductively, there is an orthogonal decomposition
\[
V = H_1 \oplus \cdots \oplus H_r \oplus V_0
\]
in which  each $H_i$ is a hyperbolic plane, $V_0$ is anisotropic, and
\begin{align*}
A  &= (A\cap H_1) \oplus \cdots \oplus (A\cap H_r) \oplus (A\cap V_0) \\
B  &= (B\cap H_1) \oplus \cdots \oplus (B\cap H_r) \oplus (B\cap V_0).
\end{align*}
The  maximality of $A$ implies that each $A\cap H_i$ is a maximal lattice of $H_i$, and that $A\cap V_0$ is a maximal
lattice of  $V_0$.  Of course similar remarks apply to $B$, and in particular   
\[
A\cap V_0= \{ x\in V_0 : q(x) \in \co \} = B\cap V_0
\]
 by Theorem \ref{thm:eichler}.

Choose a basis $e_i,f_i\in H_i$ such that $q(e_i)=q(f_i)=0$ and 
$[e_i,f_i]=1$.  Using the fact that $Fe_i$ and $Ff_i$ are the unique isotropic lines in $H_i$, it follows from \cite[Lemma 6.35]{Ger} 
that 
\begin{align*}
A\cap H_i &=   (\alpha_i) e_i \oplus   ( \alpha_i^{-1}) f_i \\
B\cap H_i &=   (\beta_i) e_i \oplus  (\beta_i^{-1})  f_i
\end{align*}
for some $\alpha_i , \beta_i\in F^\times$.  The desired decomposition of $V$ is now obtained by rescaling $e_i$ and $f_i$ so that $\alpha_i=1$.
\end{proof}
 
 \end{appendix}

 \bibliographystyle{amsalpha}
 %\bibliography{ShimuraBiblioSS}

\providecommand{\bysame}{\leavevmode\hbox to3em{\hrulefill}\thinspace}
\providecommand{\MR}{\relax\ifhmode\unskip\space\fi MR }
% \MRhref is called by the amsart/book/proc definition of \MR.
\providecommand{\MRhref}[2]{%
  \href{http://www.ams.org/mathscinet-getitem?mr=#1}{#2}
}
\providecommand{\href}[2]{#2}

\end{document}